\documentclass[leqno,11pt]{article}
\usepackage[margin=1in]{geometry}

\usepackage[latin,english]{babel}
\usepackage[utf8]{inputenc}
\usepackage[T1]{fontenc}
\usepackage{eufrak, color, mathrsfs,dsfont}
\usepackage{times}
\usepackage{geometry}
\usepackage{amsmath}
\usepackage{amssymb}
\usepackage{amsthm}
\usepackage{tikz}
\usepackage{mathtools} 
\usepackage{mathabx} 
\usepackage{cases}
\usepackage{braket}
\usepackage[toc,page]{appendix}
\usepackage{pdfsync} 
\usepackage{hyperref}
\usepackage{psfrag}
\mathtoolsset{showonlyrefs} 
\usepackage[mathscr]{euscript}
\usepackage{math}

\numberwithin{equation}{section}

\theoremstyle{plain}
\newtheorem{thm}{Theorem}[section]

\newtheorem{lem}[thm]{Lemma}
\newtheorem{prop}[thm]{Proposition}

\theoremstyle{definition}
\newtheorem{defn}[thm]{Definition}

\theoremstyle{remark}
\newtheorem{rem}[thm]{Remark}

\DeclareMathOperator{\diag}{diag}

\definecolor{blush}{rgb}{0.87, 0.36, 0.51}

\newcommand{\OpM}{{\mathcal O}{\mathcal P}{\mathcal M}}

\renewcommand{\bar}{\overline}

\newcommand{\odd}{{\rm odd}}

\newcommand{\pa}{\partial}
\newcommand{\vs}{\varsigma}
\newcommand{\vphi}{\varphi}

\newcommand{\Lip}{\mathrm{Lip}}

\newcommand{\bbeta}{{\boldsymbol{\beta}}}

\title{\bf Long time dynamics close to large amplitude quasi-periodic traveling waves in two dimensional forced rotating fluids}

\author{Roberto Feola, Luca Franzoi,  Riccardo Montalto}

\begin{document}

\date{}

\maketitle

\noindent
{\bf Abstract.} In this paper we consider the $\beta$-plane equation with a smooth 
external force which is a quasi-periodic traveling wave of large amplitude 
$O(\lambda^{\alpha - 1})$, $1 < \alpha < 2$, 
and with large speed of propagation of size $O(\lambda)$.
In a previous paper, the second and the third author proved the existence of quasi-periodic traveling wave solutions of large amplitude of order $O(\lambda^{\theta})$, for some $\theta > 0$. 
The purpose of this paper is to analyze the long time dynamics for smooth initial data close to these traveling wave solutions. In particular, we shall prove that, for initial data sufficiently close to a fixed traveling wave solution (in the $H^s$ topology), the corresponding solution remains close to the traveling wave solution for arbitrary long time (independent of the size of the traveling wave solution). As a consequence, we prove that there are open sets of large initial for which one has almost global existence, namely such that the corresponding solution remains of the same size of the initial datum for arbitrary long time (independent of the size of the initial data). The proof combines several ingredients: an analysis of the linearized PDE at any traveling wave solution via normal form methods, a sharp analysis of the transformed nonlinear problem under the change of coordinates that diagonalizes the linearized equation and energy estimates.

\smallskip



\smallskip %\\[1mm]

\tableofcontents

\section{Introduction}

In this paper we consider  the \emph{$\beta$-plane equation} on the bi-dimensional torus $\T^2$, $\T := (\R /(2 \pi \Z))$ which describes a bi-dimensional approximation of the Euler Coriolis equation for rotating fluids. 
The $\beta$-plane model was proposed by Rossby in 1939 as a flattening approximation for rotating fluids on a sphere to capture with a simpler system the effects of the latitude variation in the Coriolis force (Pedloski, \cite[Ch. 3.17]{Pedlosky}).
This and others approximate models for rotating fluids in two and three dimensions are commonly used in oceanogra\-phy and geophysical fluid dynamics. For more detailes we refer, for instance, to the monographs \cite{McWill}, \cite{Pedlosky}, \cite{chemin2006}, \cite{laure}. In its vorticity formulation, the $\beta$-plane equation takes the form 
\begin{equation}\label{rotating.2D}
	\partial_t v + u \cdot \nabla v - \beta\, \tL v  =  {\bf F}(t, x) \,,\quad  x = (x_1, x_2) \in \T^2 \,.
\end{equation}
Here ${\bf F} : \R \times \T^2 \to \R^2$ represents the external force, $u = (u_1, u_2) : \R \times \T^2 \to \R^2$ represents the velocity field with scalar vorticity $v := \partial_{x_1} u_2 - \partial_{x_2} u_1$, 
$\beta \in \R\setminus\{0\}$ represents the speed of rotation of the frame system around the rotation axis and 
\begin{align}
& 	{\tL :=  \pa_{x_1}(-\Delta)^{-1}}\,, \label{operator.tL} \\
& \tL h(x) 
= \sum_{\xi\in\Z^2\setminus\{0\}} \im\,\tL(\xi) \whh(\xi) e^{\im\,\xi\cdot x}\,, \quad { \tL (\xi) :=  \frac{\xi_{1}}{|\xi|^2} }\quad \forall\,\xi=(\xi_{1},\xi_{2}) \in \Z^2\setminus\{0\}\,.
    \end{align}
is a {\it dispersive} operator which arises 
from the Coriolis term. 
The two dimensional velocity field 
$u:\R\times\T^2 \to \R^2$ can be expressed in terms 
of the scalar vorticity $v : \R\times \T^2 \to \R$ 
by the classical Biot-Savart law 
\begin{equation}\label{biot-savart}
u(t,x) := \fB\big( v (t,x)\big) := \nabla^\perp (-\Delta)^{-1} v(t,x) 
= \sum_{\xi\in\Z^2\setminus\{0\}} {\frac{\im(\xi_{2},-\xi_{1})}{|\xi|^2} }\whv (t,\xi) e^{\im\, \xi\cdot x}\,.
\end{equation}
  There have been several investigations 
  concerning rotating fluids, 
  both in the viscous case and in the inviscid case. 
  An interesting regime that is frequently studied is
when the fluid is rapidly rotating, 
namely when $|\beta| \gg 1$ in the $\beta$-plane 
approximation, both in the viscous case and in the inviscid case. 
In this regime a very strong dispersive effect prevails, 
allowing to control 
the long time behaviour of the solutions. 
For rotating Navier-Stokes we refer to the monograph \cite{chemin2006} 
(see also references therein and \cite{KohLeeTak0}). 
In a series of works, Babin, Mahalov \& Nicolaenko 
\cite{BaMaNi0, BaMaNi97,BaMaNi99,BaMaNi2000} proved regularity and integrability properties of solutions in 3D tori both in the resonant and in the non-resonant case. 
They showed that in the limit of fast rotations, 
the solutions becomes global for rotating Navier-Stokes 
and exists for longer and longer time. 
For inviscid rotating fluids on $\R^3$, the regime of fast rotation was investigated in  \cite{AngFer}, \cite{dalibard2009}, \cite{Dutrifoy},  \cite{KohLeeTak}, \cite{takada}. 

 Independently of the speed of rotation, several results have been proved for the Euler-Coriolis equations on $\R^3$. Global well-posedness was proved by  Guo, Huang, Pausader \& Widmayer \cite{GHPW} and Guo, Pausader \& Widmayer \cite{GuoPauWid} with axisymmetric initial data, and, very recently, by Ren \& Tian \cite{RenTian} with general non-axisymmetric initial data. Independently of the speed of rotation, the long time stability of small amplitude solutions  was studied for the $\beta$-plane equation on $\R^2$ by Elgindi \& Widmayer \cite{ElgWid} and Pusateri \& Widmayer \cite{PusaWid}. 
 In these two works, 
 it is essential to exploit the dispersive effect, 
 which is absent for periodic initial data
 (i.e. the PDE on $\T^2$). In this paper,  
 we address the problem of analyzing 
 the long time dynamics of large solutions
 for the forced $\beta$-plane equation on $\T^2$.
This fits within a broader context of studying the long-time dynamics of two-dimensional fluids for generic initial data, which is essentially a widely open problem.
Indeed, for 2D incompressible equations it is conjectured that
the long time dynamics is described by ``simple states''
(i.e. steady states, traveling waves, periodic and quasi-periodic waves...), see the review \cite{Drivas-Elgindi}.

In this paper, we consider the forced $\beta$-plane equation on $\T^2$ with a forcing term ${\bf F}$ which is a {\it  quasi-periodic traveling wave} of large amplitude $O(\lambda^\alpha)$, with $1 < \alpha < 2$ and $\lambda \gg 1$. In this situation, in \cite{BFMT1} it has been proved that there exist families of  large amplitude quasi-periodic traveling wave solutions of size $O(\lambda^{\alpha - 1})$. The main purpose of the present paper is to prove the nonlinear stability of such traveling wave solutions for arbitrarily long times. As a consequence, we also prove that there are open sets of initial data of large size $O(\lambda^{\alpha - 1})$ such that the corresponding solutions stay of size $O(\lambda^{\alpha - 1})$ for arbitrarily long time (independent of $\lambda$). Note that, as it has been proved in Elgindi \& Widmayer \cite{ElgWid}, for general smooth initial data, the corresponding solution can grow more than double exponentially in time.  

\medskip

We now give the precise statement of our result. 
We assume that the forcing term ${\bf F}(t, x)$ 
in \eqref{rotating.2D}  has the form
\begin{align}
& {\bf F}(t, x) = \lambda^\alpha f(\lambda \omega t, x), \quad f \in {\mathcal C}^\infty(\T^\nu \times \T^2), \quad f \neq 0\,, \quad \omega \in \R^{\nu}\setminus\{0\}\,, \quad \alpha \in (1, 2),  \\
& f(\vf, x) = f(- \vf, - x), \quad \forall \, (\vf, x) \in \T^\nu \times \T^2\,,  \label{ipotesi forcing} \\
&  \int_{\T^2} f(\vf, x)\, \wrt x = 0\,, \quad \forall \, \vf \in \T^\nu\,.
\end{align}
% and $\omega\in\R^{\nu}$  and 
We also assume that ${\bf F}$ is a {\it quasi-periodic traveling wave}, according to the following definition. 
\begin{defn}\label{def1.QPT}
	{\bf (Quasi-periodic traveling waves).}
	Let $\nu \in \N$. We say that a function $v:\R \times \T^2\to \R$ is  a \emph{quasi-periodic traveling wave} with irrational frequency vector $\omega= (\omega_{1},...,\omega_{\nu})\in\R^{\nu}\setminus\{0\}$, that is $\omega\cdot \ell \neq 0$ for any $\ell\in\Z^\nu\setminus\{0\}$, and  ``wave vectors'' $\bar\jmath_{1},...,\bar\jmath_{\nu}\in \Z^2$ if there exists a function $\breve{v}:\T^\nu \to\R$ such that
	\begin{equation}
		v(t,x) = \breve{v} (\omega_{1}t- \bar{\jmath}_1\cdot x,...,\omega_{\nu}t-\bar{\jmath}_{\nu}\cdot x) = \breve{v}(\omega t- \pi( x)) \,, \quad \forall \, (t,x) \in\R \times\T^2\,,
	\end{equation}
	where  $\pi :\R^2 \to \R^\nu$ is the linear map $x \mapsto \pi( x) := (\bar{\jmath}_{1}\cdot x,...,\bar{\jmath}_{\nu}\cdot x)$. We also denote by $\pi^\top$ the transpose of the map $\pi$. 
\end{defn}

%  In this setting, the Coriolis force contribution rewrites as 
%%
%\begin{equation}
%\ff \, \Omega \land \uu = \beta x_2\, \Omega \land \uu = \beta x_2 \,\uu_h^\perp, 
%\end{equation}
%%
%where $\uu_h^\perp=(-u_2, u_1)$ is the orthogonal vector field to the horizontal part $\uu_h=(u_1, u_2)$ of the three dimensional velocity field $\uu$. 
%\medskip

%%%%%
%\subsection{Incompressible Porous Media (IPM) around a stable profile}
%\begin{align}\label{eq:IPM}
% \rho_t - \cR_1^2 + \bu \cdot \nabla \rho=0, \quad \bu=(\cR_1, \cR_2 \rho, - \cR_1^2 \rho). 
%\end{align}
%%%%%%
%\subsection{Dispersive SQG}
%\begin{align}\label{eq:SQG}
%\rho_t - \cR_1 \rho+\bu \cdot \nabla \rho=0, \qquad \bu= \nabla^\perp (-\Delta)^{-1/2} \rho.
%\end{align}
%Once $v(t,x)$ is solved in \eqref{rotating.but.3D},
% solution to the original problem \eqref{euler.coriolis} with $\ff(x_2)=\beta x_2$ is then recovered by solving a transport equation for the third component of the velocity field $\uu=(\uu_h,u_3)$.
In view of the assumptions in \eqref{ipotesi forcing},  we are reduced to studying
the equation
\begin{equation}\label{beta.waves.large.eq}
	\pa_{t} v  + u \cdot \nabla v - \beta \mathtt L v = \lambda^{\alpha} f(\lambda \omega t,x) \,,  \quad u=\fB(v) \,, \quad   t\in\R\,, \quad x = (x_1, x_2) \in 
	\T^2\,,
\end{equation}
where we take $\omega\in \Omega$, being 
$\Omega\subset \R^\nu$  the reference compact set
\begin{equation}\label{anello}
	\Omega := \{\omega\in\R^\nu \,: \, 1 \leq |\omega|  \leq 2 \}\,.
\end{equation}
We fix once and for all  $\nu \geq 2$ and the wave-vectors $\bar\jmath_{1},...,\bar\jmath_{\nu}\in \Z^2$ such that ${\rm dim}\,{\rm span}\{\bar{\jmath}_{1},...,\bar{\jmath}_{\nu}\} = 2$, in order to avoid trivial cases.

\noindent
To give the precise statement of our result, we even need to state precisely the theorem proved in \cite{BFMT1}. To this aim, we introduce 
%\begin{equation}\label{QP.Eq.to.solve}
%	\lambda \,\omega \cdot \pa_{\vf} v -\beta(-\Delta)^{-1/2} \cR_{1} v + u \cdot \nabla v = \lambda^{\alpha} f(\lambda \omega t,x) \,, \quad u = \nabla^\perp (-\Delta)^{-1} v\,.
%\end{equation}
the Sobolev  space on $\T^n$, $n\in\N$,
\begin{equation}\label{Hs.spaces}
	\begin{aligned}
	H^s(\T^n) & := \Big\{ u(\vartheta) = \sum_{k \in \Z^n } \widehat u(k) e^{\im \vartheta \cdot k} : \| u \|_{H^s}^2 := \sum_{k \in \Z^n} \langle k \rangle^{2 s} |\widehat u(k)|^2 < + \infty \Big\}\,, \\
	%H^s_0 \equiv 
	H^s_0(\T^n) & := \Big\{ u \in H^s(\T^n)  : \int_{\T^n} u(\vartheta) \wrt \vartheta = 0 \Big\}\,, \quad \langle k \rangle := \max \{ 1, |k| \}\,.  
	\end{aligned}
\end{equation}
For functions $u(\vf, x)$, $(\vf,x)\in \T^\nu\times\T^2$, $u \in H^s(\T^\nu \times \T^d)$ we often write $\| u \|_s \equiv \| u \|_{H^s}$, whereas for functions depending only on $x \in \T^2$, we write $\| u \|_{H^s}$. 
%We remark that the forcing of this equation is not perturbatively small, but it is actually of large size, with $\lambda\gg 1$. The range of the admissible $\alpha >0$ will be fixed later.

%{\color{red} CAMBIARE $\bar J$, $\bar J^\top$ in $\pi$, $\pi^\top$ IN TUTTO IL PAPER. ORA STO USANDO QUESTE NOTAZIONI. \\
%	CAMBIARE ANCHE NOTAZIONE $\Lambda_{1}$ IN $\tL$ (senza il pedice 1) CON FOURIER MULTIPLIER $\tL(\xi)$}

\begin{thm}[\cite{BFMT1}]\label{teo principale beta plane}
Let $\beta\in\R\setminus\{0\}$, $ \alpha \in (1,2)$ and $\nu \in \N$, $\nu\geq 2$, be fixed. Let $f \in {\mathcal C}^\infty(\T^\nu \times \T^2, \R)$ be a quasi-periodic traveling wave (see  Definition \ref{def1.QPT}) satisfying \eqref{ipotesi forcing}. There exists $\bar s = \bar s(\nu, \alpha) > 0$ such that, for any $s \geq \bar s(\nu, \alpha)$,
there exist $\lambda_0 = \lambda_0(f, s,  \nu, \alpha, \beta) \gg 1$ large enough and constants $C_1, C_2>0$, with
$C_i = C_i(f, s, \nu, \alpha, \beta) $ for  $i=1,2$, 
such that, for every $\lambda \geq \lambda_0$, the following holds. 
There exists a Borel set $\Omega_\lambda \subset \Omega$ 
of asymptotically full Lebesgue measure as $\lambda \to + \infty$,
that is $\lim_{\lambda \to + \infty} |\Omega \setminus \Omega_\lambda| = 0$, 
such that, for every $\omega \in \Omega_\lambda$, there exists a quasi-periodic traveling wave $v_\lambda(\,\cdot\, ; \omega) \in H^s(\T^\nu \times \T^2)$, that solves the equation \eqref{beta.waves.large.eq}. Moreover 
%$$
%C_1 \lambda^{b} \leq \sup_{\omega \in \Omega_\lambda} \| v_\lambda(\cdot; \omega) \|_s \leq C_2 \lambda^b
%$$ for some exponent $b = b(\nu, \alpha) > 0$. 
\begin{equation}\label{stima.grande.solutione}
	 C_1 \lambda^{\alpha-1} \leq \inf_{\omega \in \Omega_\lambda} \| v_\lambda(\,\cdot\,; \omega) \|_s  \leq    \sup_{\omega \in \Omega_\lambda} \| v_\lambda(\,\cdot\,; \omega) \|_s \leq C_2 \lambda^{\alpha-1+\tc}\,,
\end{equation}
for some $\tc \in (0,\tc_{0})$ arbitrarily small, with $\tc_{0}= \tc_{0}(\alpha,\nu)$, and 
\begin{equation}\label{travel pari media nulla teo}
\begin{aligned}
& v_\lambda(\vf, x) = - v_\lambda(- \vf, - x)\,, \quad \forall \, (\vf, x) \in \T^\nu \times \T^2\,, \\
& \int_{\T^2} v_\lambda(\vf, x) \wrt x = 0\,, \quad \forall \, \vf \in \T^\nu\,. 
\end{aligned}
\end{equation}
\end{thm}

The purpose of this paper is to prove the long time stability of the traveling wave solutions constructed in Theorem \ref{teo principale beta plane}. Given a Banach space $X$ with norm $\| \cdot \|_X$, for any $T > 0$, and any integer $k \geq 0$ we denote by ${\mathcal C}^k([0, T], X)$ the space of the ${\mathcal C}^k$ maps
$$
[0, T] \to X, \quad t \mapsto u(t)
$$
equipped with the norm
$$
\| u \|_{{\mathcal C}^k_T X} := \sum_{0 \leq n \leq k} \sup_{t \in [0, T]} \| \partial_t^n u(t) \|_X\,. 
$$The precise statement of the main result of the paper is the following.
\begin{thm}[{\bf Long time stability of large amplitude traveling wave solutions.}]\label{thm:mainsatability}
Let $\beta\in\R\setminus\{0\}$, $ \alpha \in (1,2)$ and $\nu \in \N$, $s > 2$ be fixed. Let $f \in {\mathcal C}^\infty(\T^\nu \times \T^2, \R)$ be an even quasi-periodic traveling wave (see  Definition \ref{def1.QPT}) satisfying \eqref{ipotesi forcing}. Then there exist $\bar \mu = 
\bar \mu(\nu, \alpha) > 0$ and $\lambda_0 = \lambda_0(s, f, \nu, \alpha, \beta) \gg 1$ large enough  
such that, for every $\lambda \geq \lambda_0$, 
there exists a Borel set $\mathcal{O}_{\lambda}\subseteq\Omega_\lambda \subset \Omega$ 
of asymptotically full Lebesgue measure as $\lambda \to + \infty$,
that is $\lim_{\lambda \to + \infty} |\Omega \setminus \mathcal{O}_{\lambda}| = 0$, 
such that, for every $\omega \in \mathcal{O}_{\lambda}$,
the following holds. let $v_\lambda(\lambda \omega t, x)$, $\omega \in \Omega_\lambda$ be a traveling wave solution of \eqref{beta.waves.large.eq} with $v_\lambda \in H^{s + \bar \mu}(\T^\nu \times \T^2)$ and satisfying \eqref{travel pari media nulla teo}. 
Then there are $c_{s}>0$, $\delta_0 = \delta_0(s, \lambda, \alpha, \nu)>0$ (small enough) such that, for any 
$0 < \delta\leq \delta_0$ and for any initial datum $v_0 \in H^s_0(\T^2)$ with
\[
\| v_0 - v_\lambda(0, \,\cdot\,)\|_{H^s}\leq \delta\,,
\]
there exists a unique solution $v\in {\mathcal C}^0([0, T_\delta], H^s_0(\T^2)) \cap {\mathcal C}^1([0, T_\delta], H^{s - 1}_0(\T^2))$
with $v(0)=v_0$ satisfying
\[
\| v(t) - v_\lambda(\lambda \omega t, \cdot) \|_{H^s} \leq 2 \delta, \quad \forall \, t \in [0, T_\delta ]\,,
\qquad 
T_{\delta}:= c_{s}\delta^{-1}
\]
\end{thm}

\begin{rem}
We remark that Theorem \ref{thm:mainsatability}  holds also in the case $\nu = 1$ or when ${\rm dim}\,{\rm span}\{\bar{\jmath}_{1}, \dots, \bar{\jmath}_{\nu}\} = 1$. In the latter situation, the existence of a traveling wave solution follows as in \cite{BFMT1}, and the proof further simplifies since the corresponding traveling wave is one-dimensional.
In the case $\nu = 1$, the existence of a traveling wave of the form
$v_{\lambda}(\lambda \omega t - \bar{\jmath} \cdot x)$, with $\omega \in \mathbb{R} \setminus \{0\}$, is immediate, as the nonlinearity satisfies
$\mathfrak{B}(v_{\lambda}) \cdot \nabla v_{\lambda} \equiv 0$ and the equation for $v_\lambda(\vf, x)$ becomes
$$
\lambda \omega \partial_\vf v_\lambda(\vf, x) - \beta \mathtt L v_\lambda(\vf, x) = \lambda^\alpha f(\vf, x)
$$
whose solution is given by 
\begin{equation}
    v_\lambda(\vf, x) = \sum_{\begin{subarray}{c}
\ell \in \Z \,,\, j \in \Z^2 \setminus \{ 0 \} \\
\pi^\top(\ell) + j = 0
\end{subarray}} \dfrac{\lambda^\alpha \widehat f(\ell , j)}{\im (\lambda \omega \ell - \beta \mathtt L(j))} e^{\im ( \ell \vf + j x)}
\end{equation}
(recall that $\pi^\top$ in Definition \ref{def1.QPT}). Note also that in this time periodic case, for any $\ell \in \Z$ and $j \in \Z^2 \setminus \{ 0 \}$, with $\pi^\top(\ell) + j = 0$, one has that $|\lambda \omega \ell - \beta \mathtt L(j)| \geq \lambda \gamma$ for some constant $\gamma > 0$ ad for $\lambda \equiv \lambda (\beta) \gg 0$ large enough. 
We choose to state the theorem under the assumptions $\nu \geq 2$ and
${\rm dim}\,{\rm span}\{\bar{\jmath}_{1}, \dots, \bar{\jmath}_{\nu}\} = 2$, which correspond to the genuinely nontrivial cases for the existence of traveling waves. Nevertheless, the stability proof remains valid also in the simpler settings described above.
\end{rem}
As a consequence of Theorem \ref{thm:mainsatability}, we get the following.
\begin{thm}\label{corollario nonlinear stability}
Let $s > 2$. Under the assumptions of Theorem \ref{thm:mainsatability},
for any $\lambda\geq \lambda_0$ large enough, there is $\delta_0 = \delta_0(s, \lambda, \alpha, \nu) > 0$ small enough such that, for any $\delta \in (0, \delta_0)$, for any $\omega\in\mathcal{O}_\lambda$, there exist an open set 
of initial conditions $\mathcal{U}_{\lambda,\delta,\omega}\subseteq H^{s}$
and constants $C_1,C_2>0$ 
such that, for any $v_0\in \mathcal{U}_{\lambda,\delta,\omega} $
 the following holds.
There exists
 a  unique solution of \eqref{beta.waves.large.eq} 
 $v\in {\mathcal C}^0([0, T_\delta], H^s_0(\T^2)) \cap {\mathcal C}^1([0, T_\delta], H^{s - 1}_0(\T^2))$
with $v(0)=v_0$ satisfying
\begin{equation}\label{stimaaltobasso}
C_1\lambda^{\alpha-1}\leq \sup_{t \in [0, T_\delta]}\|v(t)\|_{H^s}\leq C_2\lambda^{\alpha-1 + \mathtt c}\,.  
\end{equation}
\end{thm}
Some comments are in order. 

The above results 
allow us to control the Sobolev norm of solutions evolving from
data in a ball of $H^s$ of radius $\delta$ centered at a quasi-periodic traveling wave solution $v_{\lambda}$.
First of all, we remark that the time of stability depends only on the radius $\delta$ and not on the size of the initial 
datum nor  the solution $v$. Such parameter $\delta$ can be taken arbitrarily small.
As a consequence, Theorem \ref{thm:mainsatability} provides open sets of initial data for the problem \eqref{beta.waves.large.eq}, for which the corresponding solutions are defined over an arbitrarily long time scale and they remains bounded over such long time interval.

A second crucial point is that, in order to control the 
$H^s$
norm of the solution, one needs to study its stability in a neighborhood of a quasi-periodic wave $v_{\lambda}$,
which enjoys higher regularity, belonging to 
$H^{s+\bar{\mu}}$ for some $\bar{\mu}>0$.
This is essentially unavoidable due to the loss of regularity 
which typically arises from small divisors problem.
In particular, such difficulty arises 
in studying the spectral analysis of
linearized problem, where we implement 
a reducibility scheme which diagonalizes the linearized operator. This procedure is crucial in order to show the linear evolution is remains stable for all time, which in turn
allows us to conclude that the time of the nonlinear stability $T_{\delta}$ depends only on $\delta$ and it is independent of $\lambda$ (i.e. the size of the traveling wave solution and of the initial datum).

\medskip

 \noindent
 {\sc Related results.}

 \medskip

 \noindent
 In the last year, there has been a lot of attention in the analysis of quasi-periodic waves in Fluid Mechanics. This is a difficult problem due to the presence of small divisors issues and due to the quasi-linearity of the equations that one has to deal with. In spite of this, the nonlinear stability analysis and the long time dynamics for initial data close to such quasi-periodic wave solutions is a widely open issue. The main result of this paper is a first step towards this investigation.  

\medskip

 \noindent
 For fluids that reduces to one-dimensional quasi-linear PDEs, a quite general approach has been successfully developed in the last decade, based on tools from micro-local analysis, spectral theory, perturbation theory and Nash-Moser methods. For two dimensional water waves equations, we mention
%Iooss, Plotnikov & Toland [27] for periodic standing waves, 
 \cite{BM}, \cite{BBHM}, \cite{BFM,BFM21},  \cite{FeoGiu}. For 2D Euler and $\alpha$-SQG equations, bifurcation results of quasi-periodic vortex patches (that reduces to one dimensional PDEs for the contour dynamics) have been proved in  \cite{BertiHassMasm}, \cite{HmRo},  \cite{HaHmMa}, \cite{HHMINVENTIONES}, \cite{Roulley}, \cite{HaRo}, \cite{HaHmRou}, \cite{GomIonePa}.

 \noindent
 When the fluid is described by a quasi-linear PDE in dimension greater than two, the situation is much more complicated due to much stronger resonant phenomena that typically appear in the higher dimensional framework, combined with the strong perturbative effect of the quasi-linear nonlinearities. 
Time quasi-periodic solutions were also constructed for the Euler and Navier Stokes equations on $\T^2, \T^3$ and in the channel $\R \times [- 1, 1]$, see  \cite{BM20}, \cite{FrMo}, \cite{CF}, \cite{EPsTdL}, \cite{FrMaMo} and for 3D Water Waves in finite depth, see \cite{FeMoTe}.  Finally, large amplitude quasi-periodic traveling waves have been constructed in \cite{CiMoTe} for the MHD equation and in \cite{BFMT1} for the $\beta$-plane equation. 

\noindent
The analysis of the long time dynamics close to quasi-periodic waves in fluid mechanics is widely open. In this paper we aim to address this question, by analyzing the long time dynamics of solutions of the $\beta$-plane equation close to the traveling waves constructed in \cite{BFMT1}. Results in a similar spirit have been obtained for semilinear dispersive PDEs. More precisely, this question has been explored for Nonlinear Schr\"odinger (NLS) equation and Korteweg de Vries (KdV) equations on $\T$, by exploiting their integrable structure. In \cite{albi}, the $2d$ Nonlinear Schr\"odinger (NLS) equation has been considered (see also \cite{HMLP} in the direction of proving the long time stability of small amplitude quasi-periodic solutions of NLS on $\T^2$): 
% It is well known that 1d defocusing cubic NLS equation admits family of quasi-periodic solutions called finite gap solutions.  T
the authors proved the long time stability of small amplitude finite gap solutions (namely, quasi-periodic solutions of the 1d defocusing cubic NLS equation) under bi-dimensional periodic perturbations. In \cite{Kappeler-Montalto}, the long time stability of large amplitude finite gap solutions has been addressed for semilinear perturbations of the KdV equations. 
In the context of quasi-linear PDEs (quasi-linear Klein-Gordon equation), we also mention \cite{FeoGiu2} where the stability of \emph{approximate} periodic/quasi-periodic solutions is analyzed.

\noindent
We conclude this part of the introduction by mentioning that, concerning PDEs of Fluid Mechanics, the only previous available result on nonlinear stability is proved in \cite{Diego-Angel}, in which the authors showed the stability of small amplitude periodic traveling waves (stationary in a moving frame) for the Burgers-Hilbert equation, which is a one-dimensional quasi-linear PDE that approximates the dynamics of the vortex patches for the 2D Euler equation. 

% In the next section, we give an overview of the main difficulties and of the main ideas of the proof of our results. 

\medskip

\noindent
{\sc Notation.}  Throughout the paper we shall use the following notations.

\noindent
Given two Banach spaces $X$ and $Y$, we denote by ${\mathcal B}(X, Y)$ the space of bounded linear operators from $X$ to $Y$ equipped with the standard operator norm $\| \cdot \|_{{\mathcal B}(X, Y)}$. If $X = Y$, we use the notation ${\mathcal B}(X) \equiv {\mathcal B}(X, X)$. Given a linear operator ${\mathcal A}$ (densely defined) on a suitable Hilbert space, we denote by ${\mathcal A}^*$ its adjoint. Moreover given two linear operators ${\mathcal A}, {\mathcal B}$, their commutator is given by $[{\mathcal A}, {\mathcal B}] := {\mathcal A} {\mathcal B} - {\mathcal B} {\mathcal A}$. We denote by $\langle \cdot\,,\, \cdot \rangle_{L^2}$ the standard scalar product on $L^2(\T^n)$. 

\noindent
Given some parameters $a_1, \ldots, a_n > 0$, we use the following notation. We write
$$
A \lesssim_{a_1, \ldots, a_n} B 
$$
if there exists a positive constant $C(a_1, \ldots, a_n) > 0$ such that 
$$
A \leq C(a_1, \ldots, a_n) B\,. 
$$
If the constant $C$ depends on the number of frequencies $\nu \in \N$, the exponent appearing in the non-resonance conditions $\tau$ (see for instance \eqref{DC.2gamma}) and the data of the problem, i.e. the forcing term $f$, the strength of the Coriolis force $\beta$ in \eqref{rotating.2D} and the parameter $\alpha$ in \eqref{ipotesi forcing}, we simply write $\lesssim$ instead of $\lesssim_{\nu, \tau, \alpha, \beta, f}$. 
Throughout the paper, we will systematically 
omit tracking the dependence of the constants on these parameters

\subsection{Ideas of the proof}
In this section we shall describe the main difficulties 
and ideas required to prove Theorems \ref{thm:mainsatability}, \ref{corollario nonlinear stability}. 
We want to study the stability of the traveling wave solutions 
$v_\lambda(\lambda \omega t, x)$ 
provided by Theorem \ref{teo principale beta plane}. 
Hence, we look for initial data $v_0 \in H^s_0(\T^2)$, with $s > 2$, close to the traveling wave $v_\lambda$, 
namely 
$\| v_0 - v_\lambda(0, \cdot) \|_{H^s} \leq \delta \ll 1$, 
and we look for solutions which are small perturbations of $v_\lambda$, 
namely $v (t, x) = v_\lambda(\lambda \omega t, x) + w(t, x)$. 
The Cauchy problem for $w(t, x)$ takes the form
\begin{equation}\label{Cauchy problem intro}
\begin{cases}
	 \pa_{t} w =   {\mathcal L}(\lambda \omega t) [w] - {\mathfrak B}(w) \cdot \nabla w  
	\\
     w(0, x) = w_0(x) := v_0(x) - v_\lambda(0, x)\,,
    \end{cases}
\end{equation}
where ${\mathfrak B}$ is defined in \eqref{biot-savart} and ${\mathcal L}(\lambda \omega t)$ 
is the linearized vector field at the traveling wave solution, namely 
\[
{\mathcal L}(\lambda \omega t)[w] := 
\beta {\mathtt L} w - {\mathfrak B}\big(v_\lambda(\lambda\omega t,x)\big) \cdot \nabla w 
- {\mathfrak B}(w) \cdot \nabla v_\lambda(\lambda\omega t,x)\,.
\]
Since $\| v_\lambda \|_{H^{s + \bar \mu}} \leq   O(\lambda^{\alpha-1+\tc})$ (see \eqref{stima.grande.solutione}), with $\lambda \gg 1$ and $s > 2$, $\bar \mu \gg 1$ large enough, 
standard energy methods for local existence (see \cite{Kato1}, \cite{Taylor}), 
gives stability only for short times, namely 
\[
\| w(t) \|_{H^s} \lesssim \delta\,, \quad \forall \, t \in [0, T_\lambda]\,, 
\qquad T_\lambda \simeq \frac{1}{\lambda^{\alpha - 1+\tc}} \ll 1\,.
\]
We want to improve the stability time up to order $O(\delta^{- 1})$ for $0 < \delta \ll 1$ arbitrarily small,
in such a way that the solution $v(t, x)$ remains of the same size of the initial datum $v_0$ 
for times arbitrarily long, i.e. independent of the size of the initial datum $\| v_0 \|_{H^s} \simeq \| v_\lambda(0, \cdot)\|_{H^s} \simeq \lambda^{\alpha - 1 + \mathtt c}$. 
In order to do so, one has to perform a very precise spectral analysis of the 
linearized PDE at the traveling wave solution $v_\lambda$ 
combined with sharp energy estimates 
that remain valid over arbitrarily long time intervals.
The main difficulties are:
\begin{itemize}
\item{ the resonance phenomena due to small divisors and space resonances in higher dimension}
\item{ the control of the long time dynamics for large initial data}
\item{ the control of the quasi-linear term over arbitrarily long time intervals.}
\end{itemize}
The preliminary step is the analysis of the linearized equation at $v_\lambda$, namely
\[
\partial_t w(t) = {\mathcal L}(\lambda \omega t)[w(t)]\,.
\]
This is done in Section \ref{sez:linear} and it follows in the spirit the methods developed in \cite{BFMT1}, 
that we implement in the framework of the present paper. 
The main result of this section is Theorem \ref{reducibility linearized totale} that allows to construct, 
for a large measure set of frequencies $\omega$, 
a bounded invertible map ${\mathcal U}(\lambda \omega t) : H^s_0(\T^2) \to H^s_0(\T^2)$, $s \gg 1$, 
such that, under the change of coordinates 
$w(t) = {\mathcal U}(\lambda \omega t)[\psi(t)]$, 
the linearized PDE transforms into 
\[
\partial_t \psi(t) = {\mathcal D}[\psi(t)]\,,
\]
where ${\mathcal D} := {\rm diag}_{j \in \Z^2 \setminus \{ 0 \}} \mu_\infty(j)$ 
is a diagonal operator with purely imaginary eigenvalues. 
This theorem is proved by normal form methods, 
that is a perturbative scheme which allows to remove the time dependence 
from the linearized vector field ${\mathcal L}(\lambda \omega t)$. 
A key step is to control the space-time resonances between the time frequency 
of the traveling wave solutions and the normal frequencies of small oscillations 
at the traveling wave solution. 
More precisely, one has to provide suitable lower bounds on the small divisors 
$\im \lambda \omega \cdot \ell + \mu_\infty(j) - \mu_\infty(j')$, 
$\ell \in \Z^\nu$, $j, j' \in \Z^2 \setminus\{ 0 \}$, 
which are called second Melnikov conditions. 
Due to the high degeneracy of the anisotropic dispersion relation 
${\mathtt L}(\xi) = \frac{\xi_1}{|\xi|^2}$, this small divisors are particularly dangerous 
and, in principle, they vanish for infinitely many indexes. 
It is then fundamental to use the conservation of momentum (i.e. the traveling wave structure) 
which allows to verify for a large measure set of $\omega$ 
non resonance conditions of the form 
\[
\begin{aligned}
& |\im \lambda \omega \cdot \ell + \mu_\infty(j) - \mu_\infty(j')| \geq 
\dfrac{\lambda \gamma}{|\ell|^\tau |j|^\tau}\,, 
\qquad 
\ell \neq 0\,, 
\qquad \pi^\top(\ell) + j - j' = 0\,,  
% \\
% & 0 < \gamma \ll 1, \quad \lambda \gg 1\,, 
\end{aligned}
\]
where $0 < \gamma \ll 1$ and in the regime $\lambda \gg 1$.
By imposing conditions of this type, one then removes (using a perturbative approach)
the time dependence from the vector field ${\mathcal L}(\lambda \omega t)$.
The conservation of momentum ensures that the transformed vector field is diagonal. 

\medskip

\noindent
The analysis of the nonlinear problem \eqref{Cauchy problem intro} 
is then performed in Section \ref{sezione nonlinear stability}. 
The fundamental point in order to study the nonlinear stability is to write 
the nonlinear Cauchy problem \ref{Cauchy problem intro} 
in the coordinates constructed before. 
By doing so, one gets a Cauchy problem for the new variable 
$\psi(t) := {\mathcal U}(\lambda \omega t)^{- 1}[w(t)]$ of the form 
\begin {align}
& \begin{cases}
\partial_t \psi(t) = {\mathcal D}[\psi(t)] + {\mathcal Q}(\lambda \omega t, \psi(t)) 
\\
\psi(0) = \psi_0 := {\mathcal U}(0)^{- 1}[w_0]\,, 
\end{cases} \label{cauchy problem trasformato info} \\
& {\mathcal Q}(\vf, \psi) := 
- {\mathcal U}(\vf)^{- 1} \big[{\mathfrak B}\big({\mathcal U}(\vf)[\psi] \big) \cdot \nabla {\mathcal U}(\vf)[\psi] \big]\,, 
\quad \vf \in \T^\nu \,.
\end{align}
We then  perform energy estimate on the latter system. 
In order to achieve this purpose, the crucial point is to analyze the structure 
of the transformed nonlinear term ${\mathcal Q}(\lambda \omega t, \psi)$ 
and to verify that it has a ``good structure'' for performing energy estimates. 
Clearly, this issue requires a sharp analysis of the transformation 
${\mathcal U}(\vf)$, $\vf \in \T^\nu$, and its inverse, 
which is done in details in the previous Section \ref{sez:linear}. 
More precisely, we prove the following property (see Theorem \ref{reducibility linearized totale}): 
the map $\cU(\vf)$ reads as ${\mathcal U}(\vf) := {\mathcal B}(\vf) \circ {\mathcal W}(\vf)$, 
where ${\mathcal B}(\vf)$, $\vf \in \T^\nu$, is a map induced by a diffeormorphism of $\T^2$, namely
\begin{align}
& {\mathcal B}(\vf) : h(x) \mapsto h(x + {\bf \beta}(\vf, x))\,, \qquad \vf \in \T^\nu\,, 
\\
& \| {\bf \beta} \|_{H^s} \lesssim_s \lambda^{- \eta}\,,
\end{align}
and ${\mathcal W}(\vf)$, $\vf \in \T^\nu$, is a one-smoothing perturbation of the identity, 
namely it satisfies 
\[
\| {\mathcal W}(\vf) - {\rm Id} \|_{{\mathcal B}(H^{s - 1}, H^{s})} 
\lesssim_s \lambda^{- \eta}\,,
\]
for some $\eta > 0$. The same properties hold for the inverse ${\mathcal U}(\vf)^{- 1}$. 
This structural property allows to prove in Proposition \ref{lemma struttura nonlin sistema trasformato} the crucial property that the transformed nonlinear term 
${\mathcal Q}$ is a nonlinear transport-type operator up to a bounded quadratic remainder, 
namely of the form
\begin{align}
& {\mathcal Q}(\vf, \psi)  = \ta(\vf, \psi) \cdot \nabla \psi + {\mathcal R}_{\mathcal Q}(\vf, \psi)\,, \label{trapsorto nonlin cal Q intro}
\\
& \| \ta(\vf, \psi) \|_{H^s} \lesssim_s \| \psi \|_{H^{s - 1}}\,, 
\qquad 
\| {\mathcal R}_{\mathcal Q}(\vf, \psi) \|_{H^s} \lesssim_s \| \psi \|_{H^s}^2\,. 
\end{align}
By using the latter key fact, one proves an energy estimates on $\psi(t)$ of the form 
\[
\partial_t \| \psi(t) \|_{H^s} \lesssim_s \| \psi(t) \|_{H^s}^2\,,
\]
which allows to show that 
\[
\| \psi(t) \|_{H^s} \lesssim \delta\,, 
\qquad \forall \, t \in [0, T_\delta] 
\qquad \text{with} \quad T_\delta \equiv T_\delta(s) = O(\delta^{- 1})\,.
\]
By using that ${\mathcal U}(\vf)$ is an isomorphism of $H^s$, 
the same estimate holds for $w(t) = v(t) - v_\lambda(\lambda \omega t, \cdot)$, from which we deduce Theorem \ref{thm:mainsatability}. 

\noindent
To prove Theorem \ref{corollario nonlinear stability}, 
one basically needs to establish the estimates \eqref{stimaaltobasso}. 
The upper bound is obvious by the estimate on 
$v(t) - v_\lambda (\lambda \omega t, \,\cdot\,)$ of Theorem \ref{thm:mainsatability}, 
 taking $\delta \equiv \delta (\lambda) \ll 1$ small enough. 
The lower bound is more delicate. 
It is based on the fact that, as it is proved in \cite{BFMT1}, 
the first approximation of the traveling wave solution 
$v_\lambda(\vf, x)$ is given by 
\[
\begin{aligned}
& v_\lambda \simeq g_\lambda 
\quad 
\text{where}
\quad 
(\lambda \omega \cdot \partial_\vf - \beta \mathtt L)g_\lambda(\vf, x) 
= \lambda^\alpha f(\vf, x)\,.
\end{aligned}
\]
The key fact then is that, by using the explicit formula of 
$g_\lambda$ (with $0 < \delta \ll 1$ small enough), 
one proves that
\[
\sup_{t \in [0, T_\delta]} \| g_\lambda (\lambda \omega t, \cdot) \|_{H^s} \gtrsim \lambda^{\alpha - 1}\,,
\]
This is the content of 
Lemma \ref{lower bound g lambda vari}. 
Finally, the same property of the full solution $v$ 
follows by a perturbative argument.

\medskip

\noindent {\bf Acknowledgements.} L. Franzoi and R. Montalto  
are supported by the ERC STARTING GRANT 2021 “Hamiltonian Dynamics, 
Normal Forms and Water Waves” (HamDyWWa), Project Number: 101039762. 
Views and opinions expressed are however those of the authors only and do not necessarily reflect those of the European Union or the European Research Council. Neither the European Union nor the granting authority can be held responsible for them.
R. Feola is supported 
by ``GNAMPA - INdAM'', CUP E53C25002010001,
and 
``GNAMPA - INdAM'', CUP E5324001950001.
L. Franzoi is also supported by ``GNAMPA - INdAM'', CUP E53C25002010001.

\section{Function spaces, norms and linear operators}\label{sez:functional}

\subsection{Function spaces}
\label{subsec:function spaces}
We denote by ${\mathcal C}^s \equiv {\mathcal C}^s(\T^n)$ the space of ${\mathcal C}^s$ functions (real or complex, scalar- or vector-valued) on $\T^n$, equipped with the standard ${\mathcal C}^s$ norm $\| \cdot \|_{{\mathcal C}^s}$. We also use the Sobolev space $H^s \equiv H^s(\T^n)$ of functions
(real or complex, scalar- or vector-valued) equipped by the norm given in \eqref{Hs.spaces}. Throughout the paper we shall use the following standard properties:  For any $k \in \N$ and $s > \frac{n}{2}$, $H^{s + k}(\T^n)$ is compactly embedded in ${\mathcal C}^k(\T^n)$ and 
\begin{equation}\label{embedding sob}
\| u \|_{{\mathcal C}^k} \lesssim_s \| u \|_{H^{s + k}}, \quad \forall \, u \in H^{s + k}(\T^n)\,. 
\end{equation}
Moreover
\begin{equation}\label{algebra Hs}
\begin{aligned}
& \| u v \|_{H^s} \lesssim_s \| u \|_{H^s} \| v \|_{H^s}, \quad u, v \in H^s(\T^n), \quad s > \frac{n}{2}\,, 
\\
& \| u v \|_{H^s} \lesssim_s 
\| u \|_{{\mathcal C}^s} \| v \|_{H^s}\,, 
\quad u \in {\mathcal C}^s(\T^n)\,,\, v \in H^s(\T^n)\,, 
\quad s \in \N\,.
\end{aligned}
\end{equation}
To simplify the notation, we often write $a = a(\vf,x)$
for a (real or complex-valued, scalar or vector-valued)
function  defined for 
$(\vf, x) \in \T^\nu \times \T^2$   
and  denote its Sobolev norm by
%is given by
\begin{equation*} \label{def Sobolev norm generale}
	\| a \|_s^2 := \sum_{(\ell, j) \in \Z^\nu \times \Z^2} 
	\langle \ell, j \rangle^{2s} | \widehat a(\ell,j) |^2 \,,
	\quad \ 
	\langle \ell, j \rangle := \max \{ 1, |\ell|, |j| \} \,,
\end{equation*}

%we write, in short, $H^s$ both for vectors and for matrices.

In this paper, we use Sobolev norms for 
(real or complex-valued, scalar or vector-valued) 
functions $u( \vf, x; \omega)$, 
$(\vf,x) \in \T^\nu \times \T^2$, which are 
Lipschitz continuous with respect to the parameter 
$\omega\in\R^\nu\setminus\{0\}$.
%$\lambda:=(\om,\zeta) \in \R^{d+2}$.
%We use the compact notation $\lambda := (\omega,\zeta)$ to collect 
%the frequency $\om$ and the depth $\zeta$ into one parameter vector.  
We fix the threshold regularity
\begin{equation}\label{definizione s0}
	s_0 :=  \nu+ 6
\end{equation}
(according to \cite{BFMT1}) and we define the weighted Sobolev norms in the following way. %$\| \cdot \|_{s,F}^{\Lip(\gamma)}$.

\begin{defn} 
	{\bf (Weighted Sobolev norms).} 
	\label{def:Lip F uniform} 
	Let   $\gamma \in (0,1]$, $\Lambda \subseteq \R^{\nu}$ be an arbitrary closed set,
$s \geq 0$,
and consider a function $u : \Lambda \to H^s(\T^\nu \times \T^2)$, 
	$\omega \mapsto u(\omega) = u(\vf,x; \omega)$,
	which is Lipschitz continuous with respect to $\omega$. 
	We define its weighted Sobolev norm by
{	$$
	\| u \|_{s}^{\Lip(\gamma)}:= \| u \|_{s,\Lambda}^{\Lip(\gamma)} := \| u\|_{s}^{\sup} + \gamma \,\| u\|_{s-1}^{\rm lip}\,,
	$$
	where
	\begin{equation*}
		\| u\|_{s}^{\sup}: =\| u\|_{s,\Lambda}^{\sup} := \sup_{\omega\in \Lambda} \| u(\omega)\|_{s}\,, \quad \| u\|_{s}^{\rm lip}:= \| u\|_{s,\Lambda}^{\rm lip}:= 
		\sup_{\substack{\omega_1,\omega_2\in \Lambda \\ \omega_1\neq \omega_2}} 
		\frac{\| u(\omega_1)-u(\omega_2)\|_{s}}{| \omega_1-\omega_2|}\,.
	\end{equation*}}
	For $u$ independent of $(\vf,x)$, we simply denote by 
	$| u |^{\Lip(\gamma)}:= | u|^{\sup} + \gamma \, | u|^{\rm lip} $.
\end{defn}

%\begin{rem}
%{	The norm $\| \,\cdot\, \|_{s,\Lambda}^{\Lip(\gamma)}$ is monotone with respect to the set $\Lambda$ in the following sense.  Let $\Lambda\subseteq \Lambda' \subset \R^\nu$ be arbitrary closed sets and let $u : \Lambda \to H^s(\T^\nu \times \T^d)$, 
%	$\omega \mapsto u(\omega) = u(\vf,x; \omega)$ 
%	be Lipschitz continuous with respect to $\omega\in \Lambda$. Then $\| u\|_{s,\Lambda}^{\Lip(\gamma)} \leq \| u' \|_{s,\Lambda'}^{\Lip(\gamma)} $, where $u' : \Lambda' \to H^s(\T^\nu \times \T^d)$ is the Lipschitz continuous extension of $u$ from $\Lambda$ to $\Lambda'$ provided by Kirszbraun Theorem. \\
%	For sake of simplicity in the notations, in the following sections we will declare each time on which closed set $\Lambda\subset\R^\nu$ of parameters the functions and the oeprators are defined, but we will omit to specity both the set $\Lambda$ in the norms and the Kirszbraun extension to bigger sets in the estimates, since it will be clear by the context.}
%\end{rem}

%\begin{lem}{\bf (Product).}
%	\label{lemma:LS norms}
	%Consider the space $\Lip(k_0,F,s,\g)$ defined in Definition \ref{def:Lip F uniform}.
%	For all $ s \geq s_0$, % > (\nu + 3)/2 $, we have
%	\begin{align}
%		\| uv \|_{s}^{\Lip(\gamma)}
%		& \lesssim_s C(s)  
%\| u \|_{s}^{\Lip(\gamma)} \| v %\|_{s_0}^{\Lip(\gamma)}
%+ C(s_0)  \| u \|_{s_0}^{\Lip(\gamma)} 
%\| v \|_{s}^{\Lip(\gamma)}\,.
%		\label{p1-pr}
%	\end{align}
%\end{lem}

\smallskip

\noindent
We restate the definition of quasi-periodic traveling as given in Definition \ref{def1.QPT}, 
for functions of $(\vf,x)$ instead of $(t,x)$.

\begin{defn}\label{def2.QPT}{\bf (Quasi-periodic traveling waves).} 
Let $\bar\jmath_{1},...,\bar\jmath_{\nu}\in \Z^2$ be a given choice of $\nu$ vectors in $\Z^2$.
A function $u:\T^\nu \times \T^2\to \R$ is a \emph{quasi-periodic traveling wave} 
if there exists a function $\breve{u}:\T^\nu \to\R$   such that
\begin{equation}
u(\vf,x) =  \breve{u}(\vf- \pi( x)) \,, \quad \forall (\vf,x) \in\T^\nu \times\T^2\,,
\end{equation}
	where  $\pi :\R^2 \to \R^\nu$ is the linear map 
    $x \mapsto \pi( x) := (\bar{\jmath}_{1}\cdot x,...,\bar{\jmath}_{\nu}\cdot x)$. 
\end{defn}
Comparing with Definition \ref{def1.QPT}, %(with $d=2$), 
it is convenient to call \emph{quasi-periodic} traveling wave both the function $u(\vf,x) =  \breve{u}(\vf- \pi( x))$ and the function of time $u(\omega t,x) =  \breve{u}(\omega t- \pi( x))$.
We define the translation operator
\begin{equation}\label{def:vec.tau}
	\tau_\vs : h(x) \mapsto h(x + \vs)\,,\;\;\vs\in \R^{2}\,.
\end{equation}
Then, quasi-periodic traveling waves are also characterized by the relation
\begin{equation}\label{condtraembedd}
	v(\vphi-\pi(\vs),x)=
	( \tau_\vs \circ v)(\vphi,x)=v(\vphi,x+\vs)\,,  
	\quad \forall \,\vphi \in \T^\nu\,,\; \vs\in \R^{2}\,, \; x\in\T^{2}\,.
\end{equation}
Expanding in Fourier the equivalence in \eqref{condtraembedd}, we obtain that a quasi-periodic traveling wave has the form
\begin{equation}\label{QPT.forma}
	u(\vf,x) = \sum_{\begin{subarray}{c}
	(\ell, j) \in \Z^\nu \times \Z^2 \\
 \pi^\top(\ell) + j = 0
 \end{subarray}} \whu(\ell,j) e^{\im(\ell\cdot \vf + j\cdot x)}\,.
\end{equation}

\subsection{Matrix representation of linear operators}
Throughout the paper we deal with linear operators which leaves invariant the space of zero average functions. We denote by $L^2_0(\T^2) \equiv H^0_0(\T^2)$ the space of $L^2$ functions with zero average and we define the projections $\Pi_0$, $\Pi_0^\bot$ as 
\begin{equation}\label{definizione proiettore media spazio tempo}
	\Pi_0 h := \frac{1}{(2 \pi)^{ 2}} \int_{\T^{ 2}} h( x)\, \wrt  x \,,
	\qquad 
	\Pi_0^\bot := {\rm Id} - \Pi_0\,.
\end{equation} 
Let
${\mathcal R}  : L^2_0(\T^2) \to L^2_0(\T^2)$ be a linear operator. It can be represented as
\begin{equation}\label{matriciale 1}
	{\mathcal R} u (x) := \sum_{j, j' \in \Z^2 \setminus \{ 0 \}} {\mathcal R}_j^{j'}\widehat u(j') e^{\im j \cdot x} \,, 
	\quad  \text{for} \quad  u (x) = \sum_{j \in \Z^2 \setminus \{ 0 \}} \widehat u(j) e^{\im j \cdot x} \,,
\end{equation}
where, for $j, j' \in \Z^2 \setminus \{ 0 \}$, the matrix element ${\mathcal R}_j^{j'}$ is defined by 
\begin{equation}\label{rappresentazione blocchi 3 per 3}
	{\mathcal R}_j^{j'} :=  \frac{1}{(2\pi)^2} \int_{\T^2} 
	{\mathcal R}[e^{\im j' \cdot x}] e^{- \im j \cdot x} \wrt x\,. 
\end{equation}

We also consider smooth $\vphi$-dependent families of linear operators 
$\T^\nu \to {\mathcal B} (L^2_0(\T^2))$, 
$\vphi \mapsto {\mathcal R}(\vphi)$, 
which we write in Fourier series with respect to $\vphi$ as 
\begin{equation}\label{matrix representation 1}
	{\mathcal R}(\vphi) = \sum_{\ell \in \Z^\nu} \widehat{\mathcal R}(\ell) e^{\im \ell \cdot \vphi}, \quad \widehat{\mathcal R}(\ell) := \frac{1}{(2 \pi)^\nu} \int_{\T^\nu} {\mathcal R}(\vphi) e^{- \im \ell \cdot \vphi}\, \wrt \vphi \,, \quad \ell \in \Z^\nu\,. 
\end{equation}
According to \eqref{rappresentazione blocchi 3 per 3}, for any $\ell \in \Z^\nu$, the linear operator 
$\widehat{\mathcal R}(\ell) \in {\mathcal B} (L^2_0(\T^2))$ is identified 
with the matrix 
$(\widehat{\mathcal R}(\ell)_j^{j'})_{j, j' \in \Z^2 \setminus \{ 0 \}}$.

%\begin{defn}
%	{\bf (Diagonal operators).}
%	\label{def block-diagonal op}
%	Let ${\mathcal R}$ be a linear operator as in 
%	\eqref{matriciale 1}-\eqref{amatriciana}. We define ${\mathcal D}_{\mathcal R}$ as the operator defined by 
%	\begin{equation*}
%		{\mathcal D}_{\mathcal R} := {\rm diag}_{j \in \Z^2} \widehat{\mathcal R}(0)_j^j\,, \quad (\cD_{\cR})(\ell)_j^{j'} := \begin{cases}
%			\widehat{\mathcal R}(0)_j^{j} & j=j'\,, \ \ell=0\,, \\
%			0 & \text{otherwise}\,.
%		\end{cases}\,.
%	\end{equation*}
%	In particular, we say that $\cR$
%	is a \emph{diagonal operator} if $\cR \equiv \cD_{\cR}$.
%\end{defn}
	Let ${\mathcal R}$ be a linear operator as in 
	\eqref{matriciale 1}. We define ${\mathcal D}_{\mathcal R}$ as the operator
	\begin{equation}\label{diagonal.op.matrix}
		{\mathcal D}_{\mathcal R} := {\rm diag}_{j \in \Z^2} \widehat{\mathcal R}(0)_j^j\,, \quad (\cD_{\cR})(\ell)_j^{j'} := \begin{cases}
			\widehat{\mathcal R}(\ell)_j^{j'} & \quad \text{if} \quad j=j'\,, \ \ell=0\,, \\
			0 & \text{otherwise}\,.
		\end{cases}\,.
	\end{equation}
	In particular, we say that $\cR$
	is a \emph{diagonal operator} if $\cR \equiv \cD_{\cR}$.

For the purpose of the Normal Form methods for the linearized operator in 
Section \ref{sez:linear}, 
it is convenient to introduce the following norms, which take into account both the order and the 
off-diagonal decay of the matrix elements representing 
any linear operator ${\mathcal R}(\vf)$ on $L^2_0(\T^{2})$.
\begin{defn} \label{block norm}
	{\bf (Matrix decay norm and the class $\OpM^m_s$).}
	Let $m \in \R$, $s \geq 0$. We say that ${\mathcal R}$ belongs to the class $\OpM^m_s$ if 
	\begin{equation} \label{def decay norm}
		| \cR |_{m, s} := \sup_{j' \in \Z^2 \setminus \{ 0 \}} 
		\Big( \sum_{\ell \in \Z^\nu\,, \, j \in \Z^2 \setminus \{ 0 \} } 
		\langle \ell, j-j' \rangle^{2s} | \widehat \cR(\ell)_j^{j'}|^2  \Big)^{\frac12}
		 \langle j' \rangle^{- m} < \infty \,.
	\end{equation}
	If the operator $\cR = \cR(\omega)$ is Lipschitz with 
	respect to the parameter $\omega \in \Lambda \subseteq  \R^{\nu}$, 
	we define  
		\begin{align}
			& | {\mathcal R} |_{m, s}^{{\rm Lip}(\gamma)} := |{\mathcal R}|_{m, s}^{\rm sup} + \gamma |{\mathcal R}|^{\rm lip}_{m, s - 1}\,, \\
			& |{\mathcal R}|_{m, s}^{\rm sup} := 
			\sup_{\omega \in \Lambda} |{\mathcal R}(\omega)|_{m, s} \,, 
			\quad 
			|{\mathcal R}|_{m, s - 1}^{\rm lip} := \sup_{\begin{subarray}{c}
					\omega_1, \omega_2 \in \Lambda \label{def decay norm parametri}
					\\
					\omega_1 \neq \omega_2
			\end{subarray}} 
\dfrac{|{\mathcal R}(\omega_1) - {\mathcal R}(\omega_2)|_{m, s - 1}}{|\omega_1 - \omega_2|} \,.
		\end{align}
\end{defn}
\noindent
It readily follows that 
\begin{equation}\label{prop elementari}
	\begin{aligned}
		& m  \leq m' \Longrightarrow \OpM^m_s \subseteq \OpM^{m'}_s \quad \text{and} \quad |\cdot |_{m', s}^{{\rm Lip}(\gamma)} \leq |\cdot |_{m, s}^{{\rm Lip}(\gamma)}, \\
		& s \leq s' \Longrightarrow\OpM^m_{s'} \subseteq \OpM^m_s \quad \text{and} \quad | \cdot |_{m, s}^{{\rm Lip}(\gamma)} \leq |\cdot|_{m, s'}^{{\rm Lip}(\gamma)}\,. 
	\end{aligned}
\end{equation}
We now state some
standard properties of the decay norms 
that are needed for the reducibility scheme 
of Section \ref{sez:linear}.
%\ref{ridusezione}. 

\begin{lem}\label{proprieta standard norma decay}
	
	$(i)$ Let $s \geq s_0$, $m, m' \in \R$, and let ${\mathcal R} \in\OpM^m_s$, ${\mathcal Q} \in \OpM^{m'}_{s + |m|}$. %${\mathcal R}(\lambda)$, ${\mathcal Q}(\lambda)$, $\lambda \in F$ such that 
	Then ${\mathcal R} {\mathcal Q} \in \OpM^{m + m'}_s$ and 
	$$
	|{\mathcal R}{\mathcal Q}|_{m + m', s}^{{\rm Lip}(\gamma)} \lesssim_{s, m} |{\mathcal R}|_{m, s}^{{\rm Lip}(\gamma)} |{\mathcal Q}|_{m', s_0 + |m|}^{{\rm Lip}(\gamma)} + |{\mathcal R}|_{m, s_0}^{{\rm Lip}(\gamma)} |{\mathcal Q}|_{m', s + |m|}^{{\rm Lip}(\gamma)}\ \,;
	$$
$(ii)$ Let $s \geq s_0$, $m\geq 0$ and ${\mathcal R} \in \OpM^{-m}_s$. 
Then, for any integer $n \geq 1$, ${\mathcal R}^n \in \OpM^{-m}_s$ and there exist constants $C(s_0,m), C(s, m) > 0$, independent of $n$, such that 
%	\begin{equation*}
	%		\begin{aligned}
		%			& |{\mathcal R}^n|_{0, s_0}^{{\rm Lip}(\gamma)} \leq C(s_0)^{n - 1} \big(|{\mathcal R}|_{0, s_0}^{{\rm Lip}(\gamma)}\big)^{n} \,, \\
		%			& |{\mathcal R}^n|_{0, s}^{{\rm Lip}(\gamma)} \leq n\,C(s)^{n - 1} \big(C(s_0)|{\mathcal R}|_{0, s_0}^{{\rm Lip}(\gamma)}\big)^{n - 1} |{\mathcal R}|_{0, s}^{{\rm Lip}(\gamma)}\,;
		%		\end{aligned}
	%	\end{equation*}
	\begin{align}
					& |{\mathcal R}^n|_{-m, s_0}^{{\rm Lip}(\gamma)} \leq C(s_0, m)^{n - 1} \big(|{\mathcal R}|_{-m, s_0}^{{\rm Lip}(\gamma)}\big)^{n} \,, \label{stima.potenza} \\
		& |{\mathcal R}^n|_{-m, s}^{{\rm Lip}(\gamma)} \lesssim \, \big(C(s, m)|{\mathcal R}|_{-m, s_0}^{{\rm Lip}(\gamma)}\big)^{n - 1} |{\mathcal R}|_{-m, s}^{{\rm Lip}(\gamma)}\,;
	\end{align}
	$(iii)$ Let $s \geq s_0$, $m \geq 0$ and ${\mathcal R} \in \OpM^{- m}_s$.
% Then there exists $\delta(s,m) \in (0, 1)$ small enough such that, 
%if $|{\mathcal R}|_{-m, s_0}^{{\rm Lip}(\gamma)} \leq \delta(s,m)$, 
%then the map $\Phi = {\rm Id} + {\mathcal R}$ is invertible 
%and the inverse satisfies the estimate 
%\[
%|\Phi^{- 1} - {\rm Id}|_{m, s}^{{\rm Lip}(\gamma)} \lesssim_{s, m} |{\mathcal R}|_{-m, s}^{{\rm Lip}(\gamma)}. 
%\]
%OPPURE ESPONENZIALI? 
%
%Let ${\mathcal R} \in \OpM^{- m}_s$, $s \geq s_0$, $m \geq 0$, $|{\mathcal R}|_{- m , s_0}^{\Lip(\gamma)} \leq 1$. 
Then there exists $\delta(s,m) \in (0, 1)$ small enough such that, 
if $|{\mathcal R}|_{-m, s_0}^{{\rm Lip}(\gamma)} \leq \delta(s,m)$,
then  the map $\Phi := {\rm exp}({\mathcal R}) \in \OpM^{0}_s$  is invertible and satisfies the estimates 
$$
|\Phi^{\pm 1} - {\rm Id}|_{-m, s}^{\Lip(\gamma)} \lesssim_s  |{\mathcal R}|_{-m, s}^{\Lip(\gamma)}\,; 
$$

	\noindent
	$(iv)$ Let $s \geq s_0$, $m \in \R$ and ${\mathcal R} \in \OpM^m_s$. Let ${\mathcal D}_{\mathcal R}$ be the diagonal operator as in \eqref{diagonal.op.matrix}
%	Definition \ref{def block-diagonal op}.
	Then ${\mathcal D}_{\mathcal R} \in \OpM^m_s$ and $|{\mathcal D}_{\mathcal R}|_{m, s}^{{\rm Lip}(\gamma)} \lesssim |{\mathcal R}|_{m, s_0}^{{\rm Lip}(\gamma)}$ for any $s \geq s_0$.  
	As a consequence, 
	\[
	| \widehat{\mathcal R}(0)_j^j |^{{\rm Lip}(\gamma)} \lesssim \langle j \rangle^m|{\mathcal R}|_{s_0}^{{\rm Lip}(\gamma)}\,.
	\] 
\end{lem}
\begin{proof}
We refer to 
Lemma $2.6$ in \cite{BFMT1}.
\end{proof}

\noindent
For $N > 0$, we define the operators $\Pi_N {\mathcal R}$ and $\Pi_N^\perp \cR$ 
by means of their matrix representation as follows: 
\begin{equation}\label{def proiettore operatori matrici}
	(\widehat{\Pi_N {\mathcal R}})(\ell)_{j}^{j'} := \begin{cases}
		\widehat{\mathcal R}(\ell)_j^{j'} & \text{if } |\ell|, |j - j'| \leq N\,, 
		\\
		0 & \text{otherwise}\,, 
	\end{cases} \qquad   \ \Pi_N^\bot {\mathcal R} := {\mathcal R} - \Pi_N {\mathcal R}\,,
\end{equation}
and we recall the following result.

\begin{lem}\label{lemma proiettori decadimento}
{\rm [Lemma 2.7 in \cite{BFMT1}.]}
	For all $s, \alpha \geq 0$, $m \in \R$, one has 
	$|\Pi_N {\mathcal R}|_{m, s + \alpha}^{{\rm Lip}(\gamma)} \leq N^\alpha |{\mathcal R}|_{m, s}^{{\rm Lip}(\gamma)}$ and $|\Pi_N^\bot {\mathcal R}|_{m, s}^{{\rm Lip}(\gamma)} \leq N^{- \alpha} |{\mathcal R}|_{m, s + \alpha}^{{\rm Lip}(\gamma)}$. 
\end{lem}

	\subsubsection{Conjugation rule of linear time dependent vector fields and Real, reversible operators}\label{Reversible operators}
	
	%We recall the notation introduced in  \eqref{X even Y odd}, that is, 
     Let $s, m \geq 0$ and ${\mathcal G}(\vphi) \in {\mathcal B}(H^{s + m}_0(\T^2), H^s_0(\T^2))$, 
     $\vphi \in \T^\nu$. We consider the time dependent linear equation
    \begin{equation}\label{sistema dinamico astratto}
    \partial_t u(t) = {\mathcal G}(\lambda \omega t)[u(t)]\,. 
    \end{equation}
    Let $\Phi(\vphi) \in {\mathcal B}(H^\rho_0(\T^2), H^\rho_0(\T^2))$, $\rho \in [s, s+ m]$, $\vphi \in \T^\nu$, be a $\vphi$-dependent family of invertible linear operators. Then, under the change of coordinates 
    $$
    u(t) = \Phi(\lambda \omega t)[v(t)]\,,
    $$
    the equation \eqref{sistema dinamico astratto} transforms into 
    \begin{align}  
& \partial_t v(t) = {\mathcal G}_+(\lambda \omega t)[v(t)]\,, \label{campo vettoriale astratto trasformato} \\
& {\mathcal G}_+(\vphi) := \Phi_* {\mathcal G}(\vphi) :=  \Phi(\vphi)^{- 1} {\mathcal G}(\vphi) \Phi(\vphi) - \Phi(\vphi)^{- 1} \lambda \omega \cdot \partial_\vphi \Phi(\vphi) \,.
\end{align}
    We now provide the definition of reversible, reversibility preserving, and real operators. We first introduce the map
    \begin{equation}\label{involuzione}
{\mathcal S} : L^2(\T^2) \to L^2(\T^2)\,, 
\qquad u(x) \mapsto u(- x)\,. 
    \end{equation}
    Clearly, this is an involution, that is
    ${\mathcal S}^2 = {\rm Id}$. 
	
	\begin{defn}\label{reserv.operators.def} Let $s\geq 0$, $m\geq 0$.
    \\[1mm]
		$(i)$ We say that a $\vphi$-dependent family of linear operators ${\mathcal G}(\vphi) : H^{s + m}_0(\T^2) \to H^s_0(\T^2)$, $\vphi \in \T^\nu$, is \emph{reversible} 
		if ${\mathcal G}(\vphi) \circ {\mathcal S} = - {\mathcal S} \circ {\mathcal G}(- \vphi)$ for any $\vf \in \T^\nu$.
\\[1mm]
\noindent
        $(ii)$
		We say that a $\vphi$-dependent family of linear operators $\Phi(\vphi) : H^s_0(\T^2) \to H^s_0(\T^2)$ is \emph{reversibility preserving} 
		if $\Phi(\vphi) \circ {\mathcal S} = {\mathcal S} \circ \Phi(- \vphi)$ for any $\vf \in \T^\nu$.  
	\\[1mm]	
		\noindent
		$(iii)$ We say that a linear operator $\Phi$ is \emph{real} 
        if $\Phi(u)$ is real valued for any real valued function $u$. 
	\end{defn}
	It is convenient to reformulate the real and reversibility properties of linear operators in terms of their matrix representations.% provided in Section \ref{sezione matrici norme}. 
	\begin{lem}\label{lemma real rev matrici}
		A linear operator ${\mathcal R}$ is :
		% {lemma}
		
\begin{itemize}
		\item[(i)] real if and only if 
		$\widehat{\mathcal R}(\ell)_{j}^{j'} = \overline{\widehat{\mathcal R}(-\ell)_{- j}^{- j'}}$ 
		for all $\ell \in \Z^\nu$, $j, j' \in \Z^2 \setminus \{ 0 \}$;
		
		\item[(ii)] reversible if and only if 
		$\widehat{\mathcal R}(\ell)_j^{j'} = - \widehat{\mathcal R}(-\ell)_{- j}^{- j'}$ 
		for all $\ell \in \Z^\nu$, $j, j' \in \Z^2 \setminus \{ 0 \}$;
		
		\item[(iii)] reversibility preserving if and only if 
		$\widehat{\mathcal R}(\ell)_j^{j'} =  \widehat{\mathcal R}(-\ell)_{- j}^{- j'}$ 
		for all $\ell \in \Z^\nu$, $j, j' \in \Z^2 \setminus \{ 0 \}$.
        \end{itemize}
	\end{lem}
    Let ${\mathcal G}$ and $\Phi$ as in \eqref{sistema dinamico astratto}, \eqref{campo vettoriale astratto trasformato}. By a direct calculation one can show that if ${\mathcal G}$ is reversible and $\Phi$ is 
    reversibility preserving, then $\mathcal G_+$ is reversible. Moreover, if $\Phi, {\mathcal G}$ are real, then ${\mathcal G}_+$ is real. 
	\subsection{Momentum preserving operators}\label{subsec:momentum}
	
	The following definition is crucial in the construction of traveling waves. 
	
	\begin{defn}	\label{def:mom.pres}
		{\bf (Momentum preserving).}	
		A  $ \vf $-dependent family of linear operators 
		$\cR(\vf) $, $ \vf \in \T^\nu $,  is  {\em momentum preserving} if
		\begin{equation}\label{eq:mp_A_tw}
			\cR(\vf - \pi(\vs) )  \circ \tau_\vs = \tau_\vs \circ \cR(\vf ) \,  , \quad 
			\forall \,\vf \in \T^\nu \, , \ \vs \in \R^2 \,  ,
		\end{equation}
		where the translation operator $\tau_\vs$ is defined in \eqref{def:vec.tau}.
		% 	A linear matrix operator $\bA(\vf ) $ of the form \eqref{real_matrix} or \eqref{C_transformed} is  
		% 	{\em momentum preserving} if each of its components is momentum preserving.
	\end{defn}
	\noindent
	
Moreover,
momentum preserving operators are closed under several operations, as shown in the following lemma.
	\begin{lem}\label{lem:mom_prop}
		Let $\cR(\vf), \cQ(\vf)$ be momentum preserving operators. Then:
		\begin{itemize}
			\item[(i)] {\rm (Composition)}: $\cR (\vf) \circ \cQ (\vf) $ is a momentum preserving operator;
			% 		\item[(ii)] {\rm (Adjoint)}: the adjoint $ (A(\vf))^*$ is momentum preserving.
			\item[(ii)] {\rm (Inversion)}: If $\cR(\vf)$ is invertible then $\cR(\vf)^{-1}$ is momentum preserving;
			\item[(iii)] {\rm (Flow)}: Assume that ${\mathcal R}(\vf) \in {\mathcal B}(H^s(\T^2))$, then ${\rm exp}({\mathcal R}(\vf))$ is momentum preserving. 
		\end{itemize}
	\end{lem}
	\begin{proof}
		Item $(i)$ follows directly by \eqref{eq:mp_A_tw}. Item $(ii)$ follows by taking the inverse,
		of \eqref{eq:mp_A_tw} and using that
		$\tau_{-\vs} = \tau_\vs^{-1} $.  
		Finally,  item $(iii)$ holds since ${\rm exp}({\mathcal R}) = \sum_{n \geq 0} \frac{{\mathcal R}^n}{n!}$ and by item $(i)$.  
	\end{proof}
	
%	
%	We shall say that a linear operator of the form $ \lambda\, \omega\cdot \pa_\vf + \cR(\vf)$ is momentum preserving if $\cR(\vf)$ is momentum preserving.  
%	In particular, conjugating a momentum preserving operator 
%	$\lambda\, \omega\cdot\pa_\vf+\cR(\vf) $ 
%	by a family of invertible linear momentum preserving maps $\Phi(\vf)$, 
%	we obtain  the transformed operator 
%	$ \lambda\,\omega\cdot\pa_\vf + \cR_{+}(\vf) $ in
%	\eqref{trasf-op} which  is momentum preserving. 

%	\begin{lem}\label{lem:MP}
%		Let $ X $ be a vector field translation invariant, according to \eqref{eq:mom_pres}. 
%		Let $ u   $ be a quasi-periodic traveling wave. 
%		Then the linearized operator  $ \di_u X( u(\vf, \cdot) ) $  is momentum preserving.
%	\end{lem}
%	
%	\begin{proof}
%		Differentiating  \eqref{eq:mom_pres}  we get  
%		$ (\di_u X)(\tau_\vs u) \circ \tau_\vs = \tau_\vs (\di_u X)(u) $, 
%		$ \vs \in \R $. 
%		Then, apply \eqref{chartj1}. 
%	\end{proof}
	
	\noindent
	We now provide a characterization of 
	the momentum preserving property in Fourier space.

	\begin{lem}
		\label{lem:mom_pres}
		A $ \vf $-dependent family of operators $\T^{\nu}\ni\vf\mapsto \cR(\vf) $ 
		%, with $ \vf \in \T^\nu $, 
		is momentum preserving  if and only if 
		%the matrix elements of $\cR(\vf)$, defined by \eqref{amatriciana},  
		%fulfil 
		\begin{equation}\label{momentum}
			\wh\cR(\ell)_{j}^{j'} \neq 0 \quad \Rightarrow  \quad \pi^\top( \ell) + j-j' = 0 \, , \quad
			\forall\, \ell \in\Z^\nu , \, \,  j, j' \in \Z^2 \setminus \{ 0 \} \, . 
		\end{equation}
		% 	As a consequence we have that, if $\wh\cR_{j}^{j'}(\ell) = 0$ for any $\ell \neq 0$, then $j=j'$, that is, $\cR$ is a time-independent diagonal operator.
		As a consequence we have that, for any momentum preserving operator $\cR(\vf)$, the operator $\widehat{\mathcal R}(0)$ satisfies 
		\begin{equation}\label{diagonal mom pres}
			\wh\cR(0)_{j}^{j'} \neq 0 \quad  \Rightarrow \quad j=j' \,,
		\end{equation}
		that is, $\widehat\cR(0)$ is a time-independent diagonal operator (recall \eqref{diagonal.op.matrix}).
%		, according to Definition \ref{def block-diagonal op}.
	\end{lem}
\begin{proof}
It follows by a straightforward computation recalling 
Definition \ref{def:mom.pres} and 
\eqref{matriciale 1}-\eqref{matrix representation 1}.
\end{proof}
An important consequence of the above result is that 
momentum preserving operators 
that are independent of $\vf\in\T^\nu$ are actually diagonal.

    \section{Reducibility of the linearized PDE at the traveling wave solution}\label{sez:linear}
In this section we study the linearization of the PDE \eqref{beta.waves.large.eq} 
at the traveling wave solution 
$v_\lambda(\lambda \omega t, x)$ provided in Theorem 
\ref{teo principale beta plane}. 
In particular, the result in \cite{BFMT1} guarantees the existence of an asymptotically full measure set 
$\Omega_{\lambda}$ such that 
$v_\lambda(\cdot; \omega) \in H^S(\T^\nu \times \T^2)$ where $S>s_0$
 can be chosen arbitrarily large and 
 \begin{align}
 & v_\lambda(\vf, x) = - v_\lambda(- \vf, - x), \quad \forall \, (\vf, x) \in \T^\nu \times \T^2\,, \\
 & \int_{\T^2} v_\lambda(\vf, x) \wrt x = 0, \quad \forall \, \vf \in \T^\nu\,,
 \end{align}
 see Theorem \ref{teo principale beta plane}. Moreover, by \eqref{stima.grande.solutione} and by setting
 \begin{equation}\label{theta.def.ridu}
	\theta := \alpha-1+\tc \,,\qquad  \alpha\in(1,2)\,,
	%\qquad \gamma := \lambda^{- \mathtt c}
\end{equation}
we also have that there exists 
 a constant $C_0 \equiv C_0(S) \gg 0$ large enough such that
 \begin{equation}\label{ansatz}
	\|  v_\lambda \|_{S}^{\Lip(\gamma)}	  \leq C_0 \lambda^\theta\,, \qquad \gamma := \lambda^{- \mathtt c}\,. 
\end{equation}
We remark that the parameter $0 < \tc \ll 1$ is small enough in such a way that, for any $\alpha\in(1,2)$,
\begin{equation}\label{small.theta}
	\theta-1 < \theta - 1 + \tc < 0\,,
\end{equation}
see Propositions 3.3 and 7.3  in \cite{BFMT1}. In particular 
\begin{equation}\label{choice mathtt c}
0 < \mathtt c < \tfrac13 (2 - \alpha)\,. 
\end{equation}
By linearizing \eqref{beta.waves.large.eq} at the quasi-periodic 
traveling wave solution 
$v_\lambda(\lambda \omega t, x)$ we get the linear PDE
\begin{equation}\label{operatore linearizzato}
\begin{aligned}
& \partial_t v = {\mathcal L}(\lambda \omega t)[v(t)]\,, 
\;\;\qquad  {\mathcal L}(\vphi) :=  \beta\,\tL +  \ba_{0}(\vf, x)\cdot \nabla +  {\cE_{0}}(\vf)
\end{aligned}
\end{equation}
where
\begin{equation}\label{def coefficienti operatori linearized}
{\mathtt L} := \partial_{x_1}(- \Delta)^{- 1}\,, 
\qquad 	
\ba_{0}(\vf, x):=    - \fB\big[  v_\lambda(\vf, x) \big] \,, 
\qquad 
\cE_{0}(\vf)[h] :=  - \nabla  v_\lambda(\vf, \cdot) \cdot \fB [h]\,.
\end{equation}
By  \eqref{ansatz} and by the definitions 
\eqref{def coefficienti operatori linearized}, 
we get that, the function $\ba_0$ is in  $H^{S+1}$, 
the operator $\cE_{0}$ belongs to  $\OpM_{S - 1}^{- 1}$
for any $\omega\in \Omega_{\lambda}$, and moreover we have the estimates
\begin{equation}\label{stima bf a1}
\| \ba_{0} \|_{S + 1}^{\Lip(\gamma)} \lesssim_S   \lambda^\theta \,,
\qquad 
	|\cE_{0}|_{- 1, S - 1}^{\Lip(\gamma)} \lesssim_{S} \lambda^\theta\,.  
\end{equation}
By the definition of $\ba_{0}$ in \eqref{def coefficienti operatori linearized} 
and the definition of the Biot-Savart operator $\fB$ in \eqref{biot-savart}, 
we clearly have that
\begin{equation}\label{divergenza.zero.a1}
	\braket{\ba_{0}}_{x} :=\frac{1}{(2\pi)^2}\int_{\T^2}\ba_{0}(\vf,x) \wrt x = 0\,, 
	\qquad 
	{\rm div}(\ba_{0}):= \nabla\cdot \ba_{0} = 0 \,.
\end{equation}
Moreover, using also that ${\rm div} (\nabla^\perp  h) = 0$ for any $h$, 
the operators $\ba_{0} \cdot \nabla$, $\cE_{0}$ and ${\mathcal L}$ 
leave invariant the subspace of zero average functions, 
namely
\begin{equation}\label{invarianze.2}
\ba_{0} \cdot \nabla = \Pi_0^\bot \ba_{0} \cdot \nabla \Pi_0^\bot\,, 
\qquad 
\cE_{0} = \Pi_0^\bot \cE_{0} \Pi_0^\bot\,, 
\qquad 	
{\mathcal L} = \Pi_0^\bot {\mathcal L} \Pi_0^\bot \,.
\end{equation}
We always work on the space of zero average functions 
and we shall preserve this invariance along the whole paper. Since $v_\lambda$ is a traveling wave, ${\mathcal L}$ is momentum preserving, according to the definition \ref{def:mom.pres}. Furthermore, since $v_\lambda$ is odd, ${\mathcal L}$ is reversible according to the definition \ref{reserv.operators.def}. 
The main result of this section concerns the reduction to constant coefficients 
of the linear vector field ${\mathcal L}(\vphi)$. The precise statement is given below.
\begin{thm}\label{reducibility linearized totale}
There are constants $\overline \sigma \equiv \overline\sigma(\alpha, \nu)$, 
$\eta \equiv \eta(\alpha, \nu)$, 
such that, for any $S > s_0 + \overline \sigma$, 
there exists $\lambda_0 \equiv \lambda_0(S, \nu, \alpha, \beta) > 0$ such that the following holds. 
For any $\lambda \geq \lambda_0$ there exists a Borel set ${\mathcal O}_\lambda \subseteq \Omega$, 
with $|\Omega \setminus {\mathcal O}_\lambda| \to 0$ as $\lambda \to + \infty$ 
and, for any $\omega \in {\mathcal O}_\lambda$, 
there are two families of bounded and invertible operators 
${\mathcal B}(\vf)$, ${\mathcal W}(\vf)$, $\vf \in \T^\nu$ satisfying the following properties.
The map ${\mathcal B}(\vf)$ is invertible and has the form
\begin{equation}\label{prop cambio variabile teoremone lin}
\begin{aligned}
& {\mathcal B}(\vf) : w(x) \mapsto w(x +  \bbeta(\vf, x))\,, 
\qquad {\mathcal B}(\vf)^{- 1} : w(y) \mapsto w(y + \breve\bbeta(\vf, y))\,. 
\end{aligned}
\end{equation}
Moreover, $\mathcal{B}^{\pm1}(\vphi)$ satisfy the estimates
\begin{align}
 \| \bbeta \|_s\,\, \|\breve\bbeta \|_s &\lesssim_S \lambda^{- \eta}\,, 
 \qquad 
 \forall \, s_0 \leq s \leq S - \overline \sigma\,, 
 \\
\sup_{\vf \in \T^\nu} \|{\mathcal B}(\vf)^{\pm 1} \|_{{\mathcal B}(H^s)} 
&\leq 1 + C(S) \lambda^{- \eta}\,, \quad \forall \, 0 \leq s \leq S - \overline \sigma\,.
\end{align}
The maps ${\mathcal W}(\vf)^{\pm 1}$, $\vf \in \T^\nu$, are one-smoothing perturbation of the identity. 
More precisely, they satisfy the estimate 
\begin{equation}\label{prop one smoothing teo principale}
\sup_{\vf \in \T^\nu} \|{\mathcal W}(\vf)^{\pm 1} - {\rm Id} \|_{{\mathcal B}(H^{s - 1}_0, H^s_0)} 
\lesssim_S \lambda^{- \eta}\,, 
\qquad 
\forall \,1 \leq s \leq S - \overline \sigma\,. 
\end{equation}
For any $\omega \in {\mathcal O}_\lambda$, the map 
${\mathcal U}(\vf) := {\mathcal B}_\bot(\vf) \circ {\mathcal W}(\vf)$, $\vf \in \T^\nu$, 
where ${\mathcal B}_\bot(\vf) := \Pi_0^\bot {\mathcal B}(\vf) \Pi_0^\bot$, 
satisfies (see \eqref{campo vettoriale astratto trasformato}) 
${\mathcal U}_* {\mathcal L}(\vf) = {\mathcal D}$, where the real operator 
${\mathcal D} := {\rm diag}_{j \in \Z^2 \setminus \{ 0 \}} \mu_\infty(j)$ 
is diagonal with purely imaginary eigenvalues $\mu_\infty(j) \in \im \R$, $j \in \Z^2 \setminus \{ 0 \}$. 
\end{thm}

We now present the step-by-step scheme of reduction for the operator 
${\mathcal L}$ defined in \eqref{operatore linearizzato}.  

\subsection{Reduction of the transport}

We consider the composition operator
\begin{equation}\label{diffeo.maps}
\cB(\vf) [h] (x) := h ( x+ \bbeta(\vf,x)) \,, 
\qquad 
\cB(\vf)^{-1} [h]  (y) := h ( y + \breve{\bbeta}(\vf,y))\,,
\end{equation}
induced by a $\vf$-dependent family of maps of the form
\begin{equation}\label{diffeotoro0}
\T^2 \to \T^2 \,, \qquad x \mapsto y := x + \bbeta (\vf,x) \,.
\end{equation}
By using Lemma 2.3-$(ii)$, $(iii)$ in \cite{BM20} if  $\|  \bbeta\|_{s_0}^{\Lip(\gamma)}\leq \delta $ 
for some $\delta \equiv \delta(s_0) >0$ sufficiently small, 
then the map in \eqref{diffeotoro0}
%\[
%\T^2 \to \T^2\,, \qquad x \mapsto x + \bbeta(\vf, x)\,,
%\]
is invertible with an inverse given by 
\[
 \T^2 \to \T^2 \,, \qquad y \mapsto  x= y + \breve{\bbeta}(\vf, y) \,,
\]
for some function $\breve{\bbeta}$ satisfying
\[
\| \breve{\bbeta}\|_{s}^{\Lip(\gamma)}
\lesssim_{s} 
\| \bbeta\|_{s}^{\Lip(\gamma)}\,, \qquad  \text{for} \quad s \geq s_0\,.
\]
   Moreover, if  $\bbeta(\vf,x)$ is a quasi-periodic traveling wave function, 
   then also $\breve{\bbeta}(\vf,y)$ is a quasi-periodic traveling 
   wave and hence (see Lemma 2.15 in \cite{BFMT1}) 
   the linear operators ${\mathcal B}(\vf), {\mathcal B}(\vf)^{- 1}$ 
   are momentum preserving. Furthermore, in Lemma 4.2 in \cite{FrMo}, 
   it is proved that the map 
\begin{equation}\label{def cal B bot}
{\mathcal B}_\bot(\vf ) = \Pi_0^\bot {\mathcal B}(\vf) \Pi_0^\bot : H^s_0(\T^2) \to H^s_0(\T^2)\,,
\end{equation}
is invertible with inverse given by (recall  \eqref{definizione proiettore media spazio tempo})
\begin{equation}\label{cal B bot inverse}
{\mathcal B}_\bot(\vf)^{- 1 } = \Pi_0^\bot {\mathcal B}(\vf)^{- 1} \Pi_0^\bot : H^s_0(\T^2) \to H^s_0(\T^2)\,.
\end{equation}
We need to reduce to constant coefficients 
the highest order part of the vector field ${\mathcal L}(\vf)$, which is the transport-type operator 
\begin{equation}\label{transport vector field}
{\mathcal T}(\vf) := {\bf a}_0(\vf, x) \cdot \nabla\,.
\end{equation}
By recalling \eqref{campo vettoriale astratto trasformato}, a direct caluclation shows that 
\begin{align}
{\mathcal L}_1(\vf) &:= (\cB_\bot)_* {\mathcal L}(\vf) 
= 
\Pi_0^\bot {\bf b_0}(\vf, x) \cdot \nabla \Pi_0^\bot + \cB_\bot(\vphi)^{- 1} 
\big(\beta \, \mathtt L + {\mathcal E}_0(\vf) \big) \cB_\bot(\vf)\,, 
\\
{\bf b}_0 &:= \cB(\vf)^{-1} \big( \lambda \omega\cdot \pa_{\vf}\bbeta + {\bf a_0} + {\bf a}_0 \cdot \nabla \bbeta \big) \label{trasporto trasformato}
\\
& 
 = \lambda \cB(\vf)^{-1} \big( \omega\cdot \pa_{\vf}\bbeta + {\bf b} + {\bf b} \cdot \nabla \bbeta \big)\,, 
\qquad {\bf b} := \lambda^{- 1} {\bf a}_0\,. 
\end{align}
Note that, for $\lambda \gg 1$, by \eqref{theta.def.ridu}-\eqref{small.theta}, 
one has that $\varepsilon = \lambda^{\theta-1} \ll 1$ 
and the quasi-periodic traveling wave function ${\bf b}(\vphi, x)$ 
satisfies the  estimate (recall \eqref{stima bf a1}) 
\begin{equation}\label{ridu trasporto}
\| {\bf b} \|_{S + 1}^{\Lip(\gamma)} \lesssim_S \lambda^{\theta-1}\,. 
\end{equation}
For $\gamma\in (0,1)$ and $\tau>\nu-1$,
we consider the set of  Diophantine 
non-resonance conditions defined as
%in \eqref{DC.2gamma}.
 \begin{equation}\label{DC.2gamma}
\tD\tC(2\gamma,\tau) := \big\{  \omega \in \Omega_\lambda \, : \, 
|  \omega \cdot\ell  | \geq 2  \gamma \braket{\ell}^{-\tau} \ \ 
\forall \, \ell\in\Z^{\nu}\setminus\{0\}    \big\}\,,
\end{equation}
where $\Omega_{\lambda}$ is  given by Theorem 
\ref{teo principale beta plane}.
By standard volume estimates for diophantine frequencies, one has 
\begin{equation}\label{stima diofantei}
|\Omega_\lambda \setminus \tD\tC(2\gamma,\tau)| 
\lesssim \gamma \,.
%\lesssim \lambda^{- \mathtt c}\,.
\end{equation}
The following result is proved in Proposition 5.1 of \cite{BFMT1}. 
\begin{prop}\label{proposizione trasporto}
{\bf (Straightening of the transport vector field).} 
Let $\gamma\in (0,1)$ and $\tau>\nu-1$.
There exists  $\sigma := \sigma (\tau, \nu) > 0$ large enough  such that, if  
$S > s_0 + \sigma$, there exist $\delta := \delta(S, \tau, \nu) \in (0, 1)$ 
small enough such that, 
if \eqref{ansatz}, \eqref{ridu trasporto} hold  and
\begin{equation}\label{condizione piccolezza rid trasporto}
\lambda^{\theta - 1} \gamma^{- 1}  \leq \delta \,,
\end{equation} 
is fulfilled, then the following holds. 
There exists an invertible diffeomorphism 
$\T^2 \to \T^2$, $x \mapsto x + \bbeta(\vphi, x; \omega)$ 
with inverse $y \mapsto y + \breve \bbeta(\vphi, y; \omega)$,
defined for all {$\omega \in \tD\tC(2\gamma, \tau)$}, 
with the set given in \eqref{DC.2gamma},
satisfying, for any $s_0\leq s \leq S - \sigma$,
\begin{equation}\label{stima alpha trasporto}
\| \bbeta \|_s^{{\rm Lip}(\gamma)}\,,\; 
\| \breve \bbeta\|_{s}^{{\rm Lip}(\gamma)} 
\lesssim_{S}  
\lambda^{\theta - 1} \gamma^{- 1} \,,
\end{equation}
such that one has
\begin{equation}\label{coniugazione nel teo trasporto}
\omega\cdot \pa_{\vf}\bbeta + {\bf b} + {\bf b} \cdot \nabla \bbeta = 0\,. 
\end{equation}
Moreover, the invertible maps $\cB(\vf)$, $\cB(\vf)^{-1}$ (see \eqref{diffeo.maps}) 
satisfy the estimates
\begin{equation}\label{stima tame cambio variabile rid trasporto}
\begin{aligned}
& \sup_{\vf \in \T^\nu}\| {\mathcal B}(\vf)^{\pm 1}  \|_{{\mathcal B}(H^s)} 
\leq 
1 + K(S) \lambda^{\theta - 1} \gamma^{- 1}  \,, 
\qquad \forall \, 0 \leq s \leq S - \sigma\,.  
\end{aligned}
\end{equation}
Furthermore, $\bbeta,\breve{\bbeta}$ are $\odd(\vphi,x)$, 
quasi-periodic traveling waves, and the related maps ${\mathcal B}, {\mathcal B}^{- 1}$ 
are momentum preserving and reversibility preserving 
(see Definitions \ref{def:mom.pres}, \ref{reserv.operators.def}). 
\end{prop} 	 
\begin{proof}
The existence of a map of the form \eqref{diffeo.maps} 
for $\bbeta$ satisfying estimate \eqref{stima alpha trasporto} and equation
\eqref{coniugazione nel teo trasporto} is proved 
in Proposition 5.1 of \cite{BFMT1}.
The only difference relies in the estimate 
\eqref{stima tame cambio variabile rid trasporto} which we now prove.

More precisely we claim that
the composition operator $\cB$ in \eqref{diffeo.maps} 
satisfies the estimate
\begin{equation}\label{stima Hs x A vphi}
\sup_{\vf \in \T^\nu} \| {\mathcal B}(\vf) \|_{{\mathcal B}(H^s)} 
\leq 
1 + K_s(S) \lambda^{\theta - 1} \gamma^{- 1} , 
 \qquad \forall \, 0 \leq s \leq S - \sigma
\end{equation}
for some constant $K_s(S) > 0$. The inverse operator 
\begin{equation}\label{diffeo.A-1}
{\mathcal B}(\vf)^{- 1} : u(y) \mapsto u(y + \breve{\bbeta}(\vf, y))
\end{equation}
also satisfies 
\[
\sup_{\vf \in \T^\nu} \| {\mathcal B}(\vf)^{- 1} \|_{{\mathcal B}(H^s)} 
\leq 
1 + K'_s(S) \lambda^{\theta - 1} \gamma^{- 1} \,, 
\qquad \forall \, 0 \leq s \leq S-\sigma\,,
\]
for some constant $K'_s(S) > 0$ large enough.

\noindent
{\sc Proof of \eqref{stima Hs x A vphi}.} We argue by induction on $0 \leq s \leq S - \sigma$, $s \in \N$. 
If $s$ is not integer the claimed bound is obtained by standard interpolation, 
namely, for $p \in [s, s+ 1]$ and  $s \in \N$,
\[
\| \cdot \|_{{\mathcal B}(H^p)} 
\lesssim 
\| \cdot \|_{{\mathcal B}(H^s)}^\lambda \| \cdot \|_{{\mathcal B}(H^{s + 1})}^{1 - \lambda}\,,
\] 
for some $\lambda \in [0, 1]$. For $s = 0$, one has that 
for any $u \in L^2(\T^2)$, 
\begin{align}
\| {\mathcal B}(\vf)[u] \|_{L^2}^2 
& 
= \int_{\T^2} |u(x + \bbeta(\vf, x))|^2 \wrt x 
= \int_{\T^2} |u(y)|^2 {\rm det}\big({\rm Id} 
+ \nabla_y \breve{\bbeta}(\vf, y) \big) \wrt y 
\\
& \leq \| u \|_{L^2}^2 
\big(1 + C_0 \| \nabla \breve{\bbeta} \|_{L^\infty} \big)\,,
\end{align}
for some constant $C_0 > 0$. By the Sobolev embedding 
%and by Lemma \ref{prop.base.diffeo}, 
one has that 
\begin{equation}\label{nabla alpha L infty}
\begin{aligned}
\| \nabla \breve{\bbeta} \|_{\infty} & \lesssim 
\| \nabla \breve{\bbeta} \|_{\frac{2 + \nu}{2} + 1} 
\lesssim 
\| \breve \bbeta \|_{\frac{2 + \nu}{2} + 2}  
\stackrel{\eqref{definizione s0}}{\lesssim}
\| \breve \bbeta \|_{s_0} \lesssim \| \bbeta \|_{s_0} 
\lesssim_S  \lambda^{\theta - 1} \gamma^{- 1}
\end{aligned}
\end{equation}
and hence the claimed bound follows for $s = 0$. 

\noindent
We now argue by induction. Assume that the claimed bound 
follows for some $s \in \N$ and let us prove it for $s + 1$. One has that 
\begin{align}
\| {\mathcal B}(\vf)[u] \|_{H^{s + 1}} & \leq 
\| {\mathcal B}(\vf)[u] \|_{L^2} 
+  \| \nabla_x {\mathcal B}(\vf)[u] \|_{H^s} 
\\
& \leq  \| {\mathcal B}(\vf)[u] \|_{L^2} 
+  \| \big({\rm Id} 
+ \nabla_x \bbeta \big)^T {\mathcal B}(\vf)[\nabla u] \|_{H^s}\,.
\end{align}
We need to estimate only the second term in the latter equality. One has 
\[
\begin{aligned}
\| \big({\rm Id} + \nabla_x \bbeta \big)^T 
{\mathcal B}(\vf)[\nabla u] \|_{H^s} 
& \leq 
\| {\mathcal B}(\vf)[\nabla u] \|_{H^s} \Big( 1 + C_1(s) 
\| \nabla \bbeta \|_{{\mathcal C}^s}\Big)\,,
\end{aligned}
\]
for some constant $C_1(s) > 0$. By the induction hypothesis one has
\begin{equation}\label{sifulo 1}
\begin{aligned}
\| {\mathcal B}(\vf)[\nabla u] \|_{H^s} & 
\leq \big( 1 + K_s(S) \lambda^{\theta - 1} \gamma^{- 1}\big) 
\|\nabla u \|_{H^s}\,.
\end{aligned}
\end{equation}
Furthermore, by the Sobolev embedding and \eqref{stima alpha trasporto}
%and Lemma \ref{prop.base.diffeo}, 
one has, taking $s + 1 + s_0 < S - \sigma$, that
\[
\begin{aligned}
\| \nabla \bbeta \|_{{\mathcal C}^s} & \leq 
\|  \bbeta \|_{{\mathcal C}^{s + 1}} 
\lesssim_s 
\|  \bbeta \|_{s + 1 + \frac{2 + \nu}{2} + 1}  
\stackrel{\eqref{definizione s0}}{\lesssim_s} 
\|  \bbeta \|_{s + 1 + s_0} 
\lesssim_S \lambda^{\theta - 1} \gamma^{- 1}\,,
\end{aligned}
\]
and therefore 
\[
\| \big({\rm Id} + \nabla_x \bbeta \big)^T 
{\mathcal B}(\vf)[\nabla u] \|_{H^s} 
\leq  
\big(1 + C(S) \lambda^{\theta - 1} \gamma^{- 1} \big)
\big( 1 + K_s(S) \lambda^{\theta - 1} \gamma^{- 1}\big) 
\|\nabla u \|_{H^s}\,.
\]
This latter estimate together with the $L^2$-estimate implies that 
\[
\| {\mathcal B}(\vf)[ u] \|_{H^{s + 1}} 
\leq \big(1 + K_{s + 1}(S) \lambda^{\theta - 1} \gamma^{- 1} \big) 
\| u \|_{H^{s + 1}}\,,
\]
for a large constant $K_{s + 1}(S) \gg K_s(S)$ and 
by taking $\lambda^{\theta - 1} \gamma^{- 1} \ll 1$
small enough. 
\end{proof}

\noindent
The latter Proposition implies that the linear vector field ${\mathcal L}_1(\vf)$ 
takes the form 
\begin{equation}\label{cal L1 dopo stratening}
{\mathcal L}_1(\vf) = 
\cB_\bot(\vphi)^{- 1} \big(  \beta \, \mathtt L + {\mathcal E}_0(\vf) \big) 
\cB_\bot(\vf), \quad \vf \in \T^\nu\,. 
\end{equation}
The properties of the vector field ${\mathcal L}_1(\vf)$ 
are given in the following Proposition.
\begin{prop}\label{prop coniugio cal L L1}
Let $S>s_0+\sigma_{1}$, for some $\sigma_{1}=\sigma_{1}(\tau,\nu) \gg \sigma$ 
(where $\sigma $ is provided by Proposition \ref{proposizione trasporto}). 
There exists $\delta(S,\tau,\nu)\in(0,1)$ small enough such that, 
if \eqref{ansatz} and \eqref{condizione piccolezza rid trasporto} are fulfilled, 
the following holds.
 For any  {$\omega\in \tD\tC(2\gamma,\tau)$}, 
 as  in \eqref{DC.2gamma}, one has
\begin{equation}\label{cL1}
\cL_{1}(\vf) = \beta  \,\tL+ \cE_{1}(\vf)\,,
\end{equation}
where, for any $s_0 \leq s \leq S - \sigma_1$,
the operator $\cE_{1}\in \OpM_{s}^{-1}$ satisfies the estimate 
\begin{equation}\label{stima mathcal E1}
| \cE_{1} |_{-1,s}^{\Lip(\gamma)} 
\lesssim_{S} \lambda^{\theta} \,.
\end{equation}
Furthermore, the operators $\cL_{1}$ and $\cE_{1}$ are  real, 
reversible, momentum preserving.  
\end{prop}
\begin{proof}
The proof is exactly the same as the one of Proposition 5.5 in \cite{BFMT1}. 
\end{proof}

\subsection{Reduction to perturbative of the large remainder }

The next goal is to conjugate the operator $\cL_{1}$ in \eqref{cL1} 
through a series of transformation in order to reduce both 
the size and the order of the remainder $\cE_{1}\in  \OpM_{s}^{-1}$. 
We shall prove the following Proposition.
\begin{prop}\label{prop normal form lower orders}
Let $M\in\ \N$, $M > \frac{1-\tc}{2(1-\tc) - \alpha}$. 
There exist $\sigma_M= \sigma_{M}(\tau,\nu) \gg 0$ large enough such that, 
if $S>s_0 +\sigma_{M}$, there is $\delta \equiv \delta(S, M) \ll 1$ small enough such that, 
if 
\begin{equation}\label{piccoloNonpert}
\lambda^{\theta-1} \gamma^{- 1} \leq \delta\,,
\end{equation}
and if \eqref{ansatz} is fullfilled with $\sigma_{0} = \sigma_M$, 
then the following hold. For any  {$\omega\in \tD\tC(2\gamma,\tau)$}, 
there exists a real, invertible operator 
${\bf \Phi}_M \in \OpM_s^0$, $s_0 \leq s \leq S - \sigma_M$ satisfying 
\begin{equation}\label{stima bf Phi M}
\begin{aligned}
& |{\bf \Phi}_M^{\pm 1}|_{0, s}^{\Lip(\gamma)} 
\lesssim 1\,, 
\qquad 
|{\bf \Phi}_M^{\pm 1}-{\rm Id}|_{-1, s}^{\Lip(\gamma)} 
\lesssim_{S, M}
 \lambda^{\theta - 1} \gamma^{- 1}  \,, 
 \qquad \forall \, s_0 \leq s \leq S - \sigma_M\,
\end{aligned}
\end{equation}
such that, for any  {$\omega\in \tD\tC(2\gamma,\tau)$}, we have
\begin{equation}\label{forma cal LM finale}
{\mathcal L}_M(\vf) := (\Phi_M)_*{\mathcal L}_1(\vf) 
=   \beta\,\tL + {\mathcal Z}_M + {\mathcal E}_M(\vf)\,,
\end{equation}
where ${\mathcal Z}_M \in \OpM^{- 1}_s$, $s \geq 0$, 
is a diagonal operator 
${\mathcal Z}_M := {\rm diag}_{j \neq 0} z_M(j)$ 
and ${\mathcal E}_M \in \OpM^{- M}_s$, $s_0 \leq s \leq S - \sigma_M$ 
satisfy the estimates 
\begin{align}
& |{\mathcal Z}_M|_{- 1, s}^{\Lip(\gamma)} \lesssim_{S, M} \lambda^{\theta}\,, 
\qquad 
\forall \, s \geq s_0\,, 
\qquad  
\sup_{j \neq 0} |j| |z_M(j)|^{\Lip(\gamma)} \lesssim_{S, M} \lambda^{\theta}\,, \label{stime induttive cal EM cal ZM finale}
\\
&  |{\mathcal E}_M|_{- M, s}^{\Lip(\gamma)} 
\lesssim_{S, M} 
(\lambda^{\theta-1}\gamma^{-1} \big)^{M-1} \lambda^{\theta}   
\qquad 
\forall \, s_0 \leq s \leq S - \sigma_M\,. 
\end{align}

\noindent
Furthermore, the operators $\b\Phi_{M}^{\pm 1}$ and $\cL_{M}$
are real and momentum preserving, $\b\Phi_{M}^{\pm 1}$ 
are reversibility preserving and  $\cL_{M}$ is reversible.
\end{prop}

Proposition \ref{prop normal form lower orders} is proved as 
a consequence of the following iterative procedure.

\begin{lem}\label{lemma.large.to.pert}
Let $M\in\ \N$, $M > \frac{1-\tc}{2(1-\tc) - \alpha} $. 
There exist $\sigma_1 < \sigma_2 < \ldots < \sigma_M$ with 
$\sigma_{i}:=\sigma_{i}(\tau,\nu)>0$ large enough such that, 
for any $S>s_0 +\sigma_{M}$, there exists $\delta \equiv \delta(S, M) \ll 1$ 
small enough such that, if \eqref{piccoloNonpert} holds 
%$\lambda^{\theta-1} \gamma^{- 1} \leq \delta$ 
and if \eqref{ansatz} is fulfilled with $\sigma_{0} = \sigma_M$, 
then the following holds. For any $m = 1, \ldots, M$, 
there exists a real operator ${\mathcal L}_m$ of the form
\begin{equation}\label{op mathcal Lm}
{\mathcal L}_m(\vf) =  \beta \,\tL + {\mathcal Z}_m + {\mathcal E}_m(\vf)\,, \quad \vf \in \T^\nu\,,
\end{equation}
where,  for any $s_0 \leq s \leq S - \sigma_m$, 
${\mathcal Z}_m \in \OpM^{- 1}_s$ is the diagonal operator ${\mathcal Z}_m = {\rm diag}_{j \neq 0} z_m(j)$ , ${\mathcal E}_m \in \OpM^{- m}_s$ and they satisfy the estimates
\begin{align}
 |{\mathcal Z}_m|_{- 1, s}^{\Lip(\gamma)} &\lesssim_{S, m} 
\lambda^{\theta} \,,  
\qquad  
\sup_{j \neq 0} |j| |z_m(j)|^{\Lip(\gamma)} 
\lesssim_{S, m} \lambda^{\theta}\,,\label{stime induttive cal Em cal Zm}
\\
  |{\mathcal E}_m|_{- m, s}^{\Lip(\gamma)} 
&\lesssim_{S, m} 
\big( \lambda^{\theta-1} \gamma^{- 1} \big)^{m - 1} \lambda^{\theta} \,.
%\| {w} \|_{s + \sigma_m}^{\Lip(\gamma)} \,.
	\end{align}
There exist real operators 
$\{ \Phi_{m} \}_{m=1}^{M - 1}$, 
with $\Phi_{m}:= {\rm exp} (\cX_{m})\in \OpM_{s}^{0}$, 
where
\[
\cX_{m} := \cX_{m}(\vf) 
= \sum_{\ell\in\Z^{\nu}\setminus\{0\}} \wh\cX_{m}(\ell)e^{\im\,\ell\cdot\vf} 
\in \OpM_{s}^{-m}\,,
\]
satisfying the estimates 
		\begin{align}
 |\cX_m|_{- m, s}^{\Lip(\gamma)} 
&\lesssim_{S, m} 
\big(\lambda^{\theta-1} \gamma^{- 1} \big)^m\,, \label{stime cal Xm Phim nel lemma}
%\qquad \forall s_0 \leq s \leq S - \sigma_m\,, 
\\
 |\Phi_m^{\pm 1} |_{0, s}^{\Lip(\gamma)} &\lesssim 1\,, 
\qquad 
|\Phi_m^{\pm 1} - {\rm Id} |_{- m, s}^{\Lip(\gamma)} 
\lesssim_{S, m} (\lambda^{\theta - 1} \gamma^{- 1})^m \,,
\qquad \forall \, s_0 \leq s \leq S - \sigma_m\,. 
\end{align}
Moreover, for any $m = 2, \ldots, M - 1$ and any  
{$\omega\in \tD\tC(2\gamma,\tau)$}, we have that 
\begin{equation}\label{operator.pertubative}
\cL_{m}(\vf) := (\Phi_m)_* \cL_{m - 1}(\vf), \quad \vf \in \T^\nu   \,.
\end{equation}
Furthermore, the operators $\b\Phi_{m}^{\pm 1}$ and $\cL_{m}$
are real and momentum preserving and $\b\Phi_{m}^{\pm 1}$ 
are reversibility preserving and $\cL_{m}$ is reversible.
\end{lem}
\begin{proof}
We proceed by induction.
The operator $\cL_{1}$ in \eqref{cL1} is of the form \eqref{op mathcal Lm} 
with $\cZ_{1}=0$.
By \eqref{stima mathcal E1}  we have that 
$\cE_{1}$ satisfies \eqref{stime induttive cal Em cal Zm}.
Finally, by Proposition \ref{prop coniugio cal L L1} we have that 
$\cL_{1}$ is real, reversible and momentum preserving.

\noindent
We now assume that the claimed statements hold for $m = 1, \ldots, M - 1$ 
and we prove them for $m + 1$. 
We look for a transformation of the form 
\begin{equation}
\Phi_{m}:= {\rm exp}(\cX_m) \in \OpM_{s}^{0}\,, 
\quad 
\cX_{m} := \cX_{m}(\vf) = 
\sum_{\ell\in\Z^\nu\setminus\{0\}} \wh\cX_{m}(\ell) e^{\im\,\ell\cdot\vf} \in \OpM_{s}^{-m}\,,
\end{equation} 
where $\wh\cX_{m}$ has to be determined. 
By recalling \eqref{campo vettoriale astratto trasformato}, 
we compute 
   \[
   {\mathcal L}_{m + 1}(\vf) := (\Phi_m)_* {\mathcal L}_m(\vf) = 
   \Phi_m(\vf)^{- 1} (\mathtt L + {\mathcal Z}_m 
   + {\mathcal E}_m(\vf)) \Phi_m(\vf) - \Phi_m(\vf)^{- 1} \lambda 
   \omega \cdot \partial_\vf \Phi_m(\vf)\,.
   \]
  We examine separately  
  the conjugation of the three terms appearing 
  in \eqref{op mathcal Lm}. By the standard Lie expansion, 
  we get
		\begin{align}
			- \Phi_m(\vf)^{- 1} \lambda \omega \cdot \partial_\vf \Phi_m(\vf) & = - \lambda \, \omega \cdot \partial_\vphi \cX_m(\vphi) + {\mathcal Q}_m^{(1)}(\vphi) \,, \\
			{\mathcal Q}_m^{(1)}(\vphi) & = - \int_0^1 (1 - \tau ) {\rm exp}(- \tau \cX_m) [\lambda\, \omega \cdot \partial_\vphi \cX_m(\vphi), \cX_m(\vphi)]  {\rm exp}( \tau \cX_m) \wrt \tau\,, \label{coniugio a pezzi cal L1 Phi1 m} \\
			\Phi_m(\vf)^{- 1} ( \beta \tL) \Phi_m(\vf) & =  \beta \tL + {\mathcal Q}_m^{(2)}(\vf)\,, \\
			{\mathcal Q}_m^{(2)}(\vphi) & :=  \int_0^1 {\rm exp}(- \tau \cX_m(\vphi)) [\beta \tL, \cX_m(\vphi)] {\rm exp}(\tau \cX_m(\vphi)) \wrt \tau\,,
		\end{align}
and
\begin{align}
\Phi_m(\vf)^{- 1} {\mathcal Z}_m \Phi_m(\vf) & = {\mathcal Z}_m + {\mathcal Q}_m^{(3)}(\vf)\,, \\
			{\mathcal Q}_m^{(3)}(\vphi) & := \int_0^1 {\rm exp}(- \tau \cX_m(\vphi)) [{\mathcal Z}_m, \cX_m(\vphi)] {\rm exp}(\tau \cX_m(\vphi)) \wrt \tau \,, \label{coniugio a pezzi cal L1 Phi1 mBIS} \\
			\Phi_m(\vf)^{- 1} {\mathcal E}_m(\vf) \Phi_m(\vf) & = {\mathcal E}_m(\vf)  + {\mathcal Q}_m^{(4)}(\vf)\,, \\
			{\mathcal Q}_m^{(4)}(\vphi) & :=  \int_0^1 {\rm exp}(- \tau \cX_m(\vphi)) [{\mathcal E}_m(\vphi), \cX_m(\vphi)] {\rm exp}(\tau \cX_m(\vphi)) \wrt \tau\,.
\end{align}
    
The terms of order $- m$ in the expansion of ${\mathcal L}_{m + 1}(\vf)$ 
that we want to reduce are then given by 
$- \lambda\, \omega \cdot \partial_\vphi \cX_m(\vphi) + {\mathcal E}_m(\vphi)$. 
Hence, we solve the homological equation 
\begin{equation}\label{equazione omologica step - 1 m}
- \lambda \, \omega \cdot \partial_\vphi \cX_m(\vphi) + {\mathcal E}_m(\vphi) 
= \widehat{\mathcal E}_m(0)\,, 
\qquad 
\widehat{\mathcal E}_m(0) := 
\frac{1}{(2 \pi)^\nu} \int_{\T^\nu} {\mathcal E}_m(\vphi) \wrt \vphi \,,
\end{equation}
by defining, for any  {$\omega\in \tD\tC(2\gamma,\tau)$},
\begin{equation}\label{def cal X 1 m}
{\mathcal X}_m(\vphi) =  
\sum_{\ell \neq 0} \frac{\widehat{\mathcal E}_m(\ell)}{\im \lambda\, \omega \cdot \ell} 
e^{\im \ell \cdot \vphi}\,.
\end{equation}
Note that the matrix elements of ${\mathcal X}_m$ are given by 
\[
\widehat{\mathcal X}_m(\ell)_j^{j'} = 
\begin{cases}
 \dfrac{\widehat{\mathcal E}_m(\ell)_j^{j'}}{\im \lambda \,\omega \cdot \ell}\,,  
 &\ell \neq 0\,, 
 \qquad j, j' \in \Z^2 \setminus \{ 0 \}\,, \qquad \pi^\top(\ell) + j - j' = 0\,, 
 \\
0  &\text{otherwise.}
\end{cases}
\]
Then, since {$\omega\in \tD\tC(2\gamma,\tau)$}, one has that 
\[
|\widehat{\mathcal X}_m(\ell)_j^{j'}| 
\lesssim 
\gamma^{- 1} \lambda^{- 1} \langle \ell \rangle^\tau 
|\widehat{\mathcal E}_m(\ell)_j^{j'}|\,.
\]
By taking $S > s_0 + \sigma_{m} + 2 \tau + 1$, 
the latter two estimates, 
together with the Definition \ref{block norm} imply that 
$ {\mathcal X}_m \in \OpM_s^{- m}$ for any $s_0 \leq s \leq S - \sigma_{m} - 2 \tau - 1$,
with estimates
\begin{align}
& |{\mathcal X}_m|_{- m,  s}^{\Lip(\gamma)} 
\lesssim 
\gamma^{- 1} \lambda^{- 1} |{\mathcal E}_m|_{- m, s + 2 \tau + 1}^{\Lip(\gamma)}  
\stackrel{\eqref{stime induttive cal Em cal Zm}}{\lesssim_{S, m}} 
\lambda^{- 1} \gamma^{- 1} 
\big( \lambda^{\theta-1} \gamma^{- 1} \big)^{m - 1} 
\lambda^{\theta}   
\lesssim_{S, m}   
\big( \lambda^{\theta-1} \gamma^{- 1} \big)^{m } \,.\label{prop + stime cal X1 m}
\end{align}
Together with Lemma \ref{proprieta standard norma decay}-$(iii)$ 
and using that $\lambda^{\theta-1} \gamma^{- 1} \ll 1$ (see \eqref{piccoloNonpert}), 
we have that \eqref{prop + stime cal X1 m} implies,
for any $s_0 \leq s \leq S - \sigma_m - 2 \tau - 1$,
\begin{equation}\label{Prop Phi 1 Phi 1 inv m}
\sup_{\tau \in [- 1, 1]}|{\rm exp}(\tau \cX_m) - {\rm Id}|_{- m, s}^{\Lip(\gamma)} 
\lesssim_{S, m} 
\big( \lambda^{\theta-1} \gamma^{- 1} \big)^{m }\,, 
\qquad 	
\sup_{\tau \in [- 1, 1]}|{\rm exp}(\tau \cX_m)|_{0, s}^{\Lip(\gamma)}  
\lesssim 1 \,. 
\end{equation}
Using that ${\mathcal E}_m$ is real, 
reversible and momentum preserving 
one verifies that ${\mathcal X}_m$ and $\Phi_m$ are real, 
reversibility preserving and momentum preserving. 
By \eqref{coniugio a pezzi cal L1 Phi1 m}-\eqref{coniugio a pezzi cal L1 Phi1 mBIS}, we obtain that
		\begin{align}
{\mathcal L}_{m + 1}(\vf) & = 
(\Phi_m)_* {\mathcal L}_m(\vf)  
=  \beta\, \tL + {\mathcal Z}_{m + 1}  + {\mathcal E}_{m + 1}(\vf) \,, 
\\
{\mathcal Z}_{m + 1} & := {\mathcal Z}_m + \widehat{\mathcal E}_m(0)\,, \label{forma finale cal L2 m}
\\
{\mathcal E}_{m + 1}(\vf)  & := 
{\mathcal Q}_m^{(1)}(\vf) +  {\mathcal Q}_m^{(2)}(\vf) 
+   {\mathcal Q}_m^{(3)}(\vf) +  {\mathcal Q}_m^{(4)}(\vf)\,.  
\end{align}

\smallskip
\noindent	
{\sc Properties and estimates of ${\mathcal Z}_{m + 1}$.} 
Since ${\mathcal E}_m$ is momentum preserving, 
by Lemma \ref{lem:mom_pres} we have that the 
time-independent operator $\widehat{\mathcal E}(0)$ 
is diagonal and hence also
the operator ${\mathcal Z}_{m+1} = {\rm diag}_{j \neq 0} z_{m+1}(j)$ 
is a diagonal operator with 
$z_{m + 1}(j) :=  z_m(j) + \widehat{\mathcal E}_m(0)_j^{j}$, 
$j \in \Z^2 \setminus \{ 0 \}$. 
Furthermore, by Lemma \ref{proprieta standard norma decay}-$(iv)$ 
and by the induction estimates \eqref{stime induttive cal Em cal Zm}, 
we get that  the operator ${\mathcal Z}_{m + 1} \in \OpM^{- 1}_s$ 
for any $s \geq s_0$ and it satisfies, 
using also $\lambda^{\theta-1} \gamma^{- 1} \ll 1$ by \eqref{piccoloNonpert},
\begin{align}
|{\mathcal Z}_{m + 1}|_{- 1, s}^{\Lip(\gamma)} 
\lesssim |{\mathcal Z}_m|_{- 1, s}^{\Lip(\gamma)} 
+ |\widehat{\mathcal E}_m(0)|_{- m, s} 
&\lesssim |{\mathcal Z}_m|_{- 1, s}^{\Lip(\gamma)}  
+  |{\mathcal E}_m|_{- m, s_0}^{\Lip(\gamma)} \label{stima cal Z1 m} \\
&  \lesssim_{S, m} \lambda^{\theta} 
+ \big( \lambda^{\theta-1} \gamma^{- 1}\big)^{m - 1} \lambda^{\theta} 
\lesssim_{S, m} 
\lambda^{\theta}\,, 
\end{align}
which implies that
\[
\sup_{j \neq 0} |j| |z_{m + 1}(j)|^{\Lip(\gamma)} 
\lesssim_{S, m} \lambda^{\theta}\,. 
\]
	
	\medskip
	
	\noindent
	{\sc Properties and estimates of ${\mathcal E}_{m + 1}$.} We now estimate the remainder ${\mathcal E}_{m + 1}$ in \eqref{forma finale cal L2 m}. We have to analyze the four terms ${\mathcal Q}_m^{(1)}, {\mathcal Q}_m^{(2)}, {\mathcal Q}_m^{(3)}, {\mathcal Q}_m^{(4)}$ in \eqref{coniugio a pezzi cal L1 Phi1 m}-\eqref{coniugio a pezzi cal L1 Phi1 mBIS}. First, by \eqref{equazione omologica step - 1 m}, we note that
	\begin{equation}\label{seconda forma cal Q 1 (1) m}
		{\mathcal Q}_m^{(1)} = - \int_0^1 (1 - \tau ) {\rm exp}(- \tau \cX_m) { [{\mathcal Z}_{m + 1}-{\mathcal Z}_{m} - {\mathcal E}_m, \cX_m] } {\rm exp}( \tau \cX_m) \wrt \tau\,.
	\end{equation}
	By the estimates \eqref{stime induttive cal Em cal Zm}, 
	\eqref{prop + stime cal X1 m}, \eqref{stima cal Z1 m}, 
	using also Lemma \ref{proprieta standard norma decay}-$(i)$, 
	we get that, taking  $S > s_0 + \sigma_m + m + 2 \tau + 1$, 
	the operators 
	$ [{\mathcal Z}_{m + 1} , \cX_m] \,, {[{\mathcal Z}_{m} , \cX_m] } \,,[\tL, \cX_m] \, \in \OpM_s^{- m - 1},  
	[{\mathcal E}_m , \cX_m] \in \OpM_s^{- 2 m} {\subseteq} \OpM^{- m - 1}_s$ 
	for any $ s_0 \leq s \leq S - \sigma_m  - m- 2 \tau - 1$ 
	and they satisfy the following estimates, 
	recalling also that $\lambda^{\theta-1} \gamma^{- 1} \ll 1$ by \eqref{piccolo.ansatx}:
\begin{align}
| [{\mathcal Z}_{m} , \cX_m]|_{- m - 1, s}^{\Lip(\gamma)}\,,
| [{\mathcal Z}_{m + 1} , \cX_m]|_{- m - 1, s}^{\Lip(\gamma)} 
&\lesssim_{S, m}  
\lambda^{\theta} \big( \lambda^{\theta-1} \gamma^{- 1} \big)^m\,, 
\\
|[{\mathcal E}_m , \cX_m]|_{- m - 1 , s}^{\Lip(\gamma)} 
&\leq 
|[{\mathcal E}_m , \cX_m]|_{- 2m , s}^{\Lip(\gamma)}  \label{stime cal Q 1 1 3 m}
\\
&\lesssim_{S, m} \lambda^{\theta} 
\big( \lambda^{\theta-1} \gamma^{- 1}\big)^{m - 1} \big( \lambda^{\theta-1} \gamma^{- 1}\big)^{m}   
\lesssim_{S, m} \lambda^{\theta}  \big( \lambda^{\theta-1} \gamma^{- 1}\big)^{m}\,, 
\\
|[\tL, \cX_m]|_{- m - 1, s}^{\Lip(\gamma)}
& \lesssim_{S, m} 
\big( \lambda^{\theta-1} \gamma^{- 1}\big)^{m}  
\lesssim_{S, m} \lambda^{\theta} \big( \lambda^{\theta-1} \gamma^{- 1}\big)^{m} \,. 
\end{align}
Together with \eqref{Prop Phi 1 Phi 1 inv m}, 
and Lemma \ref{proprieta standard norma decay}-$(i)$, 
the latter estimates imply, 
for $\sigma_{m + 1} \geq \sigma_m + m + 2 \tau + 2$ and 
$S > s_0 + \sigma_{m + 1}$, 
that the operator ${\mathcal E}_{m + 1}$ belongs to $\OpM_s^{- m - 1}$ 
for any 
$s_0 \leq s \leq S - \sigma_{m + 1}$ and it satisfies
\begin{equation}\label{stima mathcal E2 m}
		\begin{aligned}
			&|{\mathcal E}_{m + 1}|_{- m - 1, s}^{\Lip(\gamma)} 
			\lesssim_{S, m} 
			\lambda^{\theta} \big(  \lambda^{\theta-1} \gamma^{- 1} \big)^m \,.
		\end{aligned}
	\end{equation}
Finally since by the induction hypothesis $\cL_m$ is real, 
reversible and momentum preserving and $\Phi_m$ is real, 
reversibility preserving and momentum preserving 
one gets that 
${\mathcal L}_{m + 1}, {\mathcal E}_{m + 1}, {\mathcal Z}_{m + 1}$ 
are real, reversible and momentum preserving operators.
This implies the claim at the step $m+1$.
\end{proof}
In order to prove Proposition \ref{prop normal form lower orders}, 
we first need the following.
\begin{lem}\label{composizione iterata smoothing}
Let $N \geq 1$, $s \geq s_0$, $\rho \geq 0$ and let 
${\mathcal F}_1, \ldots, {\mathcal F}_N$ with 
${\mathcal F}_n \in \OpM_s^{- \rho}$ and 
$\Phi_n := {\rm Id} + {\mathcal F}_n$. 
For any $n \in \{1, \ldots, N\}$, 
assume that $|{\mathcal F}_n|_{- \rho, s} \leq 1$ and let us define
\[
\varepsilon_N := {\rm max}\{ |{\mathcal F}_n|_{- \rho, s}^{\Lip(\gamma)} : n = 1, \ldots, N\}\,, 
\qquad N \geq 1\,.
\]
Then the map ${\bf \Phi}_N := \Phi_1 \circ \ldots \circ \Phi_N$ is of the form 
\[
{\bf \Phi}_N = {\rm Id} + {\bf F}_N\,,
\]
with ${\bf F}_N \in \OpM_s^{- \rho}$ and 
$|{\bf F}_N|_{- \rho, s}^{\Lip(\gamma)} \lesssim_s \varepsilon_N$. 
\end{lem}
\begin{proof}
To simplify notations we write $|\cdot |_{m, s}$ 
instead of $| \cdot |_{m, s}^{\Lip(\gamma)}$. 

\noindent
%{\sc Proof of $(i)$.} w
We argue by induction on $N$. For $N = 1$ the statement is trivial. 
Now assume that the claimed properties hold for ${\bf \Phi}_N$ for some $N\geq 1$ 
and we prove them for ${\bf \Phi}_{N + 1}$. One has that
\begin{align}
{\bf \Phi}_{N + 1} & = {\bf \Phi}_N \circ \Phi_{N + 1} 
= ({\rm Id} + {\bf F}_N) \circ ({\rm Id} + {\mathcal F}_{N + 1}) \label{bf Phi N N + 1 astratto}
\\
& = {\rm Id} + {\bf F}_{N + 1}, \qquad \text{with} \qquad {\bf F}_{N + 1} = {\bf F}_N 
+ {\mathcal F}_{N + 1} + {\bf F}_N \circ {\mathcal F}_{N + 1}\,. 
\end{align}
Then by Lemma \ref{proprieta standard norma decay}-$(i)$, \eqref{prop elementari} and the induction assumption on $\bF_{N}$, one gets that 
\begin{align}
|{\bf F}_{N + 1}|_{- \rho, s} 
& \lesssim_{s,\rho} |\bF_{N}|_{-\rho,s} + |\cF_{N+1}|_{-\rho,s} + |\bF_{N}|_{0,s}|\cF_{N+1}|_{-\rho,s} \\
& \lesssim_{s, \rho} 
|{\bf F}_N|_{- \rho, s} (1 + |{\mathcal F}_{N + 1}|_{- \rho, s}) 
+ |{\mathcal F}_{N + 1}|_{- \rho, s} 
 \\
& \lesssim_{s, \rho} 
|{\bf F}_N|_{- \rho, s} + |{\mathcal F}_{N + 1}|_{- \rho, s} 
\lesssim_{s, \rho} 
\varepsilon_N + |{\mathcal F}_{N + 1}|_{- \rho, s} 
\lesssim_{s, \rho} \varepsilon_{N + 1}\,. 
\end{align}
The claimed bound has then been proved.
\end{proof}

\begin{proof}[Proof of Proposition \ref{prop normal form lower orders}.]  Let ${\bf\Phi}_{1}:={\rm Id}$ 
and ${\bf \Phi}_M := \Phi_{1} \circ \Phi_{2} \circ \ldots \circ \Phi_{M - 1}$ for $M\geq 2$. 
By Lemma \ref{lemma.large.to.pert}, 
the estimates \eqref{stime cal Xm Phim nel lemma} 
and by applying Lemma \ref{composizione iterata smoothing}, 
we obtain that $\b\Phi_{M}^{\pm 1}$ satisfy the estimates 
\eqref{stima bf Phi M}. Moreover, if  {$\omega\in \tD\tC(2\gamma,\tau)$}, 
then the conjugation \eqref{forma cal LM finale}
holds for 
some $\mathcal{Z}_{M}$ and $\mathcal{E}_{M}$
satisfying 
\eqref{stime induttive cal EM cal ZM finale}
using the bounds \eqref{stime induttive cal Em cal Zm}
with $m\rightsquigarrow M$ and using also
Lemma \ref{proprieta standard norma decay}-$(iv)$.
%
%Then the conjugation \eqref{forma cal LM finale} 
%holds with ${\mathcal Z}_M$ satisfying the estimate in 
%\eqref{stime induttive cal EM cal ZM finale}, using also 
%Lemma \ref{proprieta standard norma decay}-$(v)$, 
%and ${\mathcal E}_M$ satisfying the estimate (see \eqref{stime induttive cal Em cal Zm} 
%with $m\rightsquigarrow M$), for any $s_0 \leq s \leq S - \sigma_M$,
%$$
%|{\mathcal E}_M|_{- M, s}^{\Lip(\gamma)} \lesssim_{s, M}  
%\big( \lambda^{\theta-1} \gamma^{- 1} \big)^{M - 1} \lambda^{\theta} \| {w} \|_{s + \sigma_M}^{\Lip(\gamma)} \,.
%$$
%Then, the desired bound on ${\mathcal E}_M$  in \eqref{stime induttive cal EM cal ZM finale}  
%follows since, using $M > \frac{1-\tc}{2(1-\tc) - \alpha} = \frac{1-\tc}{1-\tc -\theta} $, $\lambda \gg 1$ 
%large enough, and recalling \eqref{theta.def.ridu}, \eqref{small.theta}, 
%we have, if we assume $\gamma=\lambda^{-\tc}$,
%$$
%\big( \lambda^{\theta-1} \gamma^{- 1} \big)^{M - 1} \lambda^{\theta} 
%= \lambda^{M\theta -(M-1)(1-\tc)}  \ll 1\,.
%$$
%    
Finally, by Lemma \ref{lem:mom_prop} and Proposition \ref{prop coniugio cal L L1}, 
we obtain that $\b\Phi_{M}$, ${\mathcal L}_M$ 
are momentum preserving. Moreover $\b\Phi_{M}$ 
is real and reversibility preserving and ${\mathcal L}_M$ real and reversible. 
The proof of the Proposition is then concluded. 
\end{proof}

 \subsection{Perturbative reduction}\label{sez:iteraKAM}
 We are now in position to reduce the vector field $\cL_{M}(\vf)$  in
 \eqref{forma cal LM finale}
 %in \eqref{operator.pertubative} 
 to a diagonal vector field, with a perturbative reducibility iteration. We then define
 \begin{equation}\label{bL0_inizioKAM}
 	\bL_{0}(\vf) : =\cL_{M}(\vf) = \bD_{0}  + \bE_{0}(\vf), \quad \vf \in \T^\nu \,,
 \end{equation}
 where:
 \\[1mm]
 \noindent $\bullet$ The real and reversible diagonal operator $\bD_{0}$ is given by
 	\begin{align}
 	 &	\bD_{0}:= \beta \,\tL + \bZ_{0} = \diag_{j\in\Z^2\setminus\{0\}} \mu_{0}(j)\,, 
	 \quad \mu_{0}(j) = \im\,\beta \tL(j) + \tz_{0}(j) \,, \\
 	& \bZ_{0} := \cZ_{M} := \diag_{j\in\Z^2\setminus\{0\}} \tz_{0}(j)\,, 
	\quad \tz_{0} (j) := z_{M}(j) \,, \  \ j \in \Z^2\setminus\{0\}\,, 
 	\end{align}
  with $\tL(j)\in \R$ and $\tz_0 ( j )\in \im \R$  in Proposition 
  \ref{prop normal form lower orders};
  \\[1mm]
  \noindent $\bullet$ For any $s\in [s_0,S-\sigma_{M}]$, the operator 
  $\bE_{0}:= \cE_{M} \in \OpM_{s}^{- M}$ satisfies the estimate 
  (see \eqref{stime induttive cal EM cal ZM finale})
  \begin{equation}
  	| \bE_{0} |_{-M,s}^{\Lip(\gamma)} 
	\lesssim_{s, M} 
	\lambda^{M(\theta-1)+1} \gamma^{-(M-1)}\,.
  \end{equation}
  From now on we shall fix
\begin{equation}\label{definizioneGAMMA}
\gamma := \lambda^{- \mathtt c}\,.
\end{equation}
Thanks to this choice, recalling 
\eqref{theta.def.ridu}, \eqref{small.theta}, taking 
$M > \frac{1-\tc}{2(1-\tc)-\alpha}$ as in Proposition 
\ref{prop normal form lower orders}, we have, for 
$\lambda\gg1$ large enough, that
\begin{align}
\lambda^{\theta-1} \gamma^{-1} &= 
\lambda^{\theta - 1 + \mathtt c}=\lambda^{\alpha-2+\mathtt{c}} \ll 1 \,,\label{piccolo.ansatx}
\\
\lambda^{M(\theta-1)+1} \gamma^{-(M-1)}&=
\lambda^{M\theta -(M-1)(1-\tc)}=
\lambda^{1+\mathtt{c}-M(2(1-\mathtt{c})-\alpha) }\ll1\,.
\end{align}
We remark that in \eqref{piccolo.ansatx} we are again using that the 
constant $\mathtt{c}\ll1$ given by Theorem \ref{teo principale beta plane} is arbitrarily small.
Note also that the second in \eqref{piccolo.ansatx} implies that 
the operator $\bE_{0}$ has small size with respect to $\lambda \gg 1$ large enough.   
Now, given $ \tN_0 > 0$, we fix the constants 
 	\begin{align}
 		\tau &:= \nu + 4 \,, \qquad 
		M :=  \max\{2 \tau, \tfrac{1-\tc}{2(1-\tc)-\alpha} \}\ + 1,  \quad \tau_{1}:= 4\tau +2 +M\,, 
		\\
 		 \ta &:= 3(2 \tau\! + \!M \!+ 1) + 1  \,, \qquad  \tb:=\! \ta + 1 \,, \label{definizione.param.KAM} 
% 		\nonumber 
 		 \\
 	     \Sigma(\tb)&:= \sigma_{M} +\tb\,, \quad S > s_0 + \Sigma(\tb)\,, 
% 		 \nonumber 
 		 \\  
 		\tN_{- 1}& := 1 \,, \quad 
 		\tN_n := \tN_0^{\chi^n}\,, \quad 
 		n \geq 0\,, \quad 
 		\chi := 3/2  \,, 
 	\end{align}
 where $\sigma_{M }$ is given in Proposition \ref{prop normal form lower orders}. 
 Moreover, 
 by Proposition \ref{prop normal form lower orders}, 
 replacing $s$ by $s + \tb$ in  \eqref{stime induttive cal EM cal ZM finale} 
 and having $ \bZ_{0} = {\rm diag}_{j \in \Z^2 \setminus \{ 0 \}} \tz_{0}(j)$  diagonal, 
 and by the ansatz \eqref{ansatz}, one gets the initialization 
 conditions for the KAM reducibility, for any $s_0\leq s\leq S-\Sigma(\tb)$,
 \begin{equation}\label{stime pre rid}
 \begin{aligned}
 & 	\sup_{j \in \Z^2 \setminus \{ 0 \}} |j| |\tz_{0}(j)|^{{\rm Lip}(\gamma)} 
 \lesssim_S  
 \lambda^{\theta}\,, 
 \qquad     
 |\bE_{0}|_{- M , s + \tb}^{{\rm Lip}(\gamma)} 
 \lesssim_{S}   \lambda^{M(\theta-1)+1} \gamma^{-(M-1)}  \,. 
	\end{aligned}
 \end{equation}
 By the definition of $M$ in \eqref{definizione.param.KAM}, one has that 
 $M>\frac{1-\tc}{2(1-\tc)-\alpha}> \frac{1}{1-\theta}$. 
 We work in the regime $\lambda\gg 1$ and recalling \eqref{theta.def.ridu}, \eqref{small.theta}, 
 we define the small parameters 
 \begin{equation}\label{vare.vare}
 \varepsilon:= \lambda^{\theta-1}\,, 
 \qquad 
 \varepsilon_{M} := \lambda^{M(\theta-1)+1} = \varepsilon^{M-1} \lambda^{\theta} \,.
 \end{equation}
 
 \begin{prop}[\bf Reducibility]\label{prop riducibilita}
 	%Let $\gamma \in (0, 1)$, $\tau > 0$, $s > s_0$. 
 	Let $S > s_0 + \Sigma(\tb)$, with the notation of \eqref{definizione.param.KAM}. 
	There exist $\tN_{0} := \tN_{0}(S, \tau,\nu) > 0$ large enough and 
	$\delta := \delta(S, \tau,\nu) \in (0, 1)$ small enough such that, 
	if \eqref{ansatz} holds with $\sigma = \Sigma(\tb)$ and
 	\begin{equation}\label{KAM smallness condition}
	\tN_{0}^{\tau_1}  \varepsilon^{M}\gamma^{-M} \leq \delta \,,
\end{equation} 
 	then the following statements hold for any integer $n \geq 0$.
 	
 	\smallskip
 	\noindent 
 	${\bf (S1)}_n$ There exists a real reversible and  momentum preserving  operator 
 \begin{align}
 \bL_{n}(\vf) &:=  \bD_{n} + \bE_{n}(\vf)  \,, 
\\
  \bD_{n} &:= \beta\, \tL + \bZ_{n}=  {\rm diag}_{j \in \Z^2 \setminus \{ 0 \}} \mu_{n}(j) \,, \label{def.calLn calDn calQn}
  \\ 
\bZ_{n} &:=  {\rm diag}_{j \in \Z^2 \setminus \{ 0 \}} \tz_{n}(j)\,, 
\qquad \mu_{n}(j) :=  \im\, \beta \tL(j)+ \tz_{n}(j)\,,
 \end{align}
 	defined for any $\omega \in \Omega_{n}^\gamma$, where we define {$\Omega_{0}^\gamma := \tD\tC(2\gamma, \tau)$} for $n=0$ and, for $n \geq 1$, 
 		\begin{align}
 			%\Omega_0^\gamma := DC(\gamma, \tau) & \quad \text{and for} \quad  n \geq 1\,, \\
 \Omega_{n}^\gamma & := \Big\{ \omega \in \Omega_{n - 1}^\gamma \,: \, 
 |\im \,\lambda\,\omega \cdot \ell + \mu_{n - 1}(j) - \mu_{n - 1}(j') | 
 \geq 
 \frac{\lambda\, \gamma }{\langle \ell \rangle^\tau  |j'|^\tau}\,, \label{insiemi di cantor rid}
 \\
 & \qquad\qquad \ \forall \,\ell \in \Z^\nu \setminus \{ 0 \} \,, \ \ 
j,j' \in \Z^2 \setminus \{ 0 \}\,, \ \  \pi^\top( \ell) + j-j' =0,  \ \ |\ell|\leq \tN_{n-1} \Big\}\,. 
 		\end{align}
 	%\begin{equation} \label{cal Nn rid}
 	%\begin{aligned}
 	%& {\mathcal D}_n := \zeta \cdot \nabla + {\mathcal Q}_n = {\rm diag}_{j \in \Z^2 \setminus \{ 0 \}} {\mathcal U}_n(j) \,, \quad  
 	%{\mathcal Q}_n = {\rm diag}_{j \in \Z^2 \setminus \{ 0 \}} q_n(j)\,, \quad {\mathcal U}_n(j) := \ii \zeta \cdot j + q_n(j)\,,
 	%\end{aligned}
 	%\end{equation}
For any $j \in \Z^2 \setminus \{ 0 \}$, the eigenvalues $\mu_{n}(j)= \mu_{n}(j;\omega)$ 
%	are purely imaginary and satisfy the conditions
% 	\begin{equation}\label{reversibility reality auto}
% 		\begin{aligned}
% 			& \mu_n(j) = - \mu_n(- j) = \overline{\mu_n(- j)}, \ \ \text{or equivalently}   \\
% 			& \tz_{n}(j) = - \tz_{n}(- j)= \overline{\tz_{n}(- j)}\,, 
% 			%\\ & \text{implying that} \quad {\mathcal U}_n(j), q_n(j) \in \ii \R
% 		\end{aligned}
% 	\end{equation}
% 	and for any $j \in \Z^2 \setminus \{ 0 \}$, 
 satisfy the estimates
\begin{align}
  |\tz_{n}(j)|^{{\rm Lip}(\gamma)} 
 &\lesssim_S 
 \lambda^{\theta} |j|^{-1}\,, 
 \qquad 
 | \tz_{n}(j)- \tz_{0}(j) |^{{\rm Lip}(\gamma)} 
 \lesssim_S 
\varepsilon_{M} \gamma^{-(M-1)}|j|^{- M} \,,  \label{stime qn} 
\\
 | \tz_{n}(j) - \tz_{n - 1}(j)|^{{\rm Lip}(\gamma)}
 &\lesssim_S  
 \varepsilon_{M} \gamma^{-(M-1)} \, \tN_{n - 2}^{- \ta}\, |j|^{- M}  \,,
 \qquad \ \text{when } \qquad n\geq 1 \,,\label{cal Nn - N n - 1}
 	\end{align}
     they are purely imaginary and they satisfy
 		\begin{align}
 \mu_n(j) &= - \mu_n(- j) = - \overline{\mu_n( j)} \,, \ \ \text{or equivalently}   \label{reversibility reality auto}
 \\
  \tz_{n}(j) &= - \tz_{n}(- j)= - \overline{\tz_{n}( j)}\,, 
 %\\ & \text{implying that} \quad {\mathcal U}_n(j), q_n(j) \in \ii \R
 		\end{align}
 	The operator $\bE_{n}$ is real, reversible and momentum preserving, 
	satisfying, for any $s_0\leq s\leq S-\Sigma(\tb)$,
 	\begin{equation}\label{stime cal Rn rid}
 		\begin{aligned}
 	& |\bE_{n} |_{- M, s}^{{\rm Lip}(\gamma)} 
 \leq C_*(S)  \varepsilon_{M} \gamma^{-(M-1)}  \, \tN_{n - 1}^{- \ta}  \,, 
	\qquad  |\bE_{n}  |_{-M, s + \tb}^{{\rm Lip}(\gamma)}  
	\leq 
	C_*(S) \varepsilon_{M} \gamma^{-(M-1)}   \, \tN_{n - 1} \,,
 		\end{aligned}
 	\end{equation}
 	for some constant $C_* (S) > 0$ large enough. 
 	
	\noindent
 	When $n \geq 1$, there exists an invertible, real, reversibility preserving 
	and momentum preserving map 
 	$\Phi_{n -1} = {\rm exp}(\Psi_{n - 1})$, such that, 
 	for any $\omega \in \Omega_{n}^\gamma$, 
 	\begin{equation}\label{coniugazione rid}
 		\bL_{n}(\vf) = (\Phi_{n - 1})_* \bL_{n-1}(\vf)\,, \quad \vf \in \T^\nu\,. 
 	\end{equation}
 	Moreover, for any $s_0\leq s\leq S-\Sigma(\tb)$, the map 
	$\Psi_{n - 1} : H^s_0 \to H^s_0$ satisfies
 	\begin{equation}\label{stime Psi n rid}
 		\begin{aligned}
 	 |\Psi_{n - 1}|_{-1, s}^{{\rm Lip}(\gamma)}
 & \lesssim_S
 \varepsilon^{M} \gamma^{- M} \tN_{n - 1}^{2 \tau + 1} \tN_{n - 2}^{- \ta}\,, 
 \qquad 
 |\Psi_{n - 1} |_{- 1, s + \tb}^{{\rm Lip}(\gamma)} 
  \lesssim_S \varepsilon^{M} \gamma^{- M} \tN_{n - 1}^{2 \tau + 1}  \tN_{n - 2} \,.
 		\end{aligned}
 	\end{equation}

    \noindent
 	${\bf (S2)}_n$ For all $ j \in \Z^2 \setminus \{ 0 \}$,  there exists a 
	Lipschitz extension  of the eigenvalues 
	$\mu_n(j;\,\cdot\,) :\Omega_{n}^\gamma \to  \im \R$ to the set 
	{$\tD\tC(2\gamma, \tau)$}, denoted by
 	$ \widetilde \mu_n(j;\,\cdot\,): {\tD\tC(2\gamma, \tau) }\to \im \R$,
 	satisfying,  for $n \geq 1$, 
 	\begin{equation}\label{lambdaestesi}  
 		|\widetilde \mu_n(j) -  \widetilde \mu_{n - 1}(j) |^{{\rm Lip}(\gamma)}  
		\lesssim  |j|^{- M} |\bE_{n - 1}|_{-M,s_0}^{{\rm Lip}(\gamma)}  
		\lesssim_S  \varepsilon_{M} \gamma^{-(M-1)} \, \tN_{n - 2}^{- \ta}\, |j|^{- M}\,.
 	\end{equation}

 \end{prop}
 
 \begin{proof}
 	{\sc Proof of ${\bf (S1)}_0-{\bf (S2)}_0$.} 
 	The claimed properties follow directly 
 	from Proposition \ref{prop normal form lower orders}, recalling \eqref{bL0_inizioKAM}, \eqref{stime pre rid}, \eqref{vare.vare} 
 	and the definition of {$\Omega_{0}^\gamma := \tD\tC(2\gamma, \tau)$}. 
 	
 	\smallskip
 	
 	\noindent
	By induction, we assume the the claimed properties ${\bf (S1)}_n$-${\bf (S4)}_n$ hold for some $n \geq 0$ and we prove them at the step $n + 1$. 
	
	\noindent
 	{\sc Proof of ${\bf (S1)}_{n+1}$.} 
 	Let $\Phi_n = {\rm exp}( \Psi_n)$ where $\Psi_n$ is an operator to be determined. 
% which has to be determined. 
 	By the Lie expansion, we compute
 		\begin{align}
 			\bL_{n + 1} & = (\Phi_n)_*\bL_{n}  =  \bD_n - \lambda\, \omega \cdot \partial_\vphi \Psi_n 
			+ [\bD_n , \Psi_n] + \Pi_{\tN_n} \bE_n + \bE_{n + 1} \,,  
			\\
\bE_{n + 1} & := \Pi_{\tN_n}^\bot \bE_n   
+ \int_0^1 (1 - \tau) {\rm exp}(- \tau \Psi_n) \,[- \lambda \,\omega \cdot \partial_\vphi \Psi_n 
+ [\bD_n , \Psi_n] , \Psi_n] \, {\rm exp}(\tau \Psi_n) \wrt \tau \label{primo coniugio Ln Psin}
\\
 & \quad + \int_0^1 {\rm exp}(- \tau \Psi_n) [\bE_n , \Psi_n] {\rm exp}(\tau \Psi_n) \wrt \tau \,,  
 		\end{align}
 	where $\bD_{n}:=  \beta \,\tL + \bZ_{n}$ and the projectors 
	$\Pi_{N}$, $\Pi_{N}^\bot$ are defined in  \eqref{def proiettore operatori matrici}. 
 	Our purpose is to find a map $\Psi_n$ solving the {\it homological equation} 
 	\begin{equation}\label{equazione omologica KAM}
 - \lambda\,\omega \cdot \partial_\vphi \Psi_n + [\bD_{n} , \Psi_n] + \Pi_{\tN_{n}} \bE_{n} = \widehat{\bE}_n(0) \,,
 	\end{equation}
 	where $\widehat{\bE}_n(0) := \frac{1}{(2 \pi)^\nu} \int_{\T^\nu} \bE_n(\vphi) \wrt \vphi$ is a diagonal operator by Lemma \ref{lem:mom_pres}, since $\bE_{n}(\vf)$ is a momentum preserving operator by induction assumption.
 	 % is the diagonal operator as per Definition \ref{def block-diagonal op}.
 	%$3 \times 3$, $\vphi$-independent, block-diagonal operator defined by 
 	%\begin{equation}\label{def cal Zn}
 	%{\mathcal D}_{{\mathcal R}_n} := {\rm diag}_{j \in \Z^2 \setminus \{ 0 \}} \widehat{\mathcal R}_n(0)_j^j\,. 
 	%\end{equation}
% 	Note that since $\bE_n$ is a momentum preserving operator, one has that 
%	$$
%	\widehat{\bE}_n(\ell)_j^{j'} \neq 0 \Longrightarrow \pi^T(\ell) + j - j' = 0
%	$$ and hence in particulat for $\ell = 0$, 
%	$$
%	\widehat{\bE}_n(0)_j^{j'} \neq 0 \Longrightarrow j = j'\,,
%	$$ 
%	namely $\widehat{\bE}_n(0)$ is a diagonal operator. 
	By  \eqref{matrix representation 1} and \eqref{def.calLn calDn calQn}, 
	the homological equation \eqref{equazione omologica KAM} is equivalent to 
	\begin{align}
 	&	\big( - \im \,\lambda\, \,\omega \cdot \ell + \mu_{n}(j) - \mu_{n}(j') \big) \widehat \Psi_{n}(\ell)_j^{j'} 
	+ \widehat{\bE_{n}}(\ell)_j^{j'} = 0\,, \label{eq omologica matrici}
	\\
	&   \ell \in \Z^\nu\setminus\{0\}\,, \ \   |\ell| \leq \tN_n\,,\, \ \ j, j' \in \Z^2 \setminus \{ 0 \}\,, \ \ \pi^\top(\ell) + j - j' = 0\,. 
		\end{align}
 	 Therefore, we define the linear operator $\Psi_{n}$ by
 	\begin{equation}\label{def Psi eq omo KAM}
 		%\begin{footnotesize}
 \widehat\Psi_{n} (\ell)_j^{j'} := \begin{cases}
 \dfrac{\widehat{\bE_{n}}(\ell)_j^{j'} }{ \im \,\lambda\, \omega \cdot \ell + \mu_{n}(j') - \mu_{n}(j)  }, 
 & \begin{matrix}
 \ell \in \Z^{\nu}\setminus\{0\} \,, \quad j,j'\in \Z^2\setminus\{0\}\,, \\  |\ell|  \leq \tN_{n}\,, 
 \quad \pi^\top(\ell) + j-j'=0\,,
 \end{matrix} \\
 	0 & \text{otherwise}\,. 
 			\end{cases}
 		%\end{footnotesize}
 	\end{equation}
 	which  is a solution of \eqref{equazione omologica KAM}-\eqref{eq omologica matrici}. 
 	
 	\begin{lem}\label{Lemma eq omologica riducibilita KAM}
 The operator $\Psi_n$ in \eqref{def Psi eq omo KAM}, 
 defined for any $\omega \in \Omega_{n + 1}^\gamma$, satisfies, 
 for any $s_0\leq s\leq S-\Sigma(\tb)$,
 \begin{align}
 |\Psi_n|_{- 1, s}^{{\rm Lip}(\gamma)}
 &\lesssim 
 \tN_{n}^{2 \tau + 1} \lambda^{- 1} \gamma^{- 1} |\bE_{n}|_{- M, s}^{{\rm Lip}(\gamma)}\,, \label{stime eq omologica}
 \\
|\Psi_n|_{- 1, s + \eta}^{{\rm Lip}(\gamma)} 
&\lesssim
\tN_{n}^{2 \tau + 1 + \eta} \lambda^{- 1}\gamma^{- 1} |\bE_{n}|_{- M, s}^{\Lip(\gamma)}\,, 
\qquad \forall \, \eta > 0 \,.
 \end{align}
Moreover, $\Psi_n$ is real, reversibility preserving and momentum preserving. 
\end{lem}
 	
 	\begin{proof}
 		To simplify notations, along this proof 
 		we drop the index $n$.
		
		\medskip
		
\noindent
{\bf Proof of \eqref{stime eq omologica}.}
 %and we write $+$ instead of $n + 1$ when the notation is clear. 
 Let $\wh\Psi(\ell)_{j}^{j'}=\wh\Psi(\ell;\omega)_{j}^{j'}$ as in 
 \eqref{def Psi eq omo KAM}, with $ \ell \in \Z^\nu$, $ j, j' \in \Z^2 \setminus \{ 0 \}$, 
 with $0 < |\ell| \leq \tN$ and $ \pi^\top(\ell) + j - j' = 0$. 
 For any $\omega \in \Omega_{n + 1}^\gamma$ (see \eqref{insiemi di cantor rid}), 
 we immediately get the estimate
 \begin{equation}\label{stima eq omo KAM 1}
 |\widehat\Psi (\ell;\omega)_j^{j'}| 
 \lesssim
 \lambda^{- 1}\gamma^{- 1} \langle \ell \rangle^\tau | j' |^\tau |\wh\bE(\ell;\omega)_j^{j'}| 
 \lesssim 
 \lambda^{- 1} \gamma^{- 1}  \tN^\tau  | j' |^\tau |\wh\bE (\ell;\omega)_j^{j'}| \,.
\end{equation}
 We define 
 $\delta_{\ell j j'}(\omega) :=\im\,\lambda\, \omega \cdot \ell 
 +\mu(j';\omega) - \mu(j;\omega)$. 
 Let $\omega_{1}, \omega_{2}  \in \Omega_{n + 1}^\gamma$. 
 By \eqref{def.calLn calDn calQn}, \eqref{insiemi di cantor rid}, \eqref{stime qn},  we have  
\begin{align}
\big| \big( \mu(j;\omega_1) - \mu(j';\omega_1) \big) 
&- \big( \mu(j;\omega_2) - \mu(j';\omega_2)\big) \big|  
\\
&\leq \big|\tz(j; \omega_1) - \tz(j; \omega_2)\big| + \big|\tz(j'; \omega_1) - \tz(j'; \omega_2)\big| 
\\
& \lesssim \lambda^{\theta} \gamma^{- 1} \big( |j|^{-1} + |j'|^{-1} \big) |\omega_1 - \omega_2| 
\lesssim  \lambda^{\theta} \gamma^{- 1}  |\omega_1 - \omega_2| \,,
\end{align}
and therefore, using that $\lambda^{\theta-1}\gamma^{-1}\ll 1 $ 
by \eqref{piccolo.ansatx},
\begin{align}
|\delta_{\ell j j'}(\omega_1) - \delta_{\ell j j'}(\omega_2)| 
& \lesssim 
(\lambda\, |\ell|+\lambda^{\theta}\gamma^{-1}) |\omega_1 - \omega_2|  
\lesssim 
\lambda \, \langle \ell \rangle |\omega_1 - \omega_2|\,.\label{calcetto1}
\end{align} 
Then, estimate \eqref{calcetto1}, together with the fact that 
$\omega_1, \omega_2 \in \Omega_{n + 1}^\gamma$, 
implies that
 \begin{align}
 \Big| \frac{1}{ \delta_{\ell j j'}(\omega_{1})} - \frac{1}{\delta_{\ell j j'}(\omega_{2})} \Big| \label{calcetto2}
 & \leq 
 \dfrac{|\delta_{\ell j j'}(\omega_{1}) - \delta_{\ell j j'}(\omega_{2})|}{|\delta_{\ell j j'}(\omega_{1})| 
 |\delta_{\ell j j'}(\omega_{2})|} 
% \\
% & \lesssim 
% \lambda (\lambda\gamma)^{-2} \langle \ell \rangle^{2 \tau + 1}  |j'|^{2 \tau} |\omega_{1} - \omega_{2} | 
 \\
&  \lesssim 
\lambda^{- 1} \gamma^{- 2} \langle \ell \rangle^{2 \tau + 1}  |j'|^{2 \tau} |\omega_{1} - \omega_{2} |\,. 
 \end{align}
 Therefore, by \eqref{def Psi eq omo KAM} and \eqref{calcetto2},  
 for any $\omega_{1},\omega_{2}\in \Omega_{n + 1}^\gamma$ we have that
 \begin{align}
 \big|\widehat\Psi (\ell;\omega_{1})_j^{j'} - \widehat\Psi (\ell;\omega_{2})_j^{j'}\big| 
 & \lesssim  
 \langle \ell \rangle^\tau  |j'|^\tau \lambda^{- 1} \gamma^{- 1}
  \big|\widehat\bE(\ell;\omega_{1})_j^{j'} - \widehat\bE(\ell;\omega_{2})_j^{j'}\big| 
  \\
 & \ \ 
 + \langle \ell \rangle^{2 \tau + 1}  |j'|^{2 \tau} \lambda^{- 1} \gamma^{-2}  
 \big|\widehat{\bE}(\ell;\omega_{2})_j^{j'}\big| |\omega_{1}-\omega_{2} | 
\label{Psi lam 12 KAM} \\
& \lesssim 
\tN^\tau  |j'|^\tau \lambda^{- 1} \gamma^{- 1} 
\big|\widehat\bE(\ell;\omega_{1})_j^{j'} - \widehat\bE(\ell;\omega_{2})_j^{j'}\big| 
\\
& \ \ 
+ \tN^{2 \tau + 1}  |j'|^{2 \tau} \lambda^{- 1} \gamma^{-2}  
\big|\widehat{\bE}(\ell;\omega_{2})_j^{j'}\big| |\omega_{1}-\omega_{2} |\,. 
\end{align}
% 		Using that $  \langle \ell, j - j' \rangle \leq \tN $  and the elementary chain of inequalities $ |j| \lesssim |j - j'| + |j'| \lesssim \tN + |j'| \lesssim \tN |j'|$,
% 		%\begin{equation}\label{dis elementari eq omo KAM}
% 		%\begin{aligned}
% 		%& \langle \ell, j - j' \rangle \leq N \quad \text{and the elementary chain of inequalities} \\
% 		%& |j| \lesssim |j - j'| + |j'| \lesssim N + |j'| \lesssim N |j'|, 
% 		%\end{aligned}
% 		%\end{equation}
% 		the estimates \eqref{stima eq omo KAM 1}, \eqref{Psi lam 12 KAM} take the form 
% 		\begin{equation*}
% 			\begin{aligned}
% 				|\widehat\Psi (\ell)_j^{j'}| & \lesssim  \tN^{2 \tau}\gamma^{- 1} |j'|^{2 \tau}  |{\bE}(\ell)_j^{j'}|\,, \\
% 				|\widehat\Psi (\ell)_j^{j'}(\omega_{1}) - \widehat\Psi (\ell)_j^{j'}(\omega_{2})| & \lesssim \tN^{2 \tau} \gamma^{- 1}  |j'|^{2 \tau} |\widehat{\bE}(\ell)_j^{j'}(\omega_{1}) - \widehat{\bE}(\ell)_j^{j'}(\omega_{2})| \\
% 				& \qquad + \tN^{4 \tau + 2}\gamma^{- 2}  |j'|^{ 4 \tau} |\widehat{\bE}(\ell)_j^{j'}(\omega_{2})| |\omega_{1} - \omega_{2} |\,. 
% 			\end{aligned}
% 		\end{equation*}
 Since $M  \geq 2 \tau + 1$ by \eqref{definizione.param.KAM}, recalling 
 Definition \ref{block norm} and 
 collecting the estimates \eqref{stima eq omo KAM 1}, \eqref{Psi lam 12 KAM}, 
 we obtain the bounds, for any $s\in [s_0,S-\Sigma(\tb)]$,
 \begin{align}
 |\Psi|_{- 1, s}^{\rm sup} 
 &\lesssim 
 \tN^{ \tau} \lambda^{- 1}\gamma^{- 1} |{\bE}|_{- M, s}^{\rm sup}\,, 
 \\
|\Psi|_{- 1, s}^{\rm lip} 
&\lesssim  
\tN^{2 \tau + 1} \lambda^{- 1}\gamma^{- 2}  |{\bE}|_{- M, s}^{\rm sup} 
+ \tN^{\tau} \lambda^{- 1} \gamma^{- 1}  |{\bE}|_{- M, s}^{\rm lip}\,.
 \end{align}		
 Similarly, using also that 
 $|\ell| \leq \tN$ and $\pi^\top(\ell) + j - j' = 0$ 
 imply that $|j - j'| \lesssim |\ell| \lesssim \tN$, 
 by analogous arguments we obtain that, for any $\eta > 0$,
 \begin{align}
 |\Psi|_{- 1, s + \eta}^{\rm sup} 
&\lesssim 
\tN^{\tau + \eta} \gamma^{- 1} \lambda^{- 1} |{\bE}|_{- M, s}^{\rm sup}\,, 
\\
 |\Psi|_{- 1, s+ \eta}^{\rm lip} 
 &\lesssim  
 \tN^{2 \tau + \eta +1} \gamma^{- 2} \lambda^{- 1}  |{\bE}|_{- M, s}^{\rm sup} 
 + \tN^{\tau + \eta} \lambda^{- 1} \gamma^{- 2} |{\bE}|_{- M, s}^{\rm sup}\,. 
 \end{align}
 Hence, we conclude the claimed bounds in \eqref{stime eq omologica}. 
Finally, since ${\bE}$ is real, reversible and momentum preserving, 
by \eqref{def Psi eq omo KAM}, Lemma \ref{lemma real rev matrici} 
and Lemma \ref{lem:mom_pres} we have that $\Psi$ is real, reversibility preserving 
and momentum preserving.  This concludes the proof.
		% Lemma \ref{lemma real rev matrici} and the properties \eqref{reversibility reality auto} for $\mu(j)$, we deduce that $\Psi$ is real and reversibility preserving. 
 \end{proof}

 \noindent	
By Lemma \ref{Lemma eq omologica riducibilita KAM}, 
the induction assumption on the estimates \eqref{stime cal Rn rid} 
and by \eqref{vare.vare}, we obtain, for any $s_0\leq s\leq S-\Sigma(\tb)$,
 		\begin{align}
 			|\Psi_n|_{- 1, s}^{{\rm Lip}(\gamma)} 
 			& \lesssim  \tN_n^{2 \tau + 1} \lambda^{- 1}\gamma^{- 1} |{\bE}_n |_{- M, s}^{{\rm Lip}(\gamma)}
 			\lesssim_S \tN_n^{2 \tau + 1} \tN_{n - 1}^{- \ta} 
 			\varepsilon^M \gamma^{- M}  \,,  \label{stime Psin neumann} \\
 			|\Psi_n|_{- 1, s + \tb}^{{\rm Lip}(\gamma)} & \lesssim \tN_n^{2 \tau + 1} \lambda^{- 1} \gamma^{- 1} |{\bE}_n |_{-M, s + \tb}^{{\rm Lip}(\gamma)} \lesssim_{S } \tN_n^{2 \tau + 1} \tN_{n - 1} \varepsilon^{M}\gamma^{- M}\,,
 		\end{align}
 	which are the estimates \eqref{stime Psi n rid} at the step $n + 1$. 
 	Moreover, setting $\eta=M$ in \eqref{stime eq omologica}, by the same arguments we also have
 		\begin{align}
 			|\Psi_n|_{- 1, s + M}^{{\rm Lip}(\gamma)} \,,\, 
 			& \lesssim  \tN_n^{2 \tau + 1 + M} \gamma^{- 1} \lambda^{- 1} |{\bE}_n |_{- M, s}^{{\rm Lip}(\gamma)} \lesssim_S \tN_n^{2 \tau + 1 + M} \tN_{n - 1}^{- \ta} 
 			\varepsilon^M \gamma^{- M} \,, \label{stime.Psin.neumann2} \\
 			|\Psi_n|_{- 1, s + \tb + M}^{{\rm Lip}(\gamma)} & \lesssim \tN_n^{2 \tau + 1 + M} \lambda^{- 1}\gamma^{- 1} |{\bE}_n |_{-M, s + \tb}^{{\rm Lip}(\gamma)}  \lesssim_{S} \tN_n^{2 \tau + M + 1} \tN_{n - 1} \varepsilon^M \gamma^{- M} \,.
 		\end{align}
 	In particular, by \eqref{definizione.param.KAM}, \eqref{ansatz} 
	and by the smallness condition 
	\eqref{KAM smallness condition}, we deduce, for $\tN_0>0$ large enough,
 \begin{equation} \label{2103.1}
\begin{aligned}
 |\Psi_n|_{- 1, s + M}^{{\rm Lip}(\gamma)}  
 & \lesssim_S
 \tN_n^{2 \tau +1 + M} \tN_{n - 1}^{- \ta} \varepsilon^M \gamma^{- M}   \leq \delta < 1\,.
\end{aligned}
\end{equation}
Therefore, by Lemma \ref{proprieta standard norma decay}-$(iii)$ 
and estimates \eqref{stime.Psin.neumann2}, \eqref{2103.1}, 
we have the estimates, for any $s_0 \leq s \leq S - \Sigma(\tb)$,
\begin{align}
 \sup_{\tau \in [- 1, 1]}|{\rm exp}(\tau \Psi_n)|_{0, s  + M}^{\Lip(\gamma)} 
&\lesssim_S 1 + |\Psi_n|_{- 1, s + M}^{\Lip(\gamma)} \lesssim_S 1\,,  \label{stime Phi n inv - Id} 
\\
 \sup_{\tau \in [- 1, 1]}|{\rm exp}(\tau \Psi_n)|_{0, s  + \tb +  M}^{\Lip(\gamma)} 
&\lesssim_S 1 + |\Psi_n|_{- 1, s + \tb + M}^{\Lip(\gamma)}\,. 
\end{align}
 	Then, by recalling \eqref{primo coniugio Ln Psin} and by using that $\Psi_n$ 
	solves the equation \eqref{equazione omologica KAM}, we conclude that
 				\begin{align}
 				{\bL}_{n + 1} & \,= {\bD}_{n + 1} + {\bE}_{n + 1}\,,  \\
 				{\bD}_{n + 1} & :=  \beta \tL + {\bZ}_{n + 1}\,, \qquad {\bZ}_{n + 1}  : = {\bZ}_n + \widehat{\bE}_n(0)\,,   \label{2003.1} \\
 				{\bE}_{n + 1}& \,=\Pi_{\tN_n}^\bot \bE_n   + \int_0^1 (1 - \tau) {\rm exp}(- \tau \Psi_n) \,[\widehat{\bE}_n(0) - \Pi_{\tN_n} \bE_n , \Psi_n] \, {\rm exp}(\tau \Psi_n)\wrt  \tau  \\
 				& \quad + \int_0^1 {\rm exp}(- \tau \Psi_n) [\bE_n , \Psi_n] {\rm exp}(\tau \Psi_n) \wrt \tau\,. 
 			\end{align}
 	All the operators in \eqref{2003.1}
 	%${\mathcal L}_{n+1}$, $\mN_{n+1}$, $\mQ_{n+1}$, $\mR_{n+1}$ 
 	are defined for any $\omega \in \Omega_{n + 1}^\gamma$. 
 	Moreover, by \eqref{primo coniugio Ln Psin}, \eqref{equazione omologica KAM},
 	for $\omega \in \Omega^\gamma_{n+1}$ one has the %conjugation %identity 
 	identity
 	$\bL_{n+1} = (\Phi_n)_* \bL_n$, 
 	which is \eqref{coniugazione rid} at the step $n+1$. 
% 	By Definition \ref{def block-diagonal op} applied to ${\mathcal D}_{{\bE}_n} = \wh\bE_{n}(0)$ and by \eqref{2003.1}, we have that
 	Recalling that $\wh\bE_{n}(0)$ is a diagonal operator and by \eqref{2003.1}, we have that
% 	\begin{equation}\label{frittata di maccheroni}
% 		\begin{aligned}
% 			{\bZ}_{n + 1} & := {\bZ}_n + \widehat{\bf E}_n(0) = {\rm diag}_{j \in \Z^2 \setminus \{ 0 \}} \tz_{n + 1}(j)\,, \\
%	\tz_{n + 1}(j) & := \tz_n(j) + \widehat{\bE}_n(0)_j^j\,, \\
% 			{\bD}_{n + 1} &:= -\beta\,\tL+ {\bZ}_{n + 1} = {\rm diag}_{j \in \Z^2 \setminus \{ 0 \}} \mu_{n + 1}(j)\,, \\
% 			\mu_{n + 1}(j) & := - \im \, \beta\, \tL(j)+ \tz_{n + 1}(j)\,.
% 		\end{aligned}
% 	\end{equation}
	\begin{align}
		{\bZ}_{n + 1} & = {\rm diag}_{j \in \Z^2 \setminus \{ 0 \}} \tz_{n + 1}(j)\,, \quad	\tz_{n + 1}(j)  := \tz_n(j) + \widehat{\bE}_n(0)_j^j\,,  \label{frittata di maccheroni} \\
		{\bD}_{n + 1} &= {\rm diag}_{j \in \Z^2 \setminus \{ 0 \}} \mu_{n + 1}(j)\,, \quad	\mu_{n + 1}(j) :=  \im \, \beta\, \tL(j)+ \tz_{n + 1}(j)\,.
	\end{align}
 	By Lemma \ref{proprieta standard norma decay}-$(iv)$ one gets
 		\begin{align}
 			| \mu_{n + 1}(j) - \mu_n(j)|^{{\rm Lip}(\gamma)} & = 
			| \tz_{n + 1}(j) - \tz_n(j)|^{{\rm Lip}(\gamma)}   \label{frittata.1}
			\\
 			& \leq | \widehat{\bE}_n(0)_j^j |^{{\rm Lip}(\gamma)} 
			\lesssim |{\bE}_n|_{- M, s_0}^{{\rm Lip}(\gamma)} \langle j \rangle^{- M}\,.
 		\end{align}
 	Then, \eqref{frittata.1}, together with the estimate \eqref{stime cal Rn rid},
 	%, together with the ansatz \eqref{ansatz} (with ${\mathcal U} = {\mathcal U}(\mathtt b)$) imply 
 	implies  \eqref{cal Nn - N n - 1} at the step $n + 1$. 
 	The estimate \eqref{stime qn} at the step $n + 1$ follows, as usual, 
 	by a telescoping argument, 
 	using the fact that $\sum_{n \geq 0} \tN_{n - 1}^{- \ta}$ is convergent 
 	since $\ta > 0$ (see \eqref{definizione.param.KAM}). 
 	Now we prove the estimates \eqref{stime cal Rn rid} at the step $n + 1$. 
 	By \eqref{2003.1}, estimates 
 	\eqref{stime Psin neumann}, 
 	\eqref{2103.1},
 	\eqref{stime Phi n inv - Id}, the induction estimates \eqref{stime cal Rn rid},
 	Lemma \ref{proprieta standard norma decay}-$(i)$ and 
 	Lemma \ref{lemma proiettori decadimento}, 
 	we get, for any $s_0\leq s \leq S-\Sigma(\tb)$,
 		%	\label{stime R n + 1 R n}
% 		\begin{footnotesize}
 			\begin{align}
 				& |{\bE}_{n + 1}|_{- M, s}^{{\rm Lip}(\gamma)} 
 				\lesssim_{S} 
				\tN_n^{- \tb} |{\bE}_n|_{- M, s + \tb}^{{\rm Lip}(\gamma)} 
		+ \tN_n^{2 \tau + 1 + M}\lambda^{- 1} \gamma^{- 1}  \Big(|{\bE}_n|_{- M, s}^{{\rm Lip}(\gamma)} \Big)^2  \,, \label{thuram}
		\\
 		& |{\bE}_{n + 1}|_{- M, s + \tb}^{{\rm Lip}(\gamma)} 
 	\lesssim_{S} |{\bE}_n|_{- M, s +\tb}^{{\rm Lip}(\gamma)}  
	+ \tN_{n}^{2 \tau + 1 + M} \lambda^{- 1}\gamma^{-1} |{\bE}_n|_{- M, s}^{{\rm Lip}(\gamma)} 
	|{\bE}_n|_{- M, s+\tb}^{{\rm Lip}(\gamma)}  \,. 
 				%\big( 1 + N_n^{2 \tau_0} \gamma^{- 1} 
 				%|{\mathcal R}_n \langle D \rangle^M|_{s_0}^{{\rm Lip}(\gamma)} \big)\,. 
 			\end{align}
% 		\end{footnotesize}
 	The estimates \eqref{stime cal Rn rid} at the step $n + 1$ follow by \eqref{thuram}, the induction assumption on the estimate \eqref{stime cal Rn rid}, \eqref{vare.vare} and \eqref{definizione.param.KAM}, \eqref{KAM smallness condition}, taking $\tN_0, \lambda  \gg 0$ large enough. 
%	\begin{equation}\label{milan merda 0}
%	\begin{aligned}
%	|{\bE}_{n + 1}|_{- M, s +\tb}^{{\rm Lip}(\gamma)} & \leq C(s) C_*(s)\tN_{n - 1} \Big(  1  + C_*(s) \tN_n^{2 \tau + M + 1} \tN_{n - 1}^{- \ta}\lambda^{- 1} \gamma^{- 1} \Big) \| \bar w \|_{s + \Sigma(\tb)}^{\Lip(\gamma)} \\
%	& \leq C_*(s) \tN_n \| \bar w \|_{s + \Sigma(\tb)}^{\Lip(\gamma)}
%	\end{aligned}
%	\end{equation}
	  Finally,	since $\Psi_n, \Phi_n, \Phi_n^{- 1}$ are real and momentum preserving by Lemma \ref{equazione omologica KAM} and ${\bD}_n, {\bE}_n$ are real and momentum preserving by the induction assumption, then ${\bD}_{n + 1}$, ${\bE}_{n + 1}$ are real and momentum preserving operators, by \eqref{2003.1} and Lemmata \ref{lemma real rev matrici}, \ref{lem:mom_prop}, \ref{lem:mom_pres}. Moreover  since $\Psi_n, \Phi_n, \Phi_n^{- 1}$ are reversibility preserving by Lemma \ref{equazione omologica KAM} and ${\bD}_n, {\bE}_n$ are reversible by the induction assumption, then ${\bD}_{n + 1}$, ${\bE}_{n + 1}$ are reversible by \eqref{2003.1} and Lemma \ref{lemma real rev matrici} and hence \eqref{reversibility reality auto} is verified at the step $n + 1$.
 	\\[1mm]
 	\noindent
 	{\sc Proof of ${\bf (S2)}_{n + 1}$.} We now construct a Lipschitz extension for the eigenvalues $\mu_{n + 1}(j,\,\cdot\,) : \Omega_{n + 1}^\gamma \to  \,\im \R$. By the induction hypothesis, there exists a Lipschitz extension of $\mu_n(j;\omega)$,  denoted by $\widetilde \mu_{n}(j;\omega)$, to the whole set $\tD\tC(2\gamma, \tau)$ that satisfies ${\bf (S2)}_n$. By \eqref{frittata di maccheroni}, we have $\mu_{n + 1}(j) = \mu_n(j) + \tr_n(j)$, where $\tr_n(j, \omega)=\tr_n(j,\omega) := \widehat{\bE}_n(0;\omega)_j^j$ satisfies $|\tr_n(j)|^{{\rm Lip}(\gamma)} \lesssim_S  \tN_{n - 1}^{- \ta} \varepsilon_{M} \gamma^{-(M-1)} |j|^{- M}  $. 
	%By the reversibility and the reality of ${\bE}_n$, we have $\tr_n(j) = - \tr_n(- j)=\overline{\tr_n(-j)} $, implying that $\tr_n(j) \in  \im\, \R$. 
	Hence, by the Kirszbraun Theorem (see Lemma M.5 \cite{KukPo}) there exists a Lipschitz extension $\widetilde \tr_n(j,\,\cdot\,) : {\tD\tC(2\gamma, \tau) }\to \im \R$ of $\tr_n(j,\,\cdot\,) : \Omega_{n + 1}^\gamma \to \im \R$ satisfying $|\widetilde \tr_n(j)|^{{\rm Lip}(\gamma)} \lesssim |\tr_n(j)|^{{\rm Lip}(\gamma)} \lesssim_S \tN_{n - 1}^{- \ta} \varepsilon_{M} \gamma^{-(M-1)} |j|^{- M}$. The claimed statement then follows by defining $\widetilde \mu_{n + 1}(j) := \widetilde  \mu_n(j) + \widetilde \tr_n(j)$.

\noindent	
This concludes also the proof of Proposition \ref{prop riducibilita}.
 \end{proof}

 We now conclude by analyzing the convergence of the scheme.
 \begin{lem}\label{lemma blocchi finali}
 	For any $j \in \Z^2 \setminus \{ 0 \}$, the sequence 
 	$\{ \widetilde \mu_n(j) \}_{n\in\N}$, defined in \eqref{def.calLn calDn calQn},
 	converges to some limit
% 	\begin{equation*}\label{def cal N infty nel lemma}
% 		\mu_\infty(j) = -\im \, \beta \,\tL(j) + \tz_\infty(j)\,, \quad \mu_\infty(j)= \mu_\infty(j;\,\cdot\,):\tD\tC(2\gamma,\tau)\to \im\,\R\,,
% 	\end{equation*}
	\begin{equation}\label{def cal N infty nel lemma}
	\mu_\infty(j) = \im \, \beta \,\tL(j) + \tz_\infty(j)\,, \quad \mu_\infty(j)= \mu_\infty(j;\,\cdot\,):{\tD\tC(2\gamma,\tau)}\to \im \R \,,
\end{equation}
 	satisfying the following estimates % the following estimates hold: 
 		\begin{align}
 			& |  \mu_\infty(j) - \widetilde \mu_n(j) |^{{\rm Lip}(\gamma)} = | \tz_\infty(j) - \widetilde \tz_n(j) |^{{\rm Lip}(\gamma)} \lesssim  \tN_{n - 1}^{- \ta} \varepsilon_{M} \gamma^{-(M-1)} |j|^{- M}  \,, \label{stime forma normale limite} \\
 			& | \tz_\infty(j)|^{{\rm Lip}(\gamma)} \lesssim  \lambda^{\theta}  |j|^{- 1} \,. 
 		\end{align}
% 	Let $w_1(\omega)$, $w_2(\omega)$ satisfy the ansatz \eqref{ansatz}. Then, for any $j\in\Z^2\setminus\{0\}$, we have
% 	\begin{equation}\label{auto.infty.delta12}
% 		| \Delta_{12} \mu_{\infty} (j) | = | \Delta_{12} \tz_\infty(j) |   \lesssim \lambda^{\alpha-1} |j|^{-1} \| w_1 -w_2 \|_{s_0+\Sigma(\tb)}\,.
% 	\end{equation}
 	%In addition, if we assume $ w(\vf,x)=\odd(\vf,x)$, then $\mu_\infty(j;\,\cdot\,):{\tD\tC(2\gamma,\tau)\cap\Lambda_{o}}\to \im\,\R$ for any $j \in \Z^2 \setminus \{ 0 \}$.
 	%Moreover, we have  ${\mathcal U}_\infty(j), q_\infty(j) \in \ii \,\R$. 
 	%Moreover the $3 \times 3$ block diagonal operator ${\mathcal Q}_\infty := {\rm diag}_{j \in \Z^3 \setminus \{ 0 \}} ({\mathcal Q}_\infty)_j^j$ is real and reversible. 
 \end{lem}
 
% \begin{proof}
% 	By Proposition \ref{prop riducibilita}, in particular by \eqref{reversibility reality auto} and \eqref{lambdaestesi}, we have that the sequence $\{\widetilde \mu_n(j;\omega)\}_{n\in\N}\subset \im\,\R$ is Cauchy on the closed set $\tD\tC(2\gamma,\tau)$,  therefore it is convergent for any $ \omega \in \tD\tC(2\gamma,\tau)$. The estimates in \eqref{stime forma normale limite} follow then by a telescoping argument with the estimate \eqref{lambdaestesi}. 
% \end{proof}
\begin{proof}
	By Proposition \ref{prop riducibilita} and by
%	\eqref{reversibility reality auto} and
	 \eqref{lambdaestesi}, we have that the sequence $\{\widetilde \mu_n(j;\omega)\}_{n\in\N}\subset \im \R$ is Cauchy on the closed set {$\tD\tC(2\gamma,\tau)$},  therefore it is convergent for any  {$\omega\in \tD\tC(2\gamma,\tau)$}. The estimates in \eqref{stime forma normale limite}
%	 , resp. the estimate in \eqref{auto.infty.delta12}, 
	 follow then by a telescoping argument with  \eqref{lambdaestesi}.
%	 , resp. with \eqref{r nu - 1 r nu i1 i2}
	 Finally, since $\{\widetilde \mu_n(j;\omega)\}_{n\in\N}\subset \im\R$ for any $j \in \Z^2 \setminus \{ 0 \}$, by \eqref{reversibility reality auto}, one has that  $ \mu_\infty(j;\omega) \in  \im\R$ for any $j \in \Z^2 \setminus \{ 0 \}$. 
\end{proof}
 We define the set ${\mathcal O}_\lambda$ of the non-resonance conditions for the final eigenvalues as 
 	\begin{align}
 		{\mathcal O}_\lambda %\equiv \Omega_\infty^\gamma(v) 
 		:= \Big\{ \omega& \in  {\tD\tC (2 \gamma,\tau)}%{\mathcal O}_\infty^\gamma 
 		 \, : \, |\im \, \lambda\, \omega \cdot \ell + \mu_\infty(j) -  \mu_\infty(j')| \geq \frac{2 \lambda \,\gamma}{\langle \ell \rangle^\tau | j' |^\tau } \,, \label{cantor finale ridu} \\
 		&  \quad \quad \ \ \forall \, \ell \in \Z^\nu \setminus \{ 0 \}\,, \ \ 	j,j' \in \Z^2 \setminus \{ 0 \}\,, \ \  \pi^\top(\ell) + j-j' =0\Big\}\,. 
 	\end{align}
 
 \begin{lem}\label{prima inclusione cantor}
 	We have ${\mathcal O}_\lambda \subseteq \cap_{n \geq 0} \, \Omega_n^\gamma$ and $|\Omega \setminus {\mathcal O}_\lambda| \to 0$ as $\lambda \to + \infty$. 
 \end{lem}
 
 \begin{proof} 
 The proof of the inclusion ${\mathcal O}_\lambda \subseteq \cap_{n \geq 0} \, \Omega_n^\gamma$  is exactly the same of Lemma 5.12 in \cite{BFMT1}. By Theorem \ref{teo principale beta plane} and by \eqref{stima diofantei}
 (together with \eqref{definizioneGAMMA}-\eqref{piccolo.ansatx}), it is enough to show that 
 $|{\tD\tC (2 \gamma,\tau)} \setminus {\mathcal O}_\lambda| \to 0$ as $\lambda \to + \infty$. One has that 
 \begin{align}
{\tD\tC (2 \gamma,\tau)} \setminus {\mathcal O}_\lambda & = \bigcup_{\begin{subarray}{c}
\ell \in \Z^\nu \setminus \{ 0\}\,,\, j, j' \in \Z^2 \setminus \{0 \} \\
\pi^T (\ell) + j - j' = 0
\end{subarray}} {\mathcal R}(\ell, j, j')\,,  \\
{\mathcal R}(\ell, j, j') & := \Big\{ \omega \in  {\tD\tC (2 \gamma,\tau)}  \, : \, |\im \, \lambda\, \omega \cdot \ell + \mu_\infty(j) -  \mu_\infty(j')| < \frac{2 \lambda \,\gamma}{\langle \ell \rangle^\tau | j' |^\tau } \Big\}\,.
 \end{align}
 By using the estimates of Lemma \ref{lemma blocchi finali}, arguing as in the proof of Lemma 8.3 in \cite{BFMT1}, one gets that 
 \[
 |{\mathcal R}(\ell, j, j')| \lesssim \gamma \langle 
 \ell \rangle^{- (\tau + 1)} |j'|^{- \tau}\,,
 \]
 implying that by choosing $\tau \gg 0$ large enough (see \eqref{definizione.param.KAM}), one obtains that
 $|{\tD\tC (2 \gamma,\tau)} \setminus {\mathcal O}_\lambda| \lesssim \gamma \lesssim \lambda^{- \mathtt c}$ 
 (recall \eqref{definizioneGAMMA})
 and hence the claimed statement follows. 
 \end{proof}

 Now we define the sequence of invertible maps
 \begin{equation}\label{trasformazioni tilde ridu}
 	\widetilde \Phi_n := \Phi_0 \circ \Phi_1 \circ \ldots \circ \Phi_n \,, \quad n \in \N\,. 
 \end{equation}
 In order to prove the convergence of the sequence $(\widetilde \Phi_n)_{n \geq 0}$, 
 we need the following lemma. 
 \begin{lem}\label{composizione iterata smoothingInfinita}
Let $s \geq s_0$, $\rho \geq 0$ and let $({\mathcal F}_n)_{n \geq 1}$ with 
${\mathcal F}_n \in \OpM_s^{- \rho}$ and $\Phi_n := {\rm Id} + {\mathcal F}_n$. 
We assume that 
$|{\mathcal F}_n|_{- \rho, s}^{\Lip(\gamma)} \leq \varepsilon_n$, for any $n \geq 1$,
and 
$\sum_{n = 1}^{+ \infty} \varepsilon_n < + \infty$. 
Then ${\bf \Phi}_N := \Phi_1 \circ \ldots \circ \Phi_N$  converges to ${\bf \Phi}_\infty$ 
as $N \to + \infty$ with respect to the norm $| \cdot |_{0, s}^{\Lip(\gamma)}$ 
and 
\begin{align}
& {\bf \Phi}_\infty = {\rm Id} + {\bf F}_\infty \quad \text{with} \quad {\bf F}_\infty \in \OpM_s^{- \rho}\,, 
\quad  
|{\bf F}_\infty|_{- \rho, s}^{\Lip(\gamma)} 
\lesssim 
\sum_{n = 1}^{+ \infty} \varepsilon_n\,. 
\end{align}
\end{lem}
\begin{proof}
To simplify the notation, we write $|\cdot |_{m, s}$ 
instead of $| \cdot |_{m, s}^{\Lip(\gamma)}$. 

\noindent
 First, by applying Lemma \ref{proprieta standard norma decay}-$(i)$ 
 inductively, one shows that 
\[
|\Phi_N|_{0, s} \leq \prod_{n = 1}^N \big( 1 + C(s) |{\mathcal F}_n|_{- \rho, s} \big) 
\leq 
\prod_{n = 1}^N \big( 1 + C(s) \varepsilon_n \big)\,,
\]
for some constant $C(s) > 0$. Then 
\[
\log \Big(\prod_{n = 1}^N \big( 1 + C(s) \varepsilon_n \big) \Big) 
= 
\sum_{n = 1}^N \log\big(1 + C(s) \varepsilon_n \big) 
\leq 
C(s) \sum_{n = 1}^{+ \infty} \varepsilon_n =: K(s)\,,
\]
and hence 
\[
\sup_{N \geq 1} |\Phi_N|_{0, s} \leq {\rm exp}(K(s)) =: K_1(s)\,. 
\]
Using the latter bound, Lemma \ref{proprieta standard norma decay}-$(i)$ 
and recalling \eqref{bf Phi N N + 1 astratto} from the proof of Lemma \ref{composizione iterata smoothing}, 
one has that, for any $N \geq 1$, 
\begin{align}
|{\bf F}_{N + 1} - {\bf F}_N|_{- \rho, s} & \leq 
|{\mathcal F}_{N + 1} + {\bf F}_N \circ {\mathcal F}_{N + 1}|_{- \rho, s} 
\leq |  {\bf \Phi}_N \circ {\mathcal F}_{N + 1}|_{- \rho, s}   
\\
& \lesssim_s |\Phi_N|_{0, s} |{\mathcal F}_{N + 1}|_{- \rho, s} 
\lesssim_s |{\mathcal F}_{N + 1}|_{- \rho, s} 
\lesssim_s \varepsilon_{N + 1}\,,
\end{align}
and therefore
\[
\sum_{N \geq 1} |{\bf F}_{N + 1} - {\bf F}_N|_{- \rho, s} 
\lesssim_s \sum_{N \geq 1} \varepsilon_{N + 1} 
< +\infty\,.
\]
Thus, by a standard telescoping argument, 
the sequence $({\bf F}_N)_{N \geq 1}$ converges in 
$\OpM_s^{- \rho}$ to ${\bf F}_\infty$  and 
\[
|{\bf F}_\infty|_{- \rho, s} \lesssim_s \sum_{N = 1}^{+ \infty} \varepsilon_N\,. 
\]
the claimed statement then follows by setting ${\bf \Phi}_\infty := {\rm Id} + {\bf F}_\infty$. 
\end{proof}
 
 %The following Lemma holds.
 
 \begin{prop}\label{lemma coniugio finale}
 	%	\label{lemma convergenza trasformazioni}
 	%Let $\gamma \in (0, 1)$, $\tau > 0$ and $s > s_0$. 
 	Let $S > s_0 + \Sigma(\tb)$. There exists $\delta := \delta (S) > 0$ such that, if  \eqref{KAM smallness condition} is verified, then the following holds. 
 	For any $\omega \in \mathcal{O}_{\lambda}$ 
    (see \eqref{cantor finale ridu}), 
    the sequence $(\widetilde \Phi_n)_{n\in\N}$
 	converges in norm $| \cdot |_{0, s}^{{\rm Lip}(\gamma)}$ 
 	to an invertible map $\Phi_\infty$, 
 	satisfying, for any $s_0\leq s\leq S-\Sigma(\tb)$,
 	%satisfying the estimates 
 	\begin{equation}\label{stima Phi infty}
 		\begin{aligned}
 		\Phi_\infty^{\pm 1} - {\rm Id} \in \OpM_s^{- 1}\,, 
		\quad  
		|\Phi_\infty^{\pm 1} - {\rm Id}|_{- 1, s}^{{\rm Lip}(\gamma)} 
 			\lesssim_{S}  \lambda^{\theta - 1} \gamma^{- 1}  \,.
 		\end{aligned}
 	\end{equation}
 The operators $\Phi_\infty^{\pm 1} : H^s_0 \to H^s_0$ 
 are real, reversibility preserving 
 and momentum preserving. Moreover, 
 for any $\omega\in \Omega_\infty^\gamma$, 
one has the conjugation (recall the operator $\bL_{0}$ 
given in \eqref{bL0_inizioKAM})
\begin{equation}\label{cal L infty e}%begin{equation}\label{eq coniugio finale}
 (\Phi_\infty)_* {\bL}_{0} =  {\mathcal D}\,, 
\quad 
{\mathcal D} := 
{\rm diag}_{j \in \Z^2 \setminus \{ 0 \}} \mu_\infty(j)  \,,
 \end{equation} %end{equation}
% 	 \eqref{def cal L (3)}-\eqref{op inizio riducibilita} 
 	%(recall that by \eqref{op inizio riducibilita} ${\mathcal L}^{(2)}_e \equiv {\mathcal L}_0$)
 	where the final eigenvalues $\mu_\infty(j)$ are given in 
	Lemma \ref{lemma blocchi finali}. 
 %   and $\Phi_{\infty}^{\pm 1}$ 
%	are reversibility preserving 
 %   and 
   Moreover, ${\mathcal D}$ is a real 
	and reversible diagonal operator.
 \end{prop}
 
 \begin{proof}
 	By \eqref{stime Psi n rid} and recalling the definition of $\varepsilon$ in \eqref{vare.vare} and 
	Lemma \ref{proprieta standard norma decay}, 
	one has  $\Phi_n = {\rm Id} + {\mathcal F}_n$  for any $n \geq 0$, with $ {\mathcal F}_n \in \OpM_s^{- 1}$ satisfying
    \begin{align}
    % &\,, 
    % \\
    & |{\mathcal F}_n|_{- 1, s}^{\Lip(\gamma)} 
    \lesssim_S  \tN_{n}^{2 \tau + 1} \tN_{n - 1}^{- \ta} \varepsilon^{M} \gamma^{- M} 
    \lesssim_S \tN_{n}^{2 \tau + 1} \tN_{n - 1}^{- \ta} (\lambda^{\theta - 1} \gamma^{- 1})^M 
    \stackrel{\lambda^{\theta - 1} \gamma^{- 1} \ll 1}{\lesssim_S} 
    \tN_{n}^{2 \tau + 1} \tN_{n - 1}^{- \ta} \lambda^{\theta - 1} \gamma^{- 1}\,,
    \end{align}
    and by \eqref{definizione.param.KAM} 
    $\sum_{n \geq 0} \tN_{n}^{2 \tau + 1} \tN_{n - 1}^{- \ta} \lambda^{\theta - 1} \gamma^{- 1} 
    \lesssim \lambda^{\theta - 1} \gamma^{- 1}$. 
    Then the convergence of the maps $\widetilde \Phi_n$ to $\Phi_\infty$ 
    and the corresponding properties of $\Phi_\infty$ 
    follow by Lemma \ref{composizione iterata smoothingInfinita}. 
    For $\widetilde \Phi_n^{- 1}$ and $\Phi_\infty^{- 1}$ one argues similarly. 
 	By \eqref{trasformazioni tilde ridu}, Lemma \ref{prima inclusione cantor} 
 	and Proposition \ref{prop riducibilita}, one has 
 	$(\widetilde \Phi_n)_* {\bL}_0  
 	= {\bD}_n + {\bE}_n$ for all $n \geq 0$. 
 	The claimed statement then follows by passing to the limit as $n \to \infty$, 
 	by using \eqref{stime cal Rn rid}, \eqref{stima Phi infty} 
	and Lemma \ref{lemma blocchi finali}. 
 	%\ref{lemma convergenza trasformazioni}. 
 \end{proof}

 \subsection{Proof of Theorem \ref{reducibility linearized totale}}

Before starting the proof we need the following technical result.

 \begin{lem}\label{bounded decay HS x}
Let $s \geq 0$, $m \in \R$ 
and ${\mathcal R} \in \OpM^m_{s + 2 s_0}$.  Then the map $\vf\mapsto \cR(\vf) \in \cB(H^{s+m}_0,H^s_0)$ is continuous, with
\begin{equation}
    \sup_{\vf \in \T^\nu} 
\| {\mathcal R}(\vf) \|_{{\mathcal B}(H^{s + m}_0, H^s_0)} 
\lesssim_s |{\mathcal R}|_{m, s + 2 s_0}\,.
\end{equation}
\end{lem}
\begin{proof}
Let $u \in H^{s + m}_0(\T^2)$. For any 
$\vf \in \T^\nu$, one has that 
\begin{equation}
    \| {\mathcal R}(\vf)[u] \|_{H^s}^2  
\leq 
\sum_{j \in \Z^2 \setminus \{ 0 \}} \langle j \rangle^{2 s} 
\Big( \sum_{j' \in \Z^2 \setminus \{ 0 \}} 
|{\mathcal R}(\vf)_j^{j'}| |\widehat u(j')| \Big)^2\,.
\end{equation}
Using that  $  \langle j \rangle^{ s} \lesssim_s \langle j' \rangle^{s} + \langle j - j' \rangle^{s} \lesssim_s \langle j' \rangle^{s} \langle j - j' \rangle^{s} $,
% \begin{equation}
%     \langle j \rangle^{ s} \lesssim_s \langle j' \rangle^{s} + \langle j - j' \rangle^{s} \lesssim_s \langle j' \rangle^{s} \langle j - j' \rangle^{s}\,,
% \end{equation}
one gets that 
\begin{equation}
    \| {\mathcal R}(\vf)[u] \|_{H^s}^2 
\lesssim_s 
\sum_{j \in \Z^2 \setminus \{ 0 \}} \Big( \sum_{j' \in \Z^2 \setminus \{ 0 \}} 
\langle j - j' \rangle^{s + s_0}
|{\mathcal R}(\vf)_j^{j'}| \langle j' \rangle^{- m} 
\langle j' \rangle^{s + m}
|\widehat u(j')| 
\frac{1}{\langle j - j' \rangle^{s_0}} \Big)^2\,
\end{equation}
(recall \eqref{definizione s0}). Then, by Cauchy Schwartz inequality (using that 
$\sum_{j'} \langle j - j' \rangle^{- 2 s_0} 
= \sum_{k} \langle k \rangle^{- 2 s_0} 
= C(s_0) < + \infty$), 
one gets 
\begin{align}
\| {\mathcal R}(\vf)[u] \|_{H^s}^2 & \lesssim_s 
\sum_{j \in \Z^2 \setminus \{ 0 \}}  \sum_{j' \in \Z^2} 
\langle j - j' \rangle^{2(s + s_0)}
|{\mathcal R}(\vf)_j^{j'}|^2 \langle j' \rangle^{- 2 m} 
\langle j' \rangle^{2(s + m)}|\widehat u(j')|^2  \Big)^2 
\\
& \lesssim_s \sum_{j' \in \Z^2 \setminus \{ 0 \}} 
\langle j \rangle^{2(s + m)} |\widehat u(j')|^2 
\sum_{j \in \Z^2 \setminus \{ 0 \}} 
\langle j - j' \rangle^{2(s + s_0)}
|{\mathcal R}(\vf)_j^{j'}|^2 \langle j' \rangle^{- 2 m} 
\\
& \lesssim_s M( s) \| u \|_{H^{s + m}}^2 \,,
\end{align}
where 
\begin{equation}
    M( s) := \sup_{\vf \in \T^\nu} 
\sup_{j' \in \Z^2 \setminus \{ 0 \}} 
\sum_{j \in \Z^2 \setminus \{ 0 \}} \langle j - j' \rangle^{2(s + s_0)}
|{\mathcal R}(\vf)_j^{j'}|^2 \langle j' \rangle^{- 2 m}\,.
\end{equation}
For any $\vf \in \T^\nu$, $j, j' \in \Z^2$, by  
the Cauchy Schwartz inequality (use that 
$\sum_\ell \langle \ell \rangle^{- 2 s_0} < + \infty$), 
one deduces 
\begin{align}
|{\mathcal R}(\vf)_j^{j'}| & \leq \sum_{\ell \in \Z^\nu}\langle \ell \rangle^{s_0}|\widehat{\mathcal R}(\ell)_j^{j'}|\langle \ell \rangle^{- s_0} \lesssim \Big( \sum_{\ell \in \Z^\nu} \langle \ell \rangle^{2 s_0} |\widehat{\mathcal R}(\ell)_j^{j'}|^2 \Big)^{\frac12}\,.\label{sifulo 100}
\end{align}
Hence, for any $j' \in \Z^2 \setminus \{ 0 \}$, $\vf \in \T^\nu$, one has that 
\begin{align}
\sum_{j \in \Z^2 \setminus \{ 0 \}} 
\langle j - j' \rangle^{2(s + s_0)}
|{\mathcal R}(\vf)_j^{j'}|^2 \langle j' \rangle^{- 2 m} 
& \lesssim 
\sum_{(\ell , j)\in \Z^{\nu} \times (\Z^2 \setminus \{ 0 \})} 
\langle \ell \rangle^{2 s_0} 
\langle j - j' \rangle^{2(s + s_0)}
|\widehat{\mathcal R}(\ell)_j^{j'}|^2 \langle j' \rangle^{- 2 m} 
\\& 
\lesssim 
\sum_{(\ell , j)\in \Z^{\nu} \times (\Z^2 \setminus \{ 0 \})} 
\langle \ell, j - j' \rangle^{2(s + 2 s_0)} 
|\widehat{\mathcal R}(\ell)_j^{j'}|^2 
\langle j' \rangle^{- 2 m} \\
& \lesssim |{\mathcal R}|_{m, s + 2 s_0}^2\,.
\end{align}
which implies 
$M(s) \lesssim |{\mathcal R}|_{m, s + 2 s_0}^2$. 
This concludes the proof.
\end{proof}

We are now in position to prove 
Theorem \ref{reducibility linearized totale}.
Recall our choice of $\gamma= \lambda^{- \mathtt c}$ in \eqref{definizioneGAMMA}
and (see \eqref{theta.def.ridu}) define
\begin{equation}
    \eta:= 1 - \theta - \mathtt c > 0\,.
\end{equation}
Hence, we note that
\begin{align}
\lambda^{\theta - 1} \gamma^{- 1} &= \lambda^{- \eta} \ll 1\,,
\\
\varepsilon^{M}\gamma^{-M}&\stackrel{\eqref{vare.vare},\eqref{definizioneGAMMA}}{=}
\lambda^{M(\theta-1)+M\mathtt{c}}
\stackrel{\eqref{theta.def.ridu}}{=}\lambda^{-M(2(1-\mathtt{c}) -\alpha)}
\stackrel{\eqref{definizione.param.KAM}}{\ll}1
\end{align}
for $\lambda\gg1$ large. Therefore,
    the smallness assumptions of 
   Propositions 
    \ref{prop coniugio cal L L1}, \ref{prop normal form lower orders}, \ref{lemma coniugio finale}
    (see \eqref{condizione piccolezza rid trasporto}, \eqref{piccoloNonpert}, 
    \eqref{KAM smallness condition})
   are satisfied taking $\lambda$ large enough. Then,
   for $\omega \in {\mathcal O}_\lambda$ (see \eqref{cantor finale ridu}), 
   we define the map
    $$
{\mathcal U}(\vf) := {\mathcal B}_\bot(\vf) \circ {\mathcal W}(\vf), \quad {\mathcal W}(\vf) := {\bf \Phi}_M(\vf) \circ \Phi_\infty(\vf), \quad \vf \in \T^\nu
    $$
    First of all the measure estimates on the set $\mathcal{O}_{\lambda}$
    follow by Lemma \ref{prima inclusione cantor}.
The conjugation results follow by \eqref{cal L infty e}.
It remains to prove the estimate 
\eqref{prop one smoothing teo principale} on the maps $\mathcal{U}^{\pm1}$.
By estimates \eqref{stima bf Phi M} and \eqref{stima Phi infty}, and using 
Lemma \ref{composizione iterata smoothing} with $N=2$
we get
\begin{equation}
    |\mathcal{W}^{\pm 1} - {\rm Id}|_{- 1, s}^{{\rm Lip}(\gamma)} 
 			\lesssim_{S}  \lambda^{-\eta}   \,.
\end{equation}
Therefore the estimate \eqref{prop one smoothing teo principale} follows by the bound above combined with Lemma \ref{bounded decay HS x}.
The estimate on the inverse map follows by similar arguments.
The bounds on $\mathcal{B}_{\perp}$ follows recalling 
 \eqref{stima tame cambio variabile rid trasporto}.

\section{Nonlinear stability of traveling waves and proof of the main results}\label{sezione nonlinear stability}
We want to study the Cauchy problem (corresponding  to equation 
\eqref{beta.waves.large.eq})
\begin{equation}\label{Cauchy-prob-beta-plane}
\left\{\begin{aligned}
	\,& \pa_{t} v =   \beta  \mathtt L  \,  v - {\mathfrak B}(v) \cdot \nabla v + \lambda^{\alpha} f(\lambda \omega t,x) \,, 
    \\
    \,& v(0, x) = v_0(x)
    \end{aligned}
    \right.
\end{equation}
where the initial condition $v_0 \in H^s_0(\T^2)$, $s > 2$, is such that 
\begin{equation}\label{stimaw0}
\| v_0 - v_\lambda(0, \cdot)\|_{H^s} \leq \delta \ll 1
\end{equation}
for some $\delta$ small enough, and 
where $v_\lambda(\lambda \omega t, x)$, 
$\omega \in {\mathcal O}_\lambda$, is a 
traveling wave solution of the beta-plane equation as in Theorem \ref{teo principale beta plane} (recall that $v_\lambda$ satisfies \eqref{ansatz}) 
By standard local existence arguments for quasi-linear 
hyperbolic PDEs (see for instance \cite{Kato1}, \cite{Taylor})
%, \cite{Kato2} ), 
one has that there exist a time of existence $T_{\rm loc}>0$ 
%$\simeq_s \frac{1}{\| v_0 \|_{H^s}}$ 
and a unique solution 
$$
v \in  {\mathcal C}^0\big([0, T_{\rm loc}], H^s_0(\T^2) \big) 
\cap {\mathcal C}^1\big([0, T_{\rm loc}], H^{s - 1}_0(\T^2) \big)
$$
of the Cauchy problem \eqref{Cauchy-prob-beta-plane}. In order to study perturbatively the stability of the traveling wave solution $v_\lambda $,
It is natural to write 
\begin{equation}\label{v v lambda w}
v(t, x) = v_\lambda(\lambda \omega t, x) + w(t, x)\,.
\end{equation}
Then the Cauchy problem \eqref{Cauchy-prob-beta-plane}, 
written in terms of the function $w$, becomes
\begin{equation}\label{Cauchy-prob-beta-plane2}
\left\{
\begin{aligned}
	\,& \pa_{t} w =   {\mathcal L}(\lambda \omega t) [w] + {\mathcal N}[w, w] \\
  \,  & w(0, x) = w_0(x) := v_0(x) - v_\lambda(0, x)
    \end{aligned}
    \right.
\end{equation}
where ${\mathcal L}(\lambda \omega t)$ is the linearized vector field 
at the traveling wave solution 
$v_\lambda(\lambda \omega t, x)$ (see 
\eqref{operatore linearizzato}) and
\begin{equation}\label{nuovo termine nonlineare}
{\mathcal N}[w_1, w_2] :=  - {\mathfrak B}(w_1) \cdot \nabla w_2, \quad {\mathfrak B}(w) =  \nabla^\perp (-\Delta)^{-1} w \,.
\end{equation}
In the following subsections we analyze the stability properties of the Sobolev norms of the function $w$. First, we will show that the size of solution $w(t)$ of the Cauchy problem \eqref{Cauchy-prob-beta-plane2} remains controlled by the size of the initial datum $w_0$ in $H^s$-topology in a small, but non-trivial time interval of existence (Proposition \ref{prop:shorttime}) which depends of the size of the traveling wave solution. Then, in Theorem \ref{maximal time} we will promote this stability property for arbitrary long time, i.e. independent of the size of the traveling wave solution $v_\lambda$. We will conclude the section with the proofs of Theorem \ref{thm:mainsatability} and Theorem \ref{corollario nonlinear stability}.

\subsection{Short time stability}
In this section we prove short time stability estimates 
for solutions $w$ of the problem 
\eqref{Cauchy-prob-beta-plane2}. 
More precisely we have the following.

\begin{prop}\label{prop:shorttime}
Let $s > 2$ and consider the problem 
\eqref{Cauchy-prob-beta-plane2} with initial datum $w_0$ 
satisfying the smallness condition \eqref{stimaw0}. 
Then there exist $\s>0$, $\gamma > 0$ such that, for $v_\lambda$ satisying \eqref{ansatz} with $S \geq s + \sigma$,
$T_\lambda := \frac{\gamma}{\lambda^{\theta}}$
and $\delta<\lambda^{\theta}/6$, ($\theta = \alpha - 1 + \mathtt c$, see \eqref{theta.def.ridu})
then 
$$
\sup_{t \in [0, T_\lambda]} \| w(t) \|_{H^s} < 2 \|w_0\|_{H^{s}}\,. 
$$
\end{prop}
In order to prove Proposition \ref{prop:shorttime}
we first need some preliminary results.
We start by collecting elementary properties of the quadratic
nonlinear term ${\mathcal N}$ and some technical facts required for the energy estimates. 
\begin{lem}\label{lemma elementare nonlinearita}
Let $s > 1$. The following hold:
\\[1mm]
\noindent $(i)$ If  $w \in H^s_0(\T^2)$, then $\fB(w)\in H_{0}^{s}(\T^2)$ with estimate
$\|{\mathfrak B}(w) \|_{H^s} \lesssim_{} \| w \|_{H^{s - 1}}$;
\\[1mm]
\noindent $(ii)$ If $w_1 \in H^{s - 1}_0(\T^2)$ and
$w_2 \in H^{s + 1}_0(\T^2)$, then  
 ${\mathcal N}[w_1, w_2] \in H^s_0(\T^2)$ with estimate 
$\|{\mathcal N}[w_1, w_2] \|_{H^s} \lesssim_s \| w_1 \|_{H^{s - 1}} \| w_2 \|_{H^{s + 1}}$. 
As a consequence,
$\| {\mathcal N}[w, w] \|_{H^s}
 \lesssim_s \| w \|_{H^{s + 1}}^2$ for any $w \in H^{s + 1}_0(\T^2)$.
\end{lem}
\begin{proof}
Item $(i)$ follows immediately by computing
\[
\|{\mathfrak B}(w) \|_{H^s} = 
\| \nabla^\perp (-\Delta)^{-1} w\|_{H^s} 
\lesssim 
\| (-\Delta)^{-1} w \|_{H^{s + 1}} 
\lesssim \| w \|_{H^{s - 1}}\,.
\]
To prove item $(ii)$, using the algebra property of $H^s(\T^2)$ for $s>1$, 
one deduces 
\[
 \| {\mathfrak B}(w_1) \cdot \nabla w_2 \|_{H^s} 
 \lesssim_s 
 \| {\mathfrak B}(w_1)\|_{H^s} 
 \| \nabla w_2 \|_{H^s} 
 \lesssim_s 
 \| w_1 \|_{H^{s - 1} } \| w_2 \|_{H^{s + 1}}\,.
\]
Moreover, since ${\mathfrak B}(w_1)$ is a zero 
divergence vector field, by integration by parts one has 
\[
\int_{\T^2} {\mathfrak B}(w_1) \cdot \nabla w_2 \wrt x 
= - \int_{\T^2} {\rm div}\big({\mathfrak B}(w_1) \big) w_2 \wrt x = 0\,,
\] 
which implies that, if the functions $w_1, w_2$ have zero average, then also 
${\mathcal N}[w_1, w_2]$ has zero average. 
This concludes the proof.
\end{proof}

\noindent
We define the operator $\Lambda^s$ as 
$$
\Lambda^s u(x) := \sum_{\xi \in \Z^2 } \langle \xi \rangle^{ s} \widehat u(\xi) e^{\im x \cdot \xi}, \quad u \in H^s(\T^2)\,.
$$
Clearly $\| u \|_{H^s} = \| \Lambda^s u \|_{L^2}$. The following Kato-Ponce commutator estimate holds. 
\begin{lem}\label{Kato commutator estimate}
Let $s>2$, $u \in H^s(\T^2)$, $a \in H^s(\T^2, \R^2)$. Then 
$$
\| [ \Lambda^s\,,\,a \cdot \nabla] u\|_{L^2} \lesssim_s \| a \|_{H^s} \| u \|_{H^s} 
$$
\end{lem}
\begin{proof}
Let $A := [\Lambda^s, a \cdot \nabla ]$. By a direct calculation one has that 
$$
\widehat{A u}(\eta) =  \sum_{\xi \in \Z^2} m(\eta, \xi) \cdot \widehat a(\eta - \xi) \widehat u(\xi)\,,
\qquad 
m(\eta, \xi) :=  
\im \big( \langle \eta \rangle^s - 
\langle  \xi \rangle^s\big) \xi\,.
$$
%where 
%$$
%m(\eta, \xi) :=  \im \big( \langle \eta \rangle^s 
%- \langle  \xi \rangle^s\big) \xi\,.
%$$
By the mean value  and using that 
$\langle \eta \rangle^{s - 1} 
\lesssim_s 
\langle \xi \rangle^{s - 1} 
+ \langle \eta - \xi \rangle^{s - 1}$, 
one gets that 
$$
\begin{aligned}
|m (\eta, \xi)|\lesssim \langle\xi\rangle|\langle \eta \rangle^s - \langle \xi \rangle^s| 
&\lesssim_s \langle\xi\rangle(\langle \eta \rangle^{s - 1} 
+ \langle \xi \rangle^{s - 1})\langle \eta 
- \xi \rangle 
%\\& 
\lesssim_s \langle \xi \rangle^{s }
\langle \eta - \xi \rangle 
+ \langle \eta - \xi \rangle^s\langle\xi\rangle\,.
\end{aligned}
$$
Hence, one obtains that
\begin{align}
\| A u \|_{L^2}^2 & \lesssim_s  T_1 + T_2\,, \\
T_1 &:= \sum_{\eta \in \Z^2} \Big( \sum_{\xi \in \Z^2}  \langle \xi \rangle^s \langle \eta - \xi \rangle^{s} |\widehat a(\eta - \xi)| |\widehat u(\xi)| \frac{1}{\langle \xi\rangle^{s - 1}} \Big)^2\,, \\
T_2 & := \sum_{\eta \in \Z^2} \Big( \sum_{\xi \in \Z^2}  \langle \xi \rangle^s \langle \eta - \xi \rangle^{s} |\widehat a(\eta - \xi)| |\widehat u(\xi)| \frac{1}{\langle \eta - \xi \rangle^{s - 1 }} \Big)^2\,. 
\end{align}
We start by considering the 
term $T_1$. By the Cauchy-Schwartz inequality, using that $s - 1  > 1$ and hence $\sum_{\xi} \langle  \xi \rangle^{- 2 (s - 1) } =  C(s) < + \infty$, one gets that 
\begin{align}
T_1 & = \sum_{\eta \in \Z^2} \Big( \sum_{\xi \in \Z^2}  \langle \xi \rangle^s \langle \eta - \xi \rangle^{s} |\widehat a(\eta - \xi)| |\widehat u(\xi)| \frac{1}{\langle \xi\rangle^{s - 1}} \Big)^2 \\
& \lesssim_s \sum_{\eta \in \Z^2}  \sum_{\xi \in \Z^2}  \langle \xi \rangle^{2 s} \langle \eta - \xi \rangle^{2 s} |\widehat a(\eta - \xi)|^2 |\widehat u(\xi)|^2  \\
& \lesssim_s \sum_{\xi \in \Z^2} \langle \xi \rangle^{2 s} |\widehat u(\xi)|^2 \sum_{\eta \in \Z^2} \langle \eta - \xi \rangle^{2 s} |\widehat a(\eta - \xi)|^2 \lesssim_s \| a \|_{H^s}^2 \| u \|_{H^s}^2\,. 
\end{align}
In the same way one gets that $T_2 \lesssim_s \| a \|_{H^s}^2 \| u \|_{H^s}^2$ and hence the lemma follows. 
\end{proof}
In the following we shall make a repeated use of the following fact. By the Sobolev embedding, one has that for any $s \geq 0$, 
$f \in H^{s + \frac{\nu + 1}{2}}(\T^\nu \times \T^2)$, 
\begin{equation}\label{Sobolev embedding finale}
\sup_{t \in \R} \| f(\lambda \omega t, \cdot ) \|_{H^s} \leq \sup_{\vf \in \T^\nu} \| f(\vf, \cdot) \|_{H^s} 
\lesssim \| f \|_{s + \frac{\nu + 1}{2}}\,. 
\end{equation}

We are now in position to prove the main result 
of this subsection.

\begin{proof}[{\bf Proof of Proposition \ref{prop:shorttime}}.]
  Let $T_{\lambda}:= \frac{\gamma}{\lambda^{\theta}}$ and fix $0 < \gamma \ll 1$ small enough such that $T_\lambda < T_{\rm loc}$. 
 We shall perform an energy estimate on the Cauchy 
 problem 
 \eqref{Cauchy-prob-beta-plane2}. 
 Recalling 
 \eqref{operatore linearizzato}-\eqref{def coefficienti operatori linearized} 
 and \eqref{nuovo termine nonlineare}, we write the Cauchy problem in \eqref{Cauchy-prob-beta-plane2} as 
\begin{equation}\label{Cauchy-prob-beta-plane3}
     \begin{cases}
\partial_t w(t) = \beta \mathtt L w + a(t, w(t)) \cdot \nabla w(t) + {\mathcal E}_0(\lambda\omega t)[w(t)] \\
w(0, x) = w_0(x)
 \end{cases}
\end{equation}
 where
 \begin{equation}\label{vensanto1}
a(t,w(t))=\ba_{0}(\lambda\omega t, x)-\mathfrak{B}(w(t)) = - \mathfrak{B}(v_\lambda(\lambda \omega t, \cdot)) - \mathfrak{B}(w(t))
 \end{equation}
 where we recall the definitions of ${\bf a}_0$ and ${\mathcal E}_0$ given in \eqref{def coefficienti operatori linearized}. 
 Hence, the function $a(t,w(t))$ in \eqref{vensanto1} satisfies,
 using \eqref{stima bf a1}, \eqref{Sobolev embedding finale} and  
 Lemma \ref{lemma elementare nonlinearita}, by taking $ S > s - 1 + (\nu+1)/2 $ large enough in \eqref{ansatz}  and for any $t\in [0,T_{\lambda}]$,
 \begin{align}
 \|a(t,w(t))\|_{H^{s}} & \leq \| \mathfrak{B}(v_\lambda(\lambda \omega t, \cdot) \|_{H^s} + \| \mathfrak{B}(w(t)) \|_{H^s} \\
 & \lesssim \| v_\lambda(\lambda \omega t, \cdot) \|_{H^{s - 1}} + \| w(t) \|_{H^{s - 1}} \label{stime a cal R Qbasic} \\
 & \lesssim_{s}\|v_\lambda \|_{S}+\|w(t)\|_{H^{s-1}} \lesssim_{S}\lambda^{\theta} +\|w(t)\|_{H^{s-1}}\,.
\end{align}
 Moreover, for any $t \in \R$, by using the algebra property of $H^s(\T^2)$, $s > 2$( see \eqref{algebra Hs}), by taking $S  >  s +  1 + \frac{\nu + 1}{2}$ in \eqref{ansatz}, by using the property \eqref{Sobolev embedding finale}, by using Lemma \ref{lemma elementare nonlinearita}, one obtains
\begin{align}
\| {\mathcal E}_0(\lambda\omega t)[w(t)] \|_{H^s} & =  \|  \nabla  v_\lambda(\lambda \omega t, \cdot) \cdot \fB [w(t)] \|_{H^s} \lesssim_s \| v_\lambda(\lambda \omega t, \cdot)\|_{H^{s + 1}} \| \fB[w(t)] \|_{H^s} \label{stime cal E 0 lambda omega t} \\
& \lesssim_s \| v_\lambda \|_{S} \| w(t) \|_{H^{s - 1}} \lesssim_S \lambda^\theta \| w(t) \|_{H^{s - 1}}\,. 
\end{align}
 To shorten notations, we write 
 $a(t) \equiv a(t, w(t))$.
 %${\mathcal R}_{\mathcal Q}(t) \equiv 
 %{\mathcal R}_{\mathcal Q}(\lambda \omega t, \psi(t))$. 
By \eqref{Cauchy-prob-beta-plane3}, one computes 
\begin{align}
\partial_t \| w(t) \|_{H^s}^2 & = 
\langle \Lambda^s \partial_t w(t)\,,\, 
\Lambda^s w(t) \rangle_{L^2} 
+ \langle \Lambda^s  w(t)\,,\, 
\Lambda^s \partial_t w(t) \rangle_{L^2} 
\nonumber
\\
& = \langle \beta \Lambda^s \mathtt{L} w(t)\,,\,
\Lambda^s w(t) \rangle_{L^2} 
+ \langle \Lambda^s  w(t)\,,\, 
\beta\Lambda^s  \mathtt{L} w(t) \rangle_{L^2} 
\label{stima energia 11basic}
\\
& \quad + \langle \Lambda^s  a(t) \cdot \nabla w(t)\,,
\, \Lambda^s \psi(t) \rangle_{L^2} 
+ \langle \Lambda^s  \psi(t)\,,\, 
\Lambda^s a(t) \cdot \nabla  w(t) \rangle_{L^2}
\label{stima energia 12basic}
\\
& \quad + \langle \Lambda^s  {\mathcal E}_0(\lambda\omega t)[w(t)]
\,,\, \Lambda^s w(t) \rangle_{L^2} + \langle \Lambda^s  w(t)\,,\, 
\Lambda^s {\mathcal E}_0(\lambda\omega t)[w(t)]\rangle_{L^2}\,.
\label{stima energia 13basic}
\end{align}
Note that, since ${\mathtt L} = \partial_{x_1} (- \Delta)^{- 1}$ is a real diagonal operator and skew self-adjoint, we have that $ [\mathtt L, \Lambda^s] = 0 $ and $\mathtt{L} = - {\mathtt L}^*$.
 % one concludes that
 % \begin{equation}\label{tL.bello}
 %   [\mathtt L, \Lambda^s] = 0  \quad \textnormal{and} \quad \mathtt{L} = - {\mathtt L}^* \,.
 % \end{equation}
Hence, we deduce that 
\begin{equation}\label{stima energia 2basic}
\begin{aligned}
%& \langle {\mathcal D}\Lambda^s  \psi(t)\,,\, \Lambda^s %\psi(t) \rangle_{L^2} + \langle \Lambda^s  \psi(t)\,,\, 
%{\mathcal D}\Lambda^s  \psi(t) \rangle_{L^2} 
%\\
%& 
\eqref{stima energia 11basic}= 
\langle \beta (\mathtt{L} +\mathtt{L}^*)\Lambda^s  w(t)\,,\, \Lambda^s w(t) \rangle_{L^2}  = 0\,. 
\end{aligned}
\end{equation}
We now study the terms in \eqref{stima energia 12basic}.
By a direct calculation one has that 
\begin{equation}\label{aggiunto trasportobasic}
\big(a(t) \cdot \nabla \big)^* = - a(t) \cdot \nabla \quad \text{since} \quad {\rm div}(a(t) ) = 0
\end{equation}
Then, by using the Cauchy-Schwartz inequality, one computes
\begin{align}
%& \langle \Lambda^s  a(t) \cdot \nabla \psi(t)\,,\, %\Lambda^s \psi(t) \rangle_{L^2} 
%+ \langle \Lambda^s  \psi(t)\,,\, 
%\Lambda^s a(t) \cdot \nabla  \psi(t) \rangle_{L^2} 
%\\& 
\eqref{stima energia 12basic}&= 
\langle   a(t) \cdot \nabla \Lambda^s w(t)\,,\, \Lambda^s w(t) \rangle_{L^2} 
+ \langle \Lambda^s  w(t)\,,\,  
a(t) \cdot \nabla  \Lambda^s w(t) \rangle_{L^2} 
\\
& \qquad 
+ \langle [\Lambda^s\,,\,  a(t) \cdot \nabla] w(t)\,,\, \Lambda^s w(t) \rangle_{L^2} 
+ \langle \Lambda^s  \psi(t)\,,\, 
[\Lambda^s\,,\, a(t) \cdot \nabla]  
w(t) \rangle_{L^2}
\\
& \leq 
\big\langle   \big(a(t) \cdot \nabla 
+ (a(t) \cdot \nabla)^* \big)\Lambda^s w(t)\,,\, \Lambda^s w(t) \big\rangle_{L^2} 
+ 2\| [\Lambda^s\,,\,  
a(t) \cdot \nabla] w(t)\|_{L^2}
\| \Lambda^s w(t) \|_{L^2} 
\\
& \stackrel{\eqref{aggiunto trasportobasic}, 
\text{Lemma} \ref{Kato commutator estimate}}{\lesssim_s} 
 \| a(t) \|_{H^s}\| w(t) \|_{H^s}^2 \,.
\end{align}
Finally, by \eqref{stime a cal R Qbasic}, 
one obtains the  estimate 
\begin{equation}\label{stima energia 3basic}
%\langle \Lambda^s  a(t) \cdot \nabla \psi(t)\,,\, 
%\Lambda^s \psi(t) \rangle_{L^2} 
%+ \langle \Lambda^s  \psi(t)\,,\, 
%\Lambda^s a(t) \cdot \nabla  \psi(t) \rangle_{L^2}
\eqref{stima energia 12basic}
\lesssim_s \| a(t) \|_{H^s}\| w(t) \|_{H^s}^2 
\lesssim_S\lambda^{\theta}\|w(t)\|_{H^{s}}^2+ \| w(t) \|_{H^s}^3\,. 
\end{equation}
We are left to estimate the terms in \eqref{stima energia 13basic}.
%last two terms in \eqref{stima energia 1}. 
By \eqref{stime cal E 0 lambda omega t}, one has 
\begin{align}
%& \langle \Lambda^s  {\mathcal R}_{\mathcal Q}(t) 
%\,,\, \Lambda^s \psi(t) \rangle_{L^2} 
%+ \langle \Lambda^s  \psi(t)\,,\, 
%\Lambda^s {\mathcal R}_{\mathcal Q}(t)\rangle_{L^2}  
%\\& 
\eqref{stima energia 13basic}&\leq 
2 \| \Lambda^s {\mathcal E}_0(\lambda\omega t)[w(t)] \|_{L^2} 
\| \Lambda^s w(t) \|_{L^2} \label{stima di energia 4basic}
\\
&
\leq 
2 \| {\mathcal E}_0(\lambda\omega t)[w(t)] \|_{H^s} 
\|  w(t) \|_{H^s}  
\lesssim_S
\lambda^{\theta}\| w(t) \|_{H^s}^2\,.
\end{align}
Therefore, by 
\eqref{stima energia 11basic}-\eqref{stima energia 13basic}
and 
\eqref{stima energia 2basic}, \eqref{stima energia 3basic}, \eqref{stima di energia 4basic}, 
one gets 
\begin{equation}\label{stima energia 5basic}
\partial_t \| w(t) \|_{H^s} 
\leq {\mathfrak C}(S)
\big(\lambda^{\theta}+\|w(t)\|_{H^{s}}\big) \| w(t) \|_{H^s}\,, 
\quad \forall \, t \in [0, T_\lambda]\,,
\end{equation}
for some constant ${\mathfrak C}(S) > 0$. 
By the comparison principle for ODEs, one has that 
\begin{align}
& \| w(t) \|_{H^s} \leq z(t)\,, 
\quad \forall \, t \in [0, T_\lambda]\,, \quad \text{where} \\
& \dot z(t) = {\mathfrak C}(S)(\lambda^{\theta}+z(t)) z(t)\,, 
\quad z(0) = \| w_0 \|_{H^s}\,.
\end{align}
Hence by integrating the ODE above one gets 
$$
\dfrac{z(t)}{\lambda^\theta + z(t)} = \dfrac{\| w_0 \|_{H^s}}{\lambda^\theta + \| w_0 \|_{H^s}} {\rm exp}({\mathfrak C}(S) \lambda^\theta t), \quad \forall \,
t \in [0, T_\lambda]\,.
$$
Then, since $T_\lambda = \frac{\gamma}{\lambda^\theta}$, 
choosing $\gamma = \gamma(S) \ll 1$ small enough ensures that
${\rm exp}({\mathfrak C}(S) \lambda^\theta t) \leq \frac32$ for any $t\in [0,T_{\lambda}]$. 
Hence, by the latter equality, 
$$
z(t) \Big( 1 -  \frac32 \frac{\| w_0 \|_{H^s}}{\lambda^\theta 
+ \| w_0 \|_{H^s}} \Big) 
\leq 
\frac32\dfrac{\lambda^\theta \| w_0 \|_{H^s}}{\lambda^\theta + \| w_0 \|_{H^s}} \leq \frac32 \| w_0 \|_{H^s}\,, 
\quad \forall \, t \in [0, T_\lambda]\,. 
$$
By taking $\| w_0 \|_{H^s} \leq \delta $, 
we have
$$
1 - \frac32 \frac{\| w_0 \|_{H^s}}{\lambda^\theta + \| w_0 \|_{H^s}} 
\geq 1 - \frac{3\delta}{2\lambda^\theta}\,, 
$$
Therefore, we conclude that 
$$
z(t) \leq \frac32 \dfrac{1}{1 - \frac{3\delta}{2\lambda^\theta}} 
\| w_0 \|_{H^s} < 2 \| w_0 \|_{H^s}\,, 
\qquad \forall \, t \in [0, T_\lambda]\,,
$$
by taking $0 < \delta < \frac{\lambda^\theta}{6}$. 
\end{proof}

\subsection{Long time stability}
In this section we provide long time stability estimates for
the solution $w$ of the problem \eqref{Cauchy-prob-beta-plane2}.
Let us define
\begin{equation}\label{supremumT}
T_{*}:= T_{*}(\delta):=\sup\Big\{ T>0 \;:\;
\sup_{t\in[0,T]}\|w(t)\|_{H^{s}} < 2 \delta 
\Big\}\,.
\end{equation}
Note that, thanks to the short time stability estimate proved in Proposition \ref{prop:shorttime}, we have $T_{*}(\delta)>0$ for any $\delta \in (0,\frac{\lambda^\theta}{6})$ and
\[
\sup_{t\in[0,T_*)}\|w(t)\|_{H^{s}}=2\delta\,.
\]
We want to show that, when the size $\delta$ of the initial datum $w_0$ is sufficiently small, the supremum of the time of existence $T_{*}$ in \eqref{supremumT} is at least of order $O(\delta^{-1})$. We have the following theorem.
\begin{thm}\label{maximal time}
Let $s>2$. There is $\bar{\mu}>0$ (independent of $s$) such that,
if $v_{\lambda}$
satisfies \eqref{ansatz} with $S = s + \bar{\mu}$ 
the following holds.
Consider a solution $w$ of the problem 
\eqref{Cauchy-prob-beta-plane2} with initial datum $w_0$ 
satisfying  \eqref{stimaw0}.  There exists a small constant $\gamma(s) > 0$ 
such that, for any $\delta$ in \eqref{stimaw0} small enough, then $T_{*}$ in \eqref{supremumT} satisfies
$T_* > T_\delta$, where 
$T_\delta \equiv T_\delta(s) := \gamma(s) \delta^{- 1}$,
and one has the following estimate
\begin{equation}\label{stima long time thm}
    \sup_{t \in [0, T_\delta]} \| w(t) \|_{H^s} 
< 2 \delta\,. 
\end{equation}
\end{thm}
The main idea to prove the above result is to  
rewrite the problem 
\eqref{Cauchy-prob-beta-plane2} by  applying 
the change of coordinates 
\begin{equation}\label{definizione w (t)}
w(t) = {\mathcal U}(\lambda \omega t)[\psi(t)]
\end{equation}
where $\mathcal{U}(\lambda\omega t)$
is the map given by Theorem \ref{reducibility linearized totale}.
%we have that 
%the Cauchy problem \eqref{Cauchy-prob-beta-plane2} 
%transforms into 
If $w(t)$ is a solution of \eqref{Cauchy-prob-beta-plane2}
then the new variable $\psi$
solves the Cauchy problem
\begin{equation}\label{problema di Cauchy trasformato dopo riduzione}
\begin{cases}
\partial_t \psi = {\mathcal D} \psi + {\mathcal Q}(\lambda \omega t, \psi) \\
\psi(0, x) = \psi_0(x) 
\end{cases}
\end{equation}
where ${\mathcal D}$ is a real diagonal operator with purely imaginary eigenvalues, 
and where
\begin{align}
\psi_0 & := {\mathcal U}(0)^{- 1}[w_0]\,,\label{def psi0}
\\
{\mathcal Q}(\vf, \psi) &:= {\mathcal U}(\vf)^{- 1} 
{\mathcal N}\big[{\mathcal U}(\vf)[\psi], {\mathcal U}(\vf)[\psi] \big]\,, 
\qquad \vf \in \T^\nu\,. 
\label{trasformed-nonlinearity}
\end{align}

Recall that ${\mathcal U}(\vf) = {\mathcal B}_\bot (\vf) \circ {\mathcal W}(\vf)$, ${\mathcal B}_\bot(\vf) = \Pi_0^\bot {\mathcal B}(\vf) \Pi_0^\bot$ and by \eqref{cal B bot inverse}, ${\mathcal B}_\bot(\vf)^{- 1} = \Pi_0^\bot {\mathcal B}(\vf)^{- 1} \Pi_0^\bot$. 
Let $s > 2$. By Theorem \ref{reducibility linearized totale}, there exist a constant large constant $\bar \mu \gg \bar \sigma$ (where $\bar \sigma$ is given in Theorem \ref{reducibility linearized totale}) a positive constant $\eta >0$ such that, if $S = s + \bar\mu$ (recall \eqref{ansatz}), then, for any $1 \leq p  \leq s + 1$ (by taking $\lambda \equiv \lambda(s) \gg 0$ large enough), one has that 
\begin{align}
&{\mathcal W}(\vf)^{\pm 1} = {\rm Id} + {\mathcal F}_{\pm}(\vf)\,, \quad \vf \in \T^\nu, 
\\
& \sup_{\vf \in \T^\nu} \|{\mathcal F}_{\pm}(\vf) \|_{{\mathcal B}(H^{p - 1}_0, H^p_0)} \lesssim_s \lambda^{- \eta}\,\,
\qquad 
\sup_{\vf \in \T^\nu} \| {\mathcal W}(\vf)^{\pm 1} \|_{{\mathcal B}(H^p_0)} \leq 1 + C(s) \lambda^{- \eta} \lesssim 1 \,, \label{WPMFpm}
\\& 
\sup_{\vf \in \T^\nu} \|{\mathcal B}(\vf)^{\pm 1} \|_{{\mathcal B}(H^p_0)}, \sup_{\vf \in \T^\nu} \|{\mathcal B}_\bot(\vf)^{\pm 1} \|_{{\mathcal B}(H^p_0)} \leq 1 + C(s) \lambda^{- \eta} \lesssim 1\,, \\
& \sup_{\vf \in \T^\nu} \| {\mathcal U}(\vf)^{\pm 1} \|_{{\mathcal B}(H^p_0)} \leq 1 + C(s) \lambda^{- \eta} \lesssim 1\,, 
%\\& 
\qquad \| \bbeta \|_{s + 1 + \frac{\nu + 1}{2}}
\lesssim_s \lambda^{- \eta}\,. 
\end{align}

In order to perform a sharp energy estimate on the Cauchy problem \eqref{problema di Cauchy trasformato dopo riduzione} it is fundamental to analyze the structure of the transformed nonlinearity ${\mathcal Q}(\vf, \psi)$ in \eqref{trasformed-nonlinearity}. In the next proposition we shall prove that, up to a bounded remainder, the nonlinear term ${\mathcal Q}(\vf, \psi)$ is a nonlinear transport type operator satisfying good estimate in $H^s$ which are suited to perform energy estimates on the PDE \eqref{problema di Cauchy trasformato dopo riduzione}. 
\begin{prop}\label{lemma struttura nonlin sistema trasformato}
Let $s > 2$ and assume that $v_\lambda$ satisfies \eqref{ansatz} with $S = s + \bar \mu$, Then, the nonlinear term ${\mathcal Q}(\vf,\psi)$ in \eqref{trasformed-nonlinearity} has the form 
$$
{\mathcal Q}(\vf, \psi) = \Pi_0^\bot \big( \ta(\vf, \psi) \cdot \nabla \psi \big) + {\mathcal R}_{\mathcal Q}(\vf, \psi) \,,
$$
with estimates 
\begin{equation}\label{stime a cal R Q}
\begin{aligned}
& \| \ta(\vf, \psi) \|_{H^s} \lesssim_s \| \psi \|_{H^{s - 1}}, \quad \| {\mathcal R}_{\mathcal Q}(\vf, \psi) \|_{H^s} \lesssim_s \| \psi \|_{H^s}^2, \quad  \text{uniformly on} \quad \vf \in \T^\nu\,. 
\end{aligned}
\end{equation}
\end{prop}
\begin{proof}
By \eqref{trasformed-nonlinearity}, one has 
\begin{equation}\label{explicit mahtcal Q}
{\mathcal Q}(\vf, \psi) = {\mathcal U}(\vf)^{- 1} {\mathcal N}\big[{\mathcal U}(\vf)[\psi]\,,\,{\mathcal U}(\vf)[\psi]\big], \quad \vf \in \T^\nu\,, 
\end{equation}
where the nonlinearity $\cN[w_1,w_2]$ is as in \eqref{nuovo termine nonlineare}.
Using the expansion of the maps $\mathcal{W}^{\pm 1}(\vf)$
in \eqref{WPMFpm} we write 
\begin{equation}\label{prima espansione mathcal Q}
\begin{aligned}
{\mathcal Q}(\vf, \psi) &   = 
{\mathcal B}_\bot(\vf)^{- 1} {\mathcal N}
\big[{\mathcal U}(\vf)[\psi], 
{\mathcal B}_\bot(\vf)[\psi] \big] 
+ {\mathcal R}_{\mathcal Q}(\vf, \psi)\,, 
\end{aligned}
\end{equation}
%\begin{equation}\label{prima espansione mathcal Q}
%\begin{aligned}
%{\mathcal Q}(\vf, \psi) & = 
%{\mathcal W}(\vf)^{- 1} \circ 
%{\mathcal B}_\bot(\vf)^{- 1} {\mathcal N}
%\big[{\mathcal U}(\vf) [\psi], 
%{\mathcal U}(\vf) [\psi] \big] 
%\\
%& = ({\rm Id} + {\mathcal F}_-(\vf))\circ 
%{\mathcal B}_\bot(\vf)^{- 1} {\mathcal N}
%\big[{\mathcal U}(\vf) [\psi], 
%{\mathcal U}(\vf) [\psi] \big] 
%\\
%& = {\mathcal B}_\bot(\vf)^{- 1} {\mathcal N}
%\big[{\mathcal U}(\vf) [\psi], 
%{\mathcal B_\bot}(\vf) \circ ({\rm Id} 
%+ {\mathcal F}_+(\vf)) [\psi] \big] 
%\\
%& \qquad + 
%{\mathcal F}_-(\vf) \circ 
%{\mathcal B}_\bot(\vf)^{- 1} {\mathcal N}
%\big[{\mathcal U}(\vf) [\psi], 
%{\mathcal U}(\vf) [\psi] \big]  
%\\
%&  = {\mathcal B}_\bot(\vf)^{- 1} 
%{\mathcal N}\big[{\mathcal U}(\vf)[\psi], 
%{\mathcal B}_\bot(\vf)[\psi] \big] 
%+ {\mathcal R}_{\mathcal Q}(\vf, \psi)\,, 
%\end{aligned}
%\end{equation}
where
\begin{align}
{\mathcal R}_{\mathcal Q}(\vf, \psi) & := 
{\mathcal B}_\bot(\vf)^{- 1} {\mathcal N}\big[{\mathcal U}(\vf)[\psi], 
{\mathcal B}_\bot(\vf) \circ {\mathcal F}_+(\vf)[\psi] \big]  \label{resto espansione Q}
\\
& \qquad + {\mathcal F}_-(\vf) \circ {\mathcal B}_\bot(\vf)^{- 1}
 {\mathcal N}\big[{\mathcal U}(\vf) [\psi], {\mathcal U}(\vf) [\psi] \big] \,. 
\end{align}
We analyze each term in \eqref{prima espansione mathcal Q} separately.

\noindent
{\bf Estimate of the term ${\mathcal R}_{\mathcal Q}$.} By the estimate \eqref{WPMFpm} and Lemma \ref{lemma elementare nonlinearita}-$(ii)$, one has that
\begin{align}
\Big\| {\mathcal B}_\bot(\vf)^{- 1} {\mathcal N}
\big[{\mathcal U}(\vf)[\psi], 
{\mathcal B}_\bot(\vf) \circ {\mathcal F}_+(\vf)&[\psi] 
\big] \Big\|_{H^s} 
\lesssim 
\Big\|   {\mathcal N}\big[{\mathcal U}(\vf)[\psi], 
{\mathcal B}_\bot(\vf) \circ {\mathcal F}_+(\vf)[\psi] \big] \Big\|_{H^s}  
\\
& \lesssim_s \| {\mathcal U}(\vf)[\psi] \|_{H^{s - 1}} \| {\mathcal B}_\bot(\vf) \circ {\mathcal F}_+(\vf)[\psi]  \|_{H^{s + 1}} \\
& \lesssim_s \| \psi \|_{H^{s - 1}} \| {\mathcal F}_+(\vf)[\psi]  \|_{H^{s + 1}}  
%\\& 
\lesssim_s \| \psi \|_{H^{s - 1}} \| \psi \|_{H^s} 
\lesssim_s \| \psi \|_{H^s}^2\,, 
\end{align}
and
\begin{align}
\Big\| {\mathcal F}_-(\vf) \circ 
{\mathcal B}_\bot(\vf)^{- 1} {\mathcal N}
\big[{\mathcal U}(\vf) [\psi], 
{\mathcal U}(\vf) [\psi] \big]\Big\|_{H^s} 
& \lesssim 
\Big\| {\mathcal B}_\bot(\vf)^{- 1} 
{\mathcal N}\big[{\mathcal U}(\vf) [\psi], 
{\mathcal U}(\vf) [\psi] \big] \Big\|_{H^{s - 1}} 
\\& 
\lesssim_s 
\Big\|  {\mathcal N}\big[{\mathcal U}(\vf) [\psi], 
{\mathcal U}(\vf) [\psi] \big] \Big\|_{H^{s - 1}} 
\\& 
\lesssim_s \| {\mathcal U}(\vf) [\psi]\|_{H^s}^2  
%\\& 
\lesssim_s \| \psi \|_{H^s}^2\,.
\end{align}
The latter two estimates imply that $\cR_{\cQ}$ in \eqref{resto espansione Q} satisfies
\begin{equation}\label{stima cal R1 nonlin}
\| {\mathcal R}_{\mathcal Q}(\vf, \psi) \|_{H^s} \lesssim_s \| \psi \|_{H^s}^2, \quad \forall \, \vf \in \T^\nu\,. 
\end{equation}

\medskip

\noindent
{\bf Analysis of the term ${\mathcal B}_\bot(\vf)^{- 1} {\mathcal N}\big[{\mathcal U}(\vf)[\psi], {\mathcal B}_\bot(\vf)[\psi] \big]$ in \eqref{prima espansione mathcal Q}.} Recalling that $\cB_{\perp}^{\pm 1} = \Pi_{0}^\perp  \cB^{\pm 1} \Pi_{0}^\perp $ and that $\cN[w_1,w_2]$ in \eqref{nuovo termine nonlineare} satisfies Lemma \ref{lemma elementare nonlinearita}-$(ii)$,  one has that 
\begin{align}
{\mathcal B}_\bot(\vf)^{- 1} {\mathcal N}\big[{\mathcal U}(\vf)[\psi], {\mathcal B}_\bot(\vf)[\psi] \big] & = \Pi_0^\bot {\mathcal B}(\vf)^{- 1} \Pi_0^\bot {\mathcal N}\big[{\mathcal U}(\vf)[\psi], {\mathcal B}_\bot(\vf)[\psi] \big] \\
& = \Pi_0^\bot {\mathcal B}(\vf)^{- 1} {\mathcal N}\big[{\mathcal U}(\vf)[\psi], {\mathcal B}_\bot(\vf)[\psi] \big] \label{calciopoli0} \\
& = \Pi_0^\bot {\mathcal B}(\vf)^{- 1}\big( - {\mathfrak B}({\mathcal U}(\vf)[\psi]) \cdot \nabla {\mathcal B}(\vf)[\psi]\big)\,. 
\end{align}
Note that, by the definition of the inverse diffeomorphism in \eqref{diffeo.A-1}, one has, for any function $a(x) b(x)$, 
\[
\begin{aligned}
{\mathcal B}(\vf)^{- 1}[a b](y) 
& = (ab)(y + \breve{\bbeta}(\vf, y)) 
= a(y + \breve{\bbeta}(\vf, y)) b(y + \breve{\bbeta}(\vf, y)) 
\\
& = {\mathcal B}(\vf)^{- 1}[a](y){\mathcal B}(\vf)^{- 1}[b](y)\,,
\end{aligned}
\]
and hence, by a direct computation, one gets
\[
 {\mathcal B}(\vf)^{- 1}\nabla {\mathcal B}(\vf)[\psi](x) 
 = M(\vf, x)[\nabla \psi(x)]\,, 
 \qquad 
 M(\vf, x) := {\mathcal B}(\vf)^{- 1}\big[({\rm Id} 
 + \nabla \bbeta(\vf, x))^T\big]\,. 
\]
By \eqref{Sobolev embedding finale} and the estimates on $\bbeta$ and ${\mathcal B}(\vf)^{\pm 1}$ in \eqref{WPMFpm}, one has that 
\begin{equation}\label{stima matrice M nonlin}
\sup_{\vf \in \T^\nu}\| M(\vf, \cdot) \|_{H^s} \lesssim  \big( 1 + \sup_{\vf\in\T^\nu} \| \bbeta(\vf) \|_{H^{s + 1}}\big) \lesssim 1 + \| \bbeta \|_{s + 1 + \frac{\nu + 2}{2}}  \lesssim_s 1 +  \lambda^{- \eta} \lesssim_s 1\,.
\end{equation}
Therefore, the expression \eqref{calciopoli0} becomes
\begin{align}
{\mathcal B}_\bot(\vf)^{- 1} {\mathcal N}\big[{\mathcal U}(\vf)[\psi], {\mathcal B}_\bot(\vf)[\psi] \big] & = \Pi_0^\bot {\mathcal B}(\vf)^{- 1}\big[ - {\mathfrak B}({\mathcal U}(\vf)[\psi]) \big] \cdot \big({\mathcal B}(\vf)^{- 1}\nabla {\mathcal B}(\vf) \big)[\psi] \\
& = \Pi_0^\bot {\mathcal B}(\vf)^{- 1}\big[ - {\mathfrak B}({\mathcal U}(\vf)[\psi]) \big] \cdot M(\vf, x)[\nabla \psi]  \label{calciopoli1} \\
& = \Pi_0^\bot \ta(\vf, \psi) \cdot \nabla \psi \,,
\end{align}
where 
$$
\ta(\vf, \psi) := - M(\vf, x)^T {\mathcal B}(\vf)^{- 1}\big[  {\mathfrak B}({\mathcal U}(\vf)[\psi]) \big]\,.
$$
By the estimates \eqref{WPMFpm}, \eqref{stima matrice M nonlin}, using the algebra of $H^s(\T^2)$ (since $s > 1$) and by Lemma \ref{lemma elementare nonlinearita}-$(i)$, one obtains that, for any $\vf \in \T^\nu$,
\begin{equation}\label{stima a vphi psi nonlin}
\begin{aligned}
\| \ta(\vf, \psi) \|_{H^s} & \lesssim_s \| M(\vf, \cdot) \|_{H^s} \big\| {\mathcal B}(\vf)^{- 1}\big[  {\mathfrak B}({\mathcal U}(\vf)[\psi]) \big] \big\|_{H^s} \\
& \lesssim_s \| {\mathfrak B}({\mathcal U}(\vf)[\psi]) \|_{H^s} \lesssim_s \| {\mathcal U}(\vf)[\psi] \|_{H^{s - 1}} \lesssim_s \| \psi \|_{H^{s - 1}}\,. 
\end{aligned}
\end{equation}
The claimed statement then follows by \eqref{prima espansione mathcal Q}, \eqref{calciopoli1}, together with the estimates \eqref{stima cal R1 nonlin}, \eqref{stima a vphi psi nonlin}. 
\end{proof}
We are now ready to prove Theorem \ref{maximal time}.  
%Let $T_*$ be the maximal time of existence 
%of the solution $\psi$

\begin{proof}[{\bf Proof of Theorem \ref{maximal time}}]
By contradiction, we assume that 
 $0 < T_* \leq \gamma \delta^{- 1}$ 
 for  $0 < \gamma \ll 1$ small enough. 
 We shall perform an energy estimate on the Cauchy problem \eqref{problema di Cauchy trasformato dopo riduzione} that, in view of the latter Proposition \ref{lemma struttura nonlin sistema trasformato}, can be written as 
\begin{equation}\label{problema di Cauchy trasformato dopo riduzione2}
\begin{cases}
\partial_t \psi = {\mathcal D} \psi + \Pi_0^\bot \big( \ta(\lambda \omega t, \psi) \cdot \nabla \psi  \big)  + {\mathcal R}_{\mathcal Q}(\lambda \omega t, \psi) \\
\psi(0, x) = \psi_0(x) \,. 
\end{cases}
\end{equation} 
 To shorten notations, we write 
 \begin{equation}\label{short.not}
     \ta(t) \equiv \ta(\lambda \omega t, \psi(t))\,, \quad {\mathcal R}_{\mathcal Q}(t) \equiv 
 {\mathcal R}_{\mathcal Q}(\lambda \omega t, \psi(t)) \,.
 \end{equation}
By \eqref{problema di Cauchy trasformato dopo riduzione2}, one computes that
\begin{align}
\partial_t \| \psi(t) \|_{H^s}^2 & = 
\langle \Lambda^s \partial_t \psi(t)\,,\, 
\Lambda^s \psi(t) \rangle_{L^2} 
+ \langle \Lambda^s  \psi(t)\,,\, 
\Lambda^s \partial_t \psi(t) \rangle_{L^2} 
\nonumber
\\
& = \langle \Lambda^s {\mathcal D} \psi(t)\,,\,
\Lambda^s \psi(t) \rangle_{L^2} 
+ \langle \Lambda^s  \psi(t)\,,\, 
\Lambda^s {\mathcal D} \psi(t) \rangle_{L^2} 
\label{stima energia 11}
\\
& \quad + \langle \Lambda^s  \ta(t) \cdot \nabla \psi(t)\,,\, \Lambda^s \psi(t) \rangle_{L^2} 
+ \langle \Lambda^s  \psi(t)\,,\, 
\Lambda^s \ta(t) \cdot \nabla  \psi(t) \rangle_{L^2}
\label{stima energia 12}
\\
& \quad + \langle \Lambda^s  {\mathcal R}_{\mathcal Q}(t) \,,\, \Lambda^s \psi(t) \rangle_{L^2} + \langle \Lambda^s  \psi(t)\,,\, 
\Lambda^s {\mathcal R}_{\mathcal Q}(t)\rangle_{L^2}\,.
\label{stima energia 13}
\end{align}
First of all, since ${\mathcal D}$ is a real diagonal operator with purely imaginary eigenvalue by Theorem \ref{reducibility linearized totale}, 
 we have that
 $[{\mathcal D}, \Lambda^s] = 0$ 
 and that ${\mathcal D}$ is skew self-adjoint, i.e. 
 ${\mathcal D}^* = - {\mathcal D}$. Hence, we deduce that  
\begin{equation}\label{stima energia 2}
\begin{aligned}
%& \langle {\mathcal D}\Lambda^s  \psi(t)\,,\, \Lambda^s %\psi(t) \rangle_{L^2} + \langle \Lambda^s  \psi(t)\,,\, 
%{\mathcal D}\Lambda^s  \psi(t) \rangle_{L^2} 
%\\
%& 
\eqref{stima energia 11}= \langle ({\mathcal D} + {\mathcal D}^*)\Lambda^s  \psi(t)\,,\, \Lambda^s \psi(t) \rangle_{L^2}  = 0\,. 
\end{aligned}
\end{equation}
We now study the terms in \eqref{stima energia 12}.
By a direct calculation one has that 
\begin{equation}\label{aggiunto trasporto}
\big(\ta(t) \cdot \nabla \big)^* = - \ta(t) \cdot \nabla 
- {\rm div}(\ta(t))\,.
\end{equation}
Moreover, by using the Cauchy-Schwartz inequality,
one gets
\begin{align}
%& \langle \Lambda^s  a(t) \cdot \nabla \psi(t)\,,\, %\Lambda^s \psi(t) \rangle_{L^2} 
%+ \langle \Lambda^s  \psi(t)\,,\, 
%\Lambda^s a(t) \cdot \nabla  \psi(t) \rangle_{L^2} 
%\\& 
\eqref{stima energia 12}&= 
\langle   \ta(t) \cdot \nabla \Lambda^s \psi(t)\,,\, \Lambda^s \psi(t) \rangle_{L^2} 
+ \langle \Lambda^s  \psi(t)\,,\,  
\ta(t) \cdot \nabla  \Lambda^s \psi(t) \rangle_{L^2} 
\\
& \qquad 
+ \langle [\Lambda^s\,,\,  \ta(t) \cdot \nabla] \psi(t)\,,\, \Lambda^s \psi(t) \rangle_{L^2} 
+ \langle \Lambda^s  \psi(t)\,,\, 
[\Lambda^s\,,\, \ta(t) \cdot \nabla]  
\psi(t) \rangle_{L^2}
\\&
 \leq 
\big\langle   \big[\ta(t) \cdot \nabla 
+ (\ta(t) \cdot \nabla)^* \big]\Lambda^s \psi(t)\,,\, \Lambda^s \psi(t) \big\rangle_{L^2} 
+ 2\| [\Lambda^s\,,\,  
\ta(t) \cdot \nabla] \psi(t)\|_{L^2}
\| \Lambda^s \psi(t) \|_{L^2} 
\\&
\stackrel{\eqref{aggiunto trasporto}}{\leq} \| {\rm div}(\ta(t)) \Lambda^s \psi(t) \|_{L^2} 
\| \Lambda^s \psi(t) \|_{L^2}
+ 2\| [\Lambda^s\,,\,   
\ta(t) \cdot \nabla] \psi(t)\|_{L^2}
\| \Lambda^s \psi(t) \|_{L^2} \,.
\end{align}
Therefore, using the above estimate and  Lemma \ref{Kato commutator estimate}, 
one deduces
\begin{align}
\eqref{stima energia 12}&\lesssim_{s}
%\| {\rm div}(\ta(t)) \Lambda^s \psi(t) \|_{L^2} 
%\| \Lambda^s \psi(t) \|_{L^2}
%+ \| \ta(t) \|_{H^s}\| \psi(t) \|_{H^s}^2 
%\\
%& \lesssim_s 
\| {\rm div}(\ta(t)) \|_\infty 
\| \Lambda^s \psi(t) \|_{L^2}^2
%\| \Lambda^s \psi(t) \|_{L^2} 
+ \| \ta(t) \|_{H^s}\| \psi(t) \|_{H^s}^2 
\\
& \lesssim_s \| \ta(t) \|_{\mathcal C^1} 
\| \psi(t) \|_{H^s}^2 + \| \ta(t) \|_{H^s}
\| \psi(t) \|_{H^s}^2 
 \lesssim_s \| \ta(t) \|_{H^s}\| \psi(t) \|_{H^s}^2\,,
\end{align}
where the latter inequality holds since $H^s(\T^2)$ is compactly embedded in ${\mathcal C}^1(\T^2)$ for $s>2$, with 
$\| f \|_{{\mathcal C}^1} \lesssim_s 
\| f \|_{H^s}$ for any $f \in H^s(\T^2)$. 
By \eqref{short.not} and \eqref{stime a cal R Q} in Proposition \ref{lemma struttura nonlin sistema trasformato}, we have
$\| \ta(t) \|_{H^s} \lesssim_s \| \psi(t) \|_{H^{s - 1}} \lesssim_s \| \psi(t) \|_{H^s}$, 
and we conclude the cubic estimate 
\begin{equation}\label{stima energia 3}
%\langle \Lambda^s  a(t) \cdot \nabla \psi(t)\,,\, 
%\Lambda^s \psi(t) \rangle_{L^2} 
%+ \langle \Lambda^s  \psi(t)\,,\, 
%\Lambda^s a(t) \cdot \nabla  \psi(t) \rangle_{L^2}
\eqref{stima energia 12}
\lesssim_s \| \ta(t) \|_{H^s}\| \psi(t) \|_{H^s}^2 
\lesssim_s \| \psi(t) \|_{H^s}^3\,. 
\end{equation}
We finally estimate the terms in \eqref{stima energia 13}.
%last two terms in \eqref{stima energia 1}. 
By \eqref{stime a cal R Q} in Proposition \ref{lemma struttura nonlin sistema trasformato}, one has 
\begin{equation}\label{stima di energia 4}
\begin{aligned}
%& \langle \Lambda^s  {\mathcal R}_{\mathcal Q}(t) 
%\,,\, \Lambda^s \psi(t) \rangle_{L^2} 
%+ \langle \Lambda^s  \psi(t)\,,\, 
%\Lambda^s {\mathcal R}_{\mathcal Q}(t)\rangle_{L^2}  
%\\& 
\eqref{stima energia 13}\leq 
2 \| \Lambda^s {\mathcal R}_{\cQ}(t) \|_{L^2} 
\| \Lambda^s \psi(t) \|_{L^2} \leq 
2 \| {\mathcal R}_{\cQ}(t) \|_{H^s} 
\|  \psi(t) \|_{H^s}  
\lesssim_s
\| \psi(t) \|_{H^s}^3\,.
\end{aligned}
\end{equation}
Therefore by 
\eqref{stima energia 11}-\eqref{stima energia 13}
and 
\eqref{stima energia 3}, \eqref{stima di energia 4}, 
one gets 
\begin{equation}\label{stima energia 5}
\partial_t \| \psi(t) \|_{H^s} 
\leq {\mathfrak C}(s) \| \psi(t) \|_{H^s}^2 \,, 
\quad \forall \, t \in [0, T_*)
\end{equation}
for some constant ${\mathfrak C}(s) > 0$. 
By the comparison principle for ODEs, one has that 
\[
\begin{aligned}
& \| \psi(t) \|_{H^s} \leq z(t), \quad \forall \, t \in [0, T_*) \quad \text{where} \\
& \dot z(t) = {\mathfrak C}(s) z(t)^2, \quad z(0) = \| \psi_0 \|_{H^s}\,.
\end{aligned}
\]
Hence, one has that 
\begin{equation}\label{stima compara psi}
    \| \psi(t) \|_{H^s} \leq \dfrac{\| \psi_0 \|_{H^s}}{1 - {\mathfrak C}(s) 
\| \psi_0 \|_{H^s} t }, \quad \forall \, t \in [0, T_*)\,.
\end{equation}
By the estimates \eqref{WPMFpm}, recalling \eqref{def psi0} and \eqref{stimaw0}, one has that 
\begin{equation}\label{stima.psi0}
    \| \psi_0 \|_{H^s} = \| {\mathcal U}(0)^{- 1}[w_0] \|_{H^s} 
\leq 
\big(1 + C(s) \lambda^{- \eta} \big) \| w_0 \|_{H^s} 
\leq 
\big(1 + C(s) \lambda^{- \eta} \big) \delta \leq \frac{8}{7} \delta\,,
\end{equation}
for $\lambda \equiv \lambda(s) \gg 0$ large enough. 
Then, for any $0 < t < T_* \leq \gamma \delta^{- 1}$, one has that 
\begin{equation}\label{deno.stima}
    1 - {\mathfrak C}(s) \| \psi_0 \|_{H^s} t \geq 1 - \frac{8}{7}{\mathfrak C}(s) 
\delta \gamma \delta^{- 1} \geq 1 - O (\gamma) \,,
\end{equation}
for $\gamma \equiv \gamma(s) \ll 1$ small enough. Thus, by \eqref{stima compara psi}, \eqref{stima.psi0} and \eqref{deno.stima}, we obtain, for any $t \in [0, T_*)$,
\begin{equation}\label{maximal time psi}
\| \psi(t) \|_{H^s} \leq \frac{8}{7} \delta \frac{1}{1 - O(\gamma)} < \frac{5}{4} \delta \,,
\end{equation}
with $0 < \gamma \ll 1$ small enough.
Recalling \eqref{v v lambda w}, 
\eqref{definizione w (t)} and using the estimates
\eqref{WPMFpm} and \eqref{maximal time psi}, 
one obtains that, for any $t \in [0, T_\delta]$, with $T_{\delta}:=\gamma(s)\delta^{-1}$,
\begin{equation}\label{stbility estimate nella proof}
\begin{aligned}
\| v(t) - v_\lambda(\lambda \omega t, \cdot) \|_{H^s} & = \| w(t) \|_{H^s} = 
\| {\mathcal U}(\lambda \omega t)[\psi(t)] \|_{H^s} 
\leq 
\big(1 + C(s) \lambda^{- \eta} \big) \| \psi(t) \|_{H^s} 
\\
& \leq \big(1 + C(s) \lambda^{- \eta} \big) \frac{5}{4}\delta 
<2 \delta
\end{aligned}
\end{equation}
by taking $\lambda\equiv \lambda(s) \gg 0$ large enough. 
This contradicts the fact that the supremum 
$T_*$  in \eqref{supremumT} is less or equal 
to $T_{\delta}=\gamma(s)\delta^{-1}$. 
Theorem \ref{maximal time} then follows. 
\end{proof}

\bigskip

\noindent
{\bf Proof of Theorems \ref{thm:mainsatability}, \ref{corollario nonlinear stability} concluded.}
Theorem \ref{thm:mainsatability}  is a direct consequence of Theorem 
\ref{maximal time}. Therefore, in the following, we prove Theorem 
\ref{corollario nonlinear stability}.

From the estimate \eqref{ansatz}, using \eqref{Sobolev embedding finale},  with $S = s + \bar \mu$, with $\bar \mu \geq \frac{\nu + 1}{2}> 0$, one has that
\[
\begin{aligned}
\sup_{t\in [0,T_{\delta}]}\| v_\lambda(\lambda \omega t, \cdot) \|_{H^s} & \leq 
\sup_{\vf \in \T^\nu} \| v_\lambda(\vf, \cdot) \|_{H^s} 
\lesssim 
\| v_\lambda \|_{s + \frac{\nu + 1}{2}} 
\lesssim 
\| v_\lambda \|_{s + \bar \mu} 
\stackrel{\eqref{ansatz}}{\lesssim_s} 
\lambda^{\theta}\,.
\end{aligned}
\]
Hence, by the latter estimate, together with \eqref{v v lambda w} and \eqref{stima long time thm} in Theorem \ref{maximal time}, one has that
\[
\begin{aligned}
\sup_{t\in [0,T_{\delta}]}\| v(t) \|_{H^s} & \leq \sup_{t\in [0,T_{\delta}]}\| v_\lambda(\lambda \omega t, \cdot ) \|_{H^s} 
+ \sup_{t\in [0,T_{\delta}]}\| v(t) - v_\lambda(\lambda \omega t, \cdot ) \|_{H^s} 
%\\& 
\lesssim  \lambda^{\theta}
+  \delta \lesssim \lambda^{\theta} \,,
\end{aligned}
\]
which is exactly the upper bound in \eqref{stimaaltobasso}, since $\theta = \alpha - 1 + \mathtt c$ (see \eqref{theta.def.ridu}). We are left to prove the lower bound in 
\eqref{stimaaltobasso}. 
First of all, recalling Section $3$ in \cite{BFMT1},
we note that  $ v_\lambda(\lambda \omega t,x)$
can be expanded as
\begin{equation}\label{vlambdaglambda}
 v_\lambda(\vphi, x)=
 g_{\lambda}(\vphi,x)+z_{\lambda}(\vphi,x)
\end{equation}
where 
the function $g_{\lambda}$ is the solution of the 
equation (see formula (3.16) in \cite{BFMT1})
\begin{equation}\label{eq:glambda}
(\lambda\omega\cdot\pa_{\vphi}- \beta \mathtt{L})
g_{\lambda}(\vphi,x)=\lambda^{\alpha}f(\vphi,x)
\end{equation}
where $f(\vphi,x)$ is the forcing term
appearing in \eqref{beta.waves.large.eq}, $\mathtt L = \partial_{x_1} (- \Delta)^{- 1}$ (see \eqref{def coefficienti operatori linearized}) and 
$z_{\lambda}$ satisfies (see the bounds in 
Propositions 3.3, 7.3 and Section 8 in \cite{BFMT1})
\begin{equation}\label{stimaZetalambda}
\|z_{\lambda}\|_{s + \bar \mu}\lesssim \lambda^{\alpha - 1 -\zeta} \quad \text{where} \quad \zeta :=    2 - \alpha  - 3 \mathtt c \,. 
\end{equation}
Note that $\zeta > 0$, by \eqref{choice mathtt c}. Then, using \eqref{Sobolev embedding finale} and having $\bar\mu \geq \frac{\nu+1}{2}$, we have
\begin{equation}\label{stimaZetalambdaa}
\begin{aligned}
\sup_{t \in \R}\| z_\lambda(\lambda \omega t, \cdot) \|_{H^s} & \leq \sup_{\vf \in \T^\nu} \| z_\lambda(\vf, \cdot) \|_{H^s} \lesssim \| z_\lambda \|_{s + \frac{\nu + 1}{2}}  \lesssim \| z_\lambda \|_{s + \bar \mu} 
\stackrel{\eqref{stimaZetalambda}}{\lesssim_s} 
\lambda^{\alpha - 1 -\zeta}\,.
\end{aligned}
\end{equation}
%where $\zeta:=2 (1-\mathtt{c}-\alpha)$. 
On the other 
hand, the function $g_\lambda$ satisfies 
the following properties.
Using the fact that $f$ has zero average in $x\in\T^2$, 
one deduces  that the solution of the equation 
\eqref{eq:glambda} is given by
\begin{align}
g_{\lambda}(\vphi,x)&:= 
 \lambda^{\alpha} (\lambda  \omega \cdot \pa_{\vf} 
 - \beta\,\tL )^{-1} f(\vf,x) \label{def:glambda}
 \\
	&   :=   \sum_{\substack{(\ell,j)\in \Z^{\nu+2}, j\neq 0 \\ \pi^\top(\ell)+j =0}} 
	\frac{\lambda^\alpha}{\im \big( \lambda\omega\cdot\ell - \beta \,\tL (j)\big)} 
    \wh{f}(\ell,j) e^{\im(\ell\cdot\vf+j\cdot x)}\,,
\end{align}
for $\omega\in \mathcal{O}_{\lambda}
\subset \Omega_{\lambda}$.
Recalling formul\ae\, (4.2), (3.6) in \cite{BFMT1}, one has the lower bounds 
\[
|\lambda\omega\cdot\ell - \beta\mathtt{L}(j)|
\geq \lambda\frac{\gamma}{\langle \ell\rangle^\tau}\,,
\qquad\forall\, (\ell,j)\in \Z^{\nu+2}\setminus\{0\}\,,
\;\; \pi^{\top}(\ell)+j=0\,,
\]
so that the function $g_{\lambda}$ is well-defined. We recall that the metric on the torus $\T^\nu$ is defined by 
\begin{equation}\label{def metric T nu}
{\bf d}_{\T^\nu}(\vf_1, \vf_2) := {\rm min}_{k \in \Z^\nu} |\vf_1 - \vf_2 + 2 \pi k|\,. 
\end{equation}
The following lemma holds.
\begin{lem}\label{lower bound g lambda vari}
Let $s > 1$ and $\omega \in {\mathcal O}_\lambda$. Then, for any $\lambda
\gg1$ large enough, there exists $\delta\equiv\delta(\lambda)\ll 1$ sufficiently small such that
\begin{equation}\label{stima g lambda L infty t Hs x}
\sup_{t \in [0, T_\delta]} \| g_\lambda(\lambda \omega t, \cdot) \|_{H^s} \gtrsim_s \sup_{\begin{subarray}{c}
t \in [0, T_\delta] \\
x \in \T^2
\end{subarray}} |g_\lambda(\lambda \omega t, x)| \gtrsim_s \lambda^{\alpha - 1}. 
\end{equation}
\end{lem}
\begin{proof}
We first note that 
\[
\| g_\lambda \|_0 \equiv \| g_\lambda \|_{L^2_{\vf, x}} 
\lesssim 
\sup_{(\vf, x) \in \T^{\nu + 2}} |g_\lambda(\vf, x)| = |g_\lambda(\vf_0, x_0)|\,,
\]
for some point $(\vf_0, x_0) \in [- \pi, \pi]^{\nu + 2}$, since $g_{\lambda}$ is a smooth function on a compact domain. Moreover, 
by \eqref{def:glambda} and 
the upper bound
$
|\lambda \omega \cdot \ell + \beta \mathtt L(j)| 
\lesssim 
\lambda \langle \ell \rangle
$,
one gets that 
\begin{align}
\| g_\lambda \|_{L^2_{\vf, x}} & = \lambda^\alpha 
\Big( \sum_{\substack{
(\ell,j)\in \Z^{\nu+2}, j\neq 0 \\ \pi^\top(\ell)+j =0}} 
\dfrac{|\widehat f(\ell, j)|^2}{|\lambda \omega \cdot \ell + \beta \mathtt L(j)|^2} \Big)^{\frac12} \label{pappardella 0}
\\
& \gtrsim \lambda^{\alpha - 1} \Big( \sum_{\substack{
(\ell,j)\in \Z^{\nu+2}, j\neq 0 \\ \pi^\top(\ell)+j =0}} 
\dfrac{|\widehat f(\ell, j)|^2}{\langle \ell \rangle^2} \Big)^{\frac12} 
 \gtrsim \lambda^{\alpha - 1}\,,
\end{align}
since $f$ is not identically zero and hence $\sum_{
 j\neq 0, \,  \pi^\top(\ell)+j =0} 
\langle \ell \rangle^{-2}|\widehat f(\ell, j)|^2 > 0$. Thus \eqref{pappardella 0} implies that 
\begin{equation}\label{lower bound embedding g lambda}
\sup_{(\vf, x) \in \T^{\nu + 2}} |g_\lambda(\vf, x)| =
 |g_\lambda(\vf_0, x_0)| \gtrsim \lambda^{\alpha - 1}\,. 
\end{equation}
Using the continuity of $g_\lambda$, 
one has that  there exists $a > 0$ such that 
\[
{\bf d}_{\T^\nu}(\vf\,,\, \vf_0) \leq a 
\qquad \Longrightarrow 
\qquad 
|g_\lambda(\vf, x_0) - g_\lambda(\vf_0, x_0)| \leq 1\,.
\]
Moreover, using that the orbit 
\[
\big\{ \lambda \omega t : t \geq 0 \big\}
\]
is dense on $\T^\nu$ (because $\omega$ is diophantine), 
one obtains that there exists a time $t_a \geq 0 $ such that
\begin{equation}\label{tempo ta defa}
{\bf d}_{\T^\nu}(\lambda \omega t_a\,,\, \vf_0) \leq a\,,
\end{equation}
which implies (recalling \eqref{def metric T nu}) the existence of $k_a \in \Z^\nu$ such that 
\begin{equation}\label{tempo ta def}
|\lambda \omega t_a + 2 \pi k_a - \vf_0| \leq a\,.
\end{equation}
Therefore, we obtained that there exists a time $t_a \geq 0$ such that 
\[
|g_\lambda (\lambda \omega t_a, x_0) - g_\lambda (\vf_0, x_0)| \leq 1\,,
\]
and hence 
\begin{equation}\label{cernusco 0}
|g_\lambda (\lambda \omega t_a, x_0)| \geq |g_\lambda (\vf_0, x_0)| - 1 \stackrel{\eqref{lower bound embedding g lambda}}{\geq} C \lambda^{\alpha - 1} - 1 \gtrsim  \lambda^{\alpha - 1}
\end{equation}
for $\lambda \gg 0$ large enough. 
By \eqref{tempo ta def}, one has that the time $t_a$ satisfies the bound
\[
t_a \leq \frac{a + |\vf_0| + 2\pi |k_a|}{\lambda |\omega|} \lesssim  \frac{C(a, \omega)}{\lambda}, \quad \text{for some constant} \quad C(a, \omega) > 0\,.
\]
To ensure that $t_a \in [0, T_\delta]$, we take $\delta \equiv \delta(\lambda) \ll 1$ so small in such a way that 
$T_\delta = \dfrac{\gamma}{\delta}  > \dfrac{C(a, \omega)}{\lambda}$. Therefore, one has that 
\[
\begin{aligned}
\sup_{\begin{subarray}{c}
t \in [0, T_\delta] \\
x \in \T^2
\end{subarray}} |g_\lambda(\lambda \omega t, x)|  & \geq |g_\lambda (\lambda \omega t_a, x_0) | \stackrel{\eqref{cernusco 0}}{\gtrsim} \lambda^{\alpha - 1}\,. 
\end{aligned}
\]
The bound \eqref{stima g lambda L infty t Hs x}
then follows since $H^s(\T^2)$, $s > 1$, is compactly embedded in ${\mathcal C}^0(\T^2)$, with
\[
\sup_{\begin{subarray}{c}
t \in [0, T_\delta] \\
x \in \T^2
\end{subarray}} |g_\lambda(\lambda \omega t, x)| \lesssim_s \sup_{t \in [0, T_\delta]} \| g_\lambda(\lambda \omega t, \cdot) \|_{H^s}\,,
\]
and the proof is concluded.
\end{proof}

\noindent
We are now in position to get the lower bound in 
\eqref{stimaaltobasso}. Recalling 
\eqref{vlambdaglambda},
one obtains 
\begin{align}
\sup_{t \in [0, T_\delta]} \| v_\lambda(\lambda \omega t, \cdot) \|_{H^s} 
& \geq 
\sup_{t \in [0, T_\delta]} \| g_\lambda(\lambda \omega t, \cdot) \|_{H^s} 
- \sup_{t \in [0, T_\delta]} \| z_\lambda(\lambda \omega t, \cdot) \|_{H^s} 
\\& 
\geq 
\sup_{t \in [0, T_\delta]} \| g_\lambda(\lambda \omega t, \cdot) \|_{H^s} 
- \sup_{t \in \R} \| z_\lambda(\lambda \omega t, \cdot) \|_{H^s} 
\\ & 
\stackrel{\eqref{stima g lambda L infty t Hs x}, \eqref{stimaZetalambdaa}}{\gtrsim_s} 
\lambda^{\alpha - 1}(1 - \lambda^{- \zeta}) 
\gtrsim_s \lambda^{\alpha - 1}\,. 
\end{align}
for $\lambda \gg 0$ large enough. 
The latter estimate, together with the estimate 
\eqref{stbility estimate nella proof}, implies that 
\begin{align}
\sup_{t \in [0, T_\delta]} \| v(t) \|_{H^s} & \geq 
\sup_{t \in [0, T_\delta]} \| v_\lambda(\lambda \omega t, \cdot) \|_{H^s} 
- \sup_{t \in [0, T_\delta]} \| v(t) - v_\lambda(\lambda \omega t, \cdot) \|_{H^s} 
\\& 
\gtrsim_s \lambda^{\alpha - 1} - \delta 
\gtrsim_s \lambda^{\alpha - 1}\,,
\end{align}
by taking $\delta \equiv \delta(\lambda) \ll 1$ small enough. 
Then the thesis of Theorem \ref{corollario nonlinear stability} follows.

\bigskip

{\bf Statements and declarations}

\medskip

\noindent
The authors certify that they have no affiliations or involvement with any organization or entity with any
financial interest, or non-financial interest in the subject matter or materials discussed in this manuscript.
Moreover, data sharing is not applicable to this article as no datasets were generated or analyzed during the
current study.

\addcontentsline{toc}{section}{References}
\begin{footnotesize}
	
\end{footnotesize}

\bigskip

\begin{flushleft}
	\textbf{Roberto Feola}
	
    Dipartimento di Matematica e Fisica

    Universit\`a degli Studi RomaTre

    Largo San Leonardo Murialdo 1
    
	00146 Roma,	Italy
	\smallskip 
	
	\texttt{roberta.feola@uniroma3.it}
	
\end{flushleft}

\begin{flushleft}
	\textbf{Luca Franzoi}
	
	\smallskip
	
	Dipartimento di Matematica
	
	Universit\`a degli Studi di Roma Tor Vergata
	
	Via della Ricerca Scientifica 1
	
	00133 Roma, Italy
	
	\smallskip 
	
	\texttt{franzoi@mat.uniroma2.it}
	
\end{flushleft}
\begin{flushleft}
	\textbf{Riccardo Montalto}
	
	\smallskip
	
	Dipartimento di Matematica ``Federigo Enriques''
	
	Universit\`a degli Studi di Milano
	
	Via Cesare Saldini 50
	
	20133 Milano, Italy
	
	\smallskip 
	
	\texttt{riccardo.montalto@unimi.it}
	
\end{flushleft}

\end{document}